\documentclass[12pt]{amsart}
\usepackage[margin=1in,rmargin=1in]{geometry}
\usepackage{amsmath,latexsym,amssymb,amsthm,graphicx}
\usepackage{mathtools}
\usepackage{subfig}
\usepackage{xfrac}
\usepackage{fix-cm}
\usepackage{color,transparent}
\usepackage{verbatim}

\usepackage[bookmarks=true,%
    colorlinks=true,%
    linkcolor=blue,%
    citecolor=blue,%
    filecolor=blue,%
    menucolor=blue,%
    urlcolor=blue]{hyperref}
\usepackage{doi}
\usepackage{stmaryrd}
\usepackage{multirow}
\usepackage{float}
\usepackage{array}
\usepackage[shortlabels,inline]{enumitem}
\usepackage{pdfsync}
\usepackage{tikz}
\usepackage{tikz-cd}
\usepackage{float}
\usepackage[T1]{fontenc}
\usepackage[utf8]{inputenc}
\usepackage{hyphenat}
\usepackage{booktabs}

\newcommand*{\eprint}[1]{\url{#1}}
\newcommand{\ora}[1]{\overrightarrow{#1}}

\graphicspath{{mpdraws/}{draws/}}
\def\mathcenter#1{\vcenter{\hbox{$#1$}}}
\def\mfig#1{\mathcenter{\includegraphics{#1}}}

\setenumerate[1]{label=(\arabic*)}

\numberwithin{equation}{section}
\numberwithin{figure}{section}

\newtheorem{mainthm}{Theorem}

\newtheorem{theorem}[equation]{Theorem}
\newtheorem{lemma}[equation]{Lemma}

\newtheorem{corollary}[equation]{Corollary}
\newtheorem{proposition}[equation]{Proposition}

\newtheorem*{ex*}{Exercise}

\theoremstyle{definition}
\newtheorem{remark}[equation]{Remark}

\newtheorem{example}[equation]{Example}
\newtheorem{definition}[equation]{Definition}
\newtheorem{question}[equation]{Question}
\newtheorem{conjecture}[equation]{Conjecture}

\newtheorem{warning}[equation]{Warning}
\newtheorem{convention}[equation]{Convention}

\theoremstyle{plain}

\newcounter{parentThm}
\newenvironment{subtheorems}{%
  \refstepcounter{equation}%
  \setcounter{parentThm}{\value{equation}}%
  \setcounter{equation}{0}
}{%
  \setcounter{equation}{\value{parentThm}}%
}

\newcommand\Q{\ensuremath{\mathbb{Q}}}
\newcommand\R{\ensuremath{\mathbb{R}}}

\def\H{\ensuremath{\mathbb{H}}} % Redefining "long Hungarian umlaut"

\DeclareMathOperator{\Teich}{Teich}

\DeclareMathOperator{\ang}{ang}

\DeclareMathOperator{\tr}{tr}
\DeclareMathOperator{\diam}{diam}
\DeclareMathOperator{\EL}{EL}

\DeclareMathOperator{\Area}{Area}

\DeclareMathOperator{\id}{id}
\DeclareMathOperator{\gd}{gd}

\newcommand{\OneHalf}{{\textstyle \frac{1}{2}}}
\newcommand*{\wt}[1]{\widetilde{#1}}

\renewcommand{\epsilon}{\varepsilon}

\DeclareMathOperator{\SO}{\mathit{SO}}

\DeclareMathOperator{\PSL}{\mathit{PSL}}
\DeclareMathOperator{\SU}{\mathit{SU}}

\DeclarePairedDelimiter\abs{\lvert}{\rvert}
\DeclarePairedDelimiter\norm{\lVert}{\rVert}

\newcommand{\Simples}{\mathcal{S}}
\newcommand{\Curves}{\mathcal{C}}
\newcommand{\Meas}{\mathcal{M}}

\newcommand{\GC}{\mathcal{GC}}
\newcommand{\ML}{\mathcal{ML}}

\newcommand{\reducesto}{\mathrel{\searrow}}
\DeclareMathOperator{\MCG}{\mathit{MCG}}
\DeclareMathOperator{\supp}{supp}
\usetikzlibrary{arrows}
\tikzset{
    labl/.style={anchor=south, rotate=90, inner sep=.5mm}
}
\tikzstyle{every picture}=[> = to]
\tikzset{cdlabel/.style={execute at begin node=$\scriptstyle,execute at end node=$}}
\tikzset{implication/.style={double equal sign distance, -implies}}
\tikzset{biimplication/.style={double equal sign distance, implies-implies}}

\begin{document}
\title{From curves to currents}
\author[Martínez-Granado]{Dídac Martínez-Granado}
\address{Department of Mathematics\\University of Luxembourg\\Av.\ de la Fonte 6, Esch-sur-Alzette, L-4364, Luxembourg}
\email{didac.martinezgranado@uni.lu}
\author[Thurston]{Dylan~P.~Thurston}
\address{Department of Mathematics\\
         Indiana University,
         Bloomington, Indiana 47405\\
         USA}
\email{dpthurst@iu.edu}
%\date{\today}
\begin{abstract}
  Many natural real-valued functions of closed curves are known to extend
  continuously to the larger space of geodesic currents. For instance,
  the extension of length with respect to a fixed hyperbolic
  metric was a motivating example for the development of geodesic
  currents.  We give a simple criterion on a curve function that
  guarantees a continuous extension to geodesic currents. The main condition of
  our criterion is the smoothing property, which has played a role in
  the study of systoles of translation lengths for Anosov
  representations. It is easy to see that our criterion is satisfied
  for almost all the known examples of
  continuous functions on geodesic currents, such as non-positively curved lengths or stable lengths for surface groups, while also applying to new examples like extremal length. We use this extension to obtain a new curve counting result for extremal length.
\end{abstract}

%\subjclass[2020]{Primary xxx; Secondary xxx}
%\keywords{}
\maketitle

\tableofcontents

\section{Introduction}
Geodesic currents on surfaces are measures that realize a suitable
closure of the
space of weighted (multi-)curves on a surface. They were first introduced by Bonahon in
his seminal paper \cite{Bonahon86:EndsHyperbolicManifolds}. Many metric
structures can be embedded in the space of currents, such as
hyperbolic metrics \cite[Theorem~12]{Bonahon88:GeodesicCurrent} or
half-translation structures
\cite[Theorem~4]{DLR10:DegenerationFlatMetrics}. Thus, geodesic
currents allow one to treat curves and metric structures on surfaces
as the same type of object. Via this unifying framework, counting curves of a
given topological type and counting lattice points in the space of
deformations of geometric structures become the same problem
\cite[Main~Theorem]{RS19:CountingProblems}. Geodesic currents also
play a key step in the proof of rigidity of the marked length spectrum
for metrics, via an argument by Otal
\cite[Th\'eor\`eme~2]{Otal90:SpectreMarqueNegative}. Finally, they
provide a boundary of the Teichmüller space, in both the compact \cite[Proposition~17]{Bonahon88:GeodesicCurrent} and non-compact \cite[Theorem~2]{BS21:NoncompactBoundary} cases.
   
In this paper we consider the problem of extending continuously a function defined on the space of weighted multi-curves to its closure, the space of geodesic currents.
    
Previous work of Bonahon extended the notion of geometric intersection
number as a continuous function of two geodesic currents
\cite[Proposition~4.5]{Bonahon86:EndsHyperbolicManifolds}. This
allowed him to extend hyperbolic length to geodesic currents by
following the principle of realizing it as an intersection number with
a distinguished geodesic current
\cite[Proposition~14]{Bonahon88:GeodesicCurrent}.

The same principle using intersection numbers has been used by many
authors to extend length for many other metrics:
Otal for negatively curved Riemannian metrics
\cite[Proposition~3]{Otal90:SpectreMarqueNegative},
Croke-Fathi-Feldman for non-positively curved Riemannian metrics
\cite[Theorem~A]{CFF92:RigidityNonPosCurvedRiem},
Hersonsky-Paulin for negatively curved metrics with conical
singularities \cite[Theorem~A]{HP97:RigidityNegCurvedCone},
Bankovic-Leininger~for non-positively curved Euclidean cone metrics
\cite{BL17:FlatRigidity},
Duchin-Leininger-Rafi more explicitly for singular Euclidean structures
associated to quadratic differentials
\cite[Lemma~9]{DLR10:DegenerationFlatMetrics}, and
Erlandsson for word-length with respect to simple generating sets of
the fundamental group \cite[Theorem~1.2]{Erlandsson:WordLength}.

Another line of results on extending functions to geodesic currents
was also started by Bonahon, who
showed how to extend stable lengths to geodesic currents, not just for
surface groups but for general hyperbolic groups
\cite[Proposition~10]{Bon91:NegativelyCurvedGroups}.
This result was recently
improved by Erlandsson-Parlier-Souto \cite[Theorem~1.5]{EPS:CountingCurves}, who used the return map of the geodesic
flow to remove technical assumptions. These constructions
apply, for instance, to arbitrary Riemannian metrics and the stable
version of word-lengths for arbitrary generating sets.

The problem of extending functions to geodesic currents is interesting
in itself, since, by a result of Rafi and Souto reviewed in
Section~\ref{sec:counting}, it provides a way to
compute asymptotics of the number of curves of a fixed type with a bounded
``length'', for a notion of ``length'' that extends to currents
\cite{RS19:CountingProblems}. Their result builds on work by Mirzakhani \cite[Theorem~7.1]{Mir16:Counting} and Erlandsson-Souto \cite[Proposition~4.1]{ES16:Counting}.
Recently, Erlandsson and Souto have given a new argument to compute
these asymptotics \cite[Theorem~8.1]{ES20:GeodesicCount}.

Our main theorem gives a
simple criterion on functions defined on multi-curves that
guarantees they extend to geodesic currents.
Our result subsumes most of the previous extension results mentioned
above, and provides new extensions for other notions of ``length'', such as
extremal length, thus yielding counting asymptotics for them.
    
Our proof does not use Bonahon's principle on intersection numbers.
Although we drew some
inspiration from the dynamics of Erlandsson-Parlier-Souto
\cite{EPS:CountingCurves}, our techniques are distinct.

\subsection{Main results}
\label{subsec:mainresult}

\begin{table}
    \begin{tabular}{@{}ll@{}}
      \toprule
      Notation & Meaning \\
      \midrule
         $S$ &  topological surface \\
         $\Sigma$ & Riemannian surface \\
         $UT \Sigma$ & unit tangent bundle over $\Sigma$\\
         $\phi_t$ & geodesic flow on $UT \Sigma$ \\
         $\tau$ & cross-section to the geodesic flow \\
         $\psi$ & bump function on a cross-section \\
         $\Simples(S)$ & unoriented simple multi-curves \\
         $\Curves(S)$ & unoriented multi-curves \\
         $\Simples^+(S)$ & oriented simple multi-curves \\
         $\Curves^+(S)$ & oriented multi-curves \\
         $G(S)$ & unoriented geodesics on $\tilde{S}$ \\
         $G^+(S)$ & oriented geodesics on $\tilde{S}$ \\
         $\R\Simples(S)$ & weighted unoriented simple multi-curves \\
         $\R\Curves(S)$ & weighted unoriented multi-curves \\
         $\R\Simples^+(S)$ & weighted oriented simple multi-curves \\
         $\R\Curves^+(S)$ & weighted oriented multi-curves \\ 
         $\GC(S)$ & unoriented geodesic currents \\
         $\GC^+(S)$ & oriented geodesic currents \\
         $\gamma$ & concrete multi-curve on $S$ \\
         $C$ & multi-curve on $S$\\
      \bottomrule\addlinespace
    \end{tabular}
  \caption{Notation for the objects related to surfaces, curves, and
    geodesic currents.}\label{tab:object-notation}
\end{table}

We start by summarizing our main results. Complete definitions of the
terms are deferred to Section~\ref{sec:curves}.

\begin{definition}\label{def:properties}
  For $S$ a compact topological surface without boundary,
  let
  $f \colon \Curves^+(S) \to \R$ be a
  function defined on the space of oriented multi-curves,
  not-necessarily-simple oriented curves; see
  Definition~\ref{def:multi-curve}, and see
  Table~\ref{tab:object-notation} for a summary of notation.
  We will also refer to $f$ as a \emph{curve functional} for short.
  (\emph{Functional} means that it takes values in scalars; it is not
  assumed to be linear.)
  We will also refer to \emph{unoriented} or \emph{weighted} curve
  functionals for real-valued functions defined on the appropriate
  type of multi-curves.
  We define several
  properties that $f$ might satisfy.

  \begin{itemize}
    \item \textbf{Quasi-smoothing:} There is a constant $R\ge 0$ with the
      following property. Let $C$~be an oriented curve on~$S$, and
    let $x$~be an essential crossing of~$C$. Let $C'$ be the oriented
    smoothing of~$C$ at~$x$. Then $f(C) \ge f(C')-R$. Schematically, we
    have
    \begin{equation}\label{eq:quasismoothing}
      f\left(\mfig{curves-3}\right) \ge f\left(\mfig{curves-4}\right) -R
    \end{equation}
    See Definition~\ref{def:essential-crossing} for ``essential
    crossing''. Loosely, it is a crossing that cannot be removed by homotopy. See Definition~\ref{def:smoothing} for ``oriented smoothing''.
  \item \textbf{Smoothing:} We take $R=0$ in the above definition of quasi-smoothing:
    \begin{equation}\label{eq:smoothing}
      f\left(\mfig{curves-3}\right) \ge f\left(\mfig{curves-4}\right)
    \end{equation}
  \item \textbf{Convex union:} Let $C_1$ and~$C_2$ be two oriented curves
    on~$S$. Then
    \begin{equation}\label{eq:union-convex}
      f(C_1 \cup C_2) \le f(C_1) + f(C_2).
    \end{equation}
  \item \textbf{Additive union:} The inequality in convex union
    becomes an equality:
    \begin{equation}\label{eq:union-add}
      f(C_1 \cup C_2) = f(C_1) + f(C_2).
    \end{equation}
  \end{itemize}
\end{definition}

Many natural curve functionals satisfy the additive union property;
for instance, length with
respect to an arbitrary length metric on~$S$ satisfies it by
definition. The square root of extremal length is an example
of a curve functional satisfying convex union, but not additive union
(Section~\ref{sec:extremal-length}).
The name ``convex union'' comes from the fact that, if we extend to
weighted curves and additionally
assume homogeneity (Definition~\ref{def:homogeneous-stable}), then
for a fixed
oriented multi-curve with varying weights, $f$ is a convex as a
function of the weights
(Proposition~\ref{prop:convexweights}).

There are many curve functionals satisfying the smoothing
property, such as hyperbolic length, extremal length, intersection
number with a fixed curve, or length from a length metric
on~$S$. For an example of a natural curve functional that satisfies
quasi-smoothing but not smoothing, we have the
word-length of with respect to an arbitrary generating set of
$\pi_1(S)$
(Example~\ref{ex:puncturedtorus}).
The (quasi-)smoothing property is usually easy to check.

The smoothing property plays an important role in the study of
translation lengths associated to Anosov representations, as we
discuss in Section~\ref{sec:transl-lengths}, following
\cite{MZ19:PositivelyRatioed} and \cite{BIPP19:Currents}. These works
use the smoothing property to reduce the study of length systoles to
the case of simple closed curves. Although these papers point out the
parallelism between the smoothing property of Anosov translation
lengths and that of hyperbolic length
(\cite[Section~4]{BIPP19:Currents} or negatively curved lengths
\cite[Corollary~1.3]{MZ19:PositivelyRatioed}, the results in our paper
reveal the much more universal nature of the smoothing property.
Indeed, we show many other natural notions of lengths which are not
associated to negatively curved structures satisfy the smoothing
condition, such as lengths on Riemannian metrics with no curvature
assumption, lengths coming from more general length space structures,
extremal lengths, or
word-lengths with respect to certain generating sets; see
Section~\ref{sec:examples}.

\begin{definition}
  For $C$ an oriented multi-curve, $nC$ is the oriented multi-curve
  that consists of $n$~parallel
  copies of~$C$ (so with $n$ times as many components or,
  in the context of weighted oriented multi-curves,
  with weights multiplied by~$n$), and
  $C^n$ is the oriented multi-curve with as many components as~$C$, in which each
  component of~$C$ is covered by an $n$-fold cover. That is, if
  $g \in \pi_1(S,x)$ represents $C$, $g^n$ represents~$C^n$.
\end{definition}

\begin{definition}\label{def:homogeneous-stable}
     Let $f$ be a curve
     functional and let $n>0$ be an
     integer. We define some
     properties $f$ might satisfy:
     \begin{itemize}
     \item \textbf{Homogeneity:} For an arbitrary oriented multi-curve~$C$,
       \begin{equation}
         \label{eq:homogeneous}
         f(nC) = nf(C).
       \end{equation}
     \item \textbf{(Weak) stability:} For an arbitrary oriented multi-curve~$C$,
       \begin{equation}
         \label{eq:stable}
         f(C^n) = f(nC)
       \end{equation}
     \item \textbf{Strong stability:} For arbitrary oriented multi-curves~$C,D$,
       \begin{equation}
         \label{eq:strongstable}
         f(D\cup C^n) = f(D\cup nC)
       \end{equation}
     \end{itemize}
\end{definition}

Additive union implies homogeneity, and
if
$f$ satisfies convex union, $f(nC) \le nf(C)$. If $f$ satisfies quasi-smoothing, then $f(nC) -nR \le f(C^n)$, since the self-crossings in $C^n$ are essential crossings by definition (see Definition~\ref{def:essential-crossing}).

We furthermore note that curve functionals are not necessarily positive. 

With this background, we can state our main theorems on extensions of
curve functionals to the space $\GC^+(S)$ of oriented geodesic currents.

\begin{mainthm}\label{thm:convex}
   Let $f$ be a curve functional satisfying the
   quasi-smoothing, convex union, stability, and homogeneity
  properties. Then there is a unique continuous homogeneous function
  $\bar{f} \colon \GC^+(S) \to \R_{\ge 0}$ that extends~$f$.
\end{mainthm}

In the case of unoriented curves, there are two possible smoothings of
an essential crossing, not distinguished from each other.
Then we have the following version of the theorem, deduced from
Theorem~\ref{thm:convex} in
Section~\ref{sec:unor-current}.

\begin{corollary}\label{cor:unorconvex}
   Let $f$ be an unoriented curve functional satisfying
   quasi-smoothing for both possible smoothings of a crossing, in
   addition to the convex union,
   stability, and homogeneity
  properties. Then there is a unique continuous homogeneous function
  $\bar{f} \colon \GC(S) \to \R$ that extends~$f$.
\end{corollary}

Theorem~\ref{thm:convex} should be thought of as an analogue of the
classical theorem that a convex function defined on the rational
points in a finite-dimensional vector space
automatically extends continuously to a convex function defined on the
whole vector space (Proposition~\ref{prop:convexcont}(iv)).
As in the classical case, the functions on geodesic
currents arising from this construction are restricted, as the
next example shows.
(This example is almost the only function we are aware of where our
techniques do not suffice to prove continuity of the extension. See Section~\ref{sec:intersection-number})

\begin{example}\label{ex:selfintersection}
  Consider the curve curves given by $f(C) \coloneqq \sqrt{i(C,C)}$. 
Since intersection number is a continuous two-variable function on
currents \cite[Proposition~4.5]{Bonahon86:EndsHyperbolicManifolds}, it
follows that $f$ extends continuously to geodesic currents.
However, $f$ does not satisfy convex union. For
instance, take
$C_1$ and $C_2$ be two simple closed curves intersecting once.
Then
$f(C_1 \cup C_2)=\sqrt{2}$, but $f(C_1)+ f(C_2) =
0$, contradicting convex union.
On the other hand, $f$ satisfies smoothing, since $i(C,C)$ is
generally twice the self-intersection number. (The slight difference is
for non-primitive curves, where $i(C,C)$ does not count the
essential intersections from the multiple covers. This does not affect
smoothing.)
\end{example}

On the other hand, the stability and homogeneous properties are
necessary conditions for an extension to exist for elementary reasons,
as the multi-curves $nC$ and $C^n$ should represent the same currents
(Example~\ref{ex:puncturedtorus}).
However, a curve functional that satisfies
quasi-smoothing and convex union can be modified to get a curve functional
satisfying all the hypotheses of Theorem~\ref{thm:convex}.

\begin{mainthm}\label{thm:stable}
  Let $f$ be a curve functional satisfying quasi-smoothing and convex union.
  Then the stabilized curve functional 
  \[
    \|f\|(C) \coloneqq \lim_{n \to \infty} \frac{f(C^n)}{n}.
  \]
  satisfies quasi-smoothing, convex union, strong stability, and homogeneity,
  and thus extends to a continuous function on $\GC^+(S)$.
\end{mainthm}

Theorem~\ref{thm:stable} is proved in Section~\ref{sec:stable},
although the implication that weak stability implies strong stability
is used in the proof of Theorem~\ref{thm:convex}.

If the convex union property of $f$ is strengthened to additive union
and the quasi-smoothing property is strengthened to smoothing, then in
fact this extension to geodesic currents comes from intersection with
a fixed current (as in the proofs of extension that used Bonahon's
principle).
This will appear in a forthcoming paper. For this stronger result,
the strict smoothing property is necessary, since intersection number cannot increase after smoothing an essential crossing.

\subsection{Acknowledgments}
\label{sec:acknowledgments}

We thank Francisco Arana, Martin Bridgeman, Maxime Fortier Bourque,
Kasra Rafi, Dalton Sconce,
and Tengren Zhang for helpful conversations and comments, and the
anonymous referee
for careful reading and suggestions.
The first author was supported by the Mathematics Department at
Indiana University Bloomington, via the Hazel King Thompson fellowship
and the Indiana University Graduate School under the Dissertation
Research Fellowship, and by the Luxembourg National research Fund
AFR/Bilateral-ReSurface 22/17145118 for the last edits of this
manuscript.
The second author was supported by the National Science
Foundation under Grant Numbers DMS-1507244 and DMS-2110143.

%%% Local Variables:
%%% mode: latex
%%% TeX-master: "Smoothings"
%%% End:

\section{Background on curves and currents}
\label{sec:curves}

Throughout this paper, $S$ is a fixed oriented compact 2-manifold
without boundary. (For a discussion of the more general surface case,
see Remark~\ref{rmk:opensurface}.)
If we fix an arbitrary (hyperbolic) metric on $S$, we will denote it
by $\Sigma$. 
The various types of curves and associated objects we consider are summarized in
Table~\ref{tab:object-notation}.

\subsection{Curves}

\begin{definition}[multi-curve]\label{def:multi-curve}
A \emph{concrete multi-curve} $\gamma$ on a surface $S$ is a smooth
1-manifold without boundary $X(\gamma)$ together with a map
(also called $\gamma$) from $X(\gamma)$
into~$S$. $X(\gamma)$ is not necessarily connected.
We say that $\gamma$ is \emph{trivial} if it is
homotopic to a point.
Two concrete multi-curves $\gamma$ and~$\gamma'$ are \emph{equivalent} if
they are related by a sequence of the following moves:
\begin{itemize}
    \item \emph{homotopy} within the space of all maps from $X(\gamma)$ to~$S$;
    \item \emph{reparametrization} of the 1-manifold;
      and
    \item \emph{dropping} trivial components.
\end{itemize}
The equivalence class of $\gamma$ is denoted by~$[\gamma]$, and we will
call it a \emph{multi-curve}.
If $X(\gamma)$ is connected, we will call $[\gamma]$ a \emph{curve}; a
curve is equivalent to a conjugacy class in $\pi_1(S)$.
When we just want to refer to the equivalence class of a
(multi-)curve, without
distinguishing a representative, we will use capital letters such
as~$C$.
A concrete multi-curve~$\gamma$ is \emph{simple} if $\gamma$ is
injective, and a multi-curve is simple if it has a concrete representative
that is simple.
We write $\Simples(S)$ for the space of simple multi-curves on~$S$
and $\Curves(S)$ for the space of all multi-curves.

We also consider \emph{oriented multi-curves}, which we will still
denote by~$\gamma$, in which $X(\gamma)$ is
oriented. We add the further condition in the equivalence
relation that the reparametrizations must be orientation-preserving.
In this paper, unless stated otherwise, we will be working with
oriented multi-curves. The spaces of oriented simple and general
multi-curves are denoted $\Simples^+(S)$ and $\Curves^+(S)$,
respectively.
\end{definition}

\begin{definition}[weighted multi-curve]\label{def:weighted-multi-curve}
  A \emph{weighted multi-curve} $C=\bigcup_i a_i C_i$ is a multi-curve in
  which each connected component is given a non-negative real
  coefficient~$a_i$. If coefficients are not specified, they are~$1$.
  We add further moves to the equivalence
  relation:
  \begin{itemize}
  \item \emph{merging} two parallel components and
    adding their weights; and
  \item \emph{nullifying}, deleting a
    component with weight~$0$.
  \end{itemize}
  For
  instance, $C \cup C$ is equivalent to $2C$.
The space of weighted multi-curves up to equivalence is denoted by
appending an $\mathbb{R}$ in front of their non-weighted names, so
$\R\Simples(S)$ is the space of weighted simple multi-curves and
$\R\Curves(S)$ is the space of weighted general multi-curves. This
is a slight abuse of notation since the weights are required to be non-negative.
\end{definition}

\begin{remark}
Since the weighted curve functionals we
are considering are not necessarily positive, they may increase after
dropping a component (see Definition~\ref{def:weighted-multi-curve}).
\end{remark}

\subsection{Crossings}
Loosely speaking, an essential crossing is a crossing of a
multi-curve that can't be homotoped away. We make this definition
precise as follows.

\begin{definition}[linked points on a circle]
  We say that two sets of two points $\{a,b\}$ and $\{c,d\}$ in
  $S^1$ are \emph{linked} if the four points are distinct and both
  connected components of
  $S^1- \{ a,b \}$ have an element of $\{c,d\}$.
\end{definition}

\begin{definition}[lift of a concrete curve]\label{def:lift}
  Given a concrete multi-curve~$\gamma$ on~$S$ and a choice
  $p \in X(\gamma)$, set $x=\gamma(p) \in S$. Pick a lift
  $\wt x \in \wt S$ of~$x$. The unique lifting property gives a unique
  lift $\wt{\gamma_p}: \wt{X}(\gamma;p) \to \wt S$ of $\gamma$ with
  $\wt{\gamma_p}(\wt p) = \wt x$, where $X(\gamma;p)$ is the component
  of~$X(\gamma)$ containing~$p$ and $\wt{X}(\gamma;p)$ is its
  universal cover with basepoint~$\wt p$.
\end{definition}

\begin{definition}\label{def:essential-crossing}
  A representative~$\gamma$ of a multi-curve~$C$ is \emph{minimal} if
  it has only transverse self-intersections and a minimum number of
  them.
    Let $\gamma$ be a concrete multi-curve on~$S$ representing a multi-curve $C$. A \emph{crossing} of $\gamma$ is a pair $(p,q)$ of
  points $p,q \in X(\gamma)$ so that
  $x \coloneqq \gamma(p) = \gamma(q) \in S$.
  A crossing that appears in some minimal representative is said
  to be an \emph{essential crossing}.
\end{definition}

\begin{definition}[smoothings]
  \label{def:smoothing}
  Let $(p,q) \in X(\gamma)$ be an essential crossing of~$\gamma$
  on~$S$. To make a \emph{smoothing}~$\gamma'$ of $(p,q)$, cut
  $X(\gamma)$ at $p$ and~$q$ and reglue the resulting four endpoints
  in one of the two other possible ways, getting a new $1$-manifold
  $X(\gamma')$. The map $\gamma'$ agrees with~$\gamma$; this is
  well-defined since $\gamma(p) = \gamma(q)$. In pictures we will
  homotop $\gamma'$ slightly to round out the resulting
  corners, and we will write $\gamma \reducesto \gamma'$. If $\gamma$ is oriented, then the \emph{oriented smoothing}
  is the smoothing that respects the orientation on $X(\gamma)$:
  \begin{equation*}
    \mfig{curves-3} \reducesto \mfig{curves-4}
  \end{equation*}
  
  If we obtain a curve $C'$ from $C$ by a sequence of
$k$ smoothings of essential crossings,  we will write
$C \reducesto_k C'$. (We check whether the
crossings are essential at each stage of this process; this is
more restrictive than checking at the beginning.)
\end{definition}

``Some'' in the definition of essential crossing can be replaced by ``every'':  if there is a crossing $x$ in some minimal representative
  $\gamma$ of~$C$, then for any other  representative $\gamma'$
  there is a corresponding crossing~$x'$ giving the same
  smoothings.

\begin{lemma}\label{lem:essential-cross}
  If $\gamma$ and $\gamma'$ are homotopic concrete
  multi-curves, with $\gamma$ minimal and $(p,q) \in X(\gamma)$ is an essential crossing of~$\gamma$,
  then there is a crossing $(p',q') \in X(\gamma')$ so
  that the smoothings of $(p,q)$ and of $(p',q')$ are homotopic.
\end{lemma}

\begin{proof}
 If $\gamma$ is connected, by a result of Hass and
  Scott \cite{HS94:ShorteningCurves}, $\gamma'$ can be
  turned into the minimal form~$\gamma$ using
  Reidemeister~I, II, and~III moves, with the Reidemeister~I and~II
  moves being used only in the forward (simplifying) direction. Since
  we know that $\gamma$ has a crossing of the desired type, we can
  trace the crossings backwards through these moves: there is a bijection between the crossings before and after a
  Reidemeister~III move that does not change the homotopy types of curves
  achievable by a single smoothing, and we can ignore the additional
  crossings created by backwards Reidemeister~I and~II moves.

  For multi-curves, Hass and Scott show that we need to add one more move, swapping the position of two
  components that are homotopically powers of the
  same primitive curve $\delta$
  \cite[pp.~31--32]{HS94:ShorteningCurves}. Again this operation induces a bijection
  between crossings that preserving the curves that result from
  smoothing.
\end{proof}

As a result of Lemma~\ref{lem:essential-cross}, we can speak of essential crossings also for concrete curves $\gamma$ that are not minimal.

\begin{subtheorems}
To give a more workable and algebraic criterion for essential
crossings, we
adopt a group-theoretic point of view. For $a,b \in \pi_1(S,*)$, write
$[a]$ for the free homotopy class of curves represented by $a$;
algebraically this is the union of the conjugacy classes of~$a$ and
of~$A$. Write $[a][b]$ for the multi-curve $[a] \cup [b]$. Say that a
factorization $[ab]$ \emph{reduces} if some (equiv.\ every) minimal
representative of $[ab]$ has a crossing so that one smoothing gives
$[a][b]$ and the other smoothing gives $[aB]$. Similarly say that
$[a][b]$ \emph{reduces} if some/every minimal representative has a crossing
whose smoothings give $[ab]$ and $[aB]$. (This is an abuse of
notation, since the reduction is not invariant under individual
conjugation of $a$ and $b$.)

For $a,b \in \pi_1(S,*)$, let $a * b$ be a 4-valent graph mapped into
$S$ obtained by taking the union of the paths $a,b$ joined at the
basepoint~$*$, and let $[a * b]$
be its free homotopy class.
Free homotopy means in particular invariance under simultaneous
conjugation: $[a * b] = [caC * cbC]$. A \emph{taut} representative of $[a * b]$
is one with a minimum number of transverse crossings.

Finally, if we fix a hyperbolic structure on~$S$, a non-trivial
element $a \in \pi_1(S,*)$ determines a geodesic $\ora{a}$ in
$\wt{S}$, the unique geodesic that stays a bounded distance from
$a^k \cdot \wt{*}$, with endpoints $a^+$ and $a^-$. Depending on the
cyclic order of the endpoints up to orientation reversal, we say that
$\ora{a}$ and $\ora{b}$ are
\begin{itemize}
\item \emph{parallel} if they appear in order $a^-,a^+,b^+,b^-$;
\item \emph{anti-parallel} if they appear in order $a^-,a^+,b^-,b^+$;
  and
\item \emph{cross} if they appear in order $a^-,b^-,a^+,b^+$.
\end{itemize}

We can now reformulate essential
crossings.

\begin{proposition}\label{prop:crossings-powers}
  If $a = c^k$ and $b = c^\ell$ are common powers of a primitive
  element~$c$ with $k,\ell \ge 1$, then the factorization $[ab]$
  reduces.
\end{proposition}

\begin{proof}
  See Hass and Scott \cite[Lemma
  1.12]{HS85:IntersectionsCurves}.
\end{proof}

\begin{proposition}\label{prop:crossings-triality}
  Let $a,b \in \pi_1(S,*)$ be non-trivial elements that are not
  common powers. Then in the following three trialities the same
  possibility holds in each case.
  \begin{enumerate}
  \item (Reduction) One of the following factorizations reduces
    to the other two:\\
    \begin{enumerate*}[(a),itemjoin=\qquad]
    \item $[ab]$;
    \item $[aB]$; or
    \item $[a][b]$.
    \end{enumerate*}
  \item (Geometric) For every taut embedding of $a * b$, around the
    4-valent vertex we see in cyclic
    order, up to orientation reversal:\\
    \begin{enumerate*}[(a),itemjoin=\qquad]
    \item $\mfig{curves-30}$;
    \item $\mfig{curves-31}$; or
    \item $\mfig{curves-32}$.
    \end{enumerate*}
  \item (Algebraic) The geodesics $\ora{a}$ and $\ora{b}$ are:\\
    \begin{enumerate*}[(a),itemjoin=\qquad]
    \item parallel;
    \item anti-parallel; or
    \item cross.
    \end{enumerate*}
  \end{enumerate}
\end{proposition}

\begin{proof}
  Since $a$ and $b$ are not common powers, the endpoints of
  $\ora{a}$ and $\ora{b}$ are all distinct from each other, so
  the algebraic triality is mutually exclusive and exhaustive. For the
  geometric triality, for any given embedding of $a * b$, one of the
  three given possibilities holds, but it is not \emph{a priori} clear that
  the same possibility holds for all taut embeddings. For the
  reduction triality, the possibilities are mutually exclusive (we
  can't have circular reductions, since at each reduction length with
  respect any Riemannian metric strictly decreases) but now it is not
  clear that any of the three possibilities holds.
  
  We start with the geometric triality. Consider any taut representative
  $\rho$ of $[a * b]$ in one of the cases (a), (b), or~(c). Consider
  the (concrete) multi-curve~$\gamma$ given by replacing the
  $4$-valent vertex with a crossing. If we pick orientations on $\gamma$
  correctly, this is an embedding of $[ab]$, $[aB]$, or $[a][b]$,
  respectively, with a distinguished crossing~$x$.

  We claim that $\gamma$ is minimal. Otherwise, there is some reducing
  isotopy using RIII moves and some RI/RII reduction
  \cite{HS94:ShorteningCurves}. Any RIII move on~$\gamma$
  can be copied with an isotopy of~$\rho$, as can an RI or RII reduction of
  $\gamma$ that does not involve~$x$. An RI
  reduction at~$x$ is only possible if $a$ or $b$ is the identity, and
  an RII reduction involving~$x$ also reduces $\rho$,
  reducing the number of intersections of $\rho$ by one. In any case
  this contradicts minimality of~$\rho$ or
  non-triviality of $a,b$.

  We have thus shown that one of $[ab]$, $[aB]$, or $[a][b]$ has a
  minimal representative $\gamma$ with a crossing~$x$ that can be
  smoothed to give the other two multi-curves. Since we cannot have a
  circular chain of essential smoothings (for instance because we get
  strict inequalities of hyperbolic length), it follows that the cases
  are exclusive and thus that \emph{every} taut representative of
  $[a * b]$ has the same cyclic order around the 4-valent vertex.
  Combined this
  proves equivalence of the first two trialities.

  For the remaining equivalence, suppose first we
  are in case~(c) of the geometric triality. For the fixed hyperbolic
  structure, take $\gamma$ to be the geodesic representative of the free
  homotopy class $[a][b]$, intersecting at the distinguished point
  $x \in S$, and let $\ora{a}$ and $\ora{b}$ be the respective
  lifts passing through some lift $\wt{x}$ of~$x$. Then $\ora{a}$ and
  $\ora{b}$ intersect only once and thus cross at infinity, so we are
  also in case~(c) of the algebraic triality.

  If we are in case~(a) of the geometric triality, again take the
  geodesic representative $\gamma$ of $ab$, with distinguished
  crossing~$x$. For concreteness, if necessary reverse the
  orientation of~$S$ so the branches appear in the order shown below.
  If we
  perform the disconnected (oriented) smoothing at $x$, we get two
  components $\gamma_a \in [a]$ and $\gamma_b \in [b]$; these concrete
  curves are typically not minimal. However, $\gamma_a$ and $\gamma_b$ are
  regular isotopic to a taut curve \cite[Lemma 5.5]{Thurston14:PBI},
  i.e., there is no RI move required in making them taut. This is
  equivalent to saying that their lifts $\wt{\gamma_a}$ and $\wt{\gamma_b}$ to
  $\wt{S}$ passing through a neighborhood of $\wt{x}$ are embedded in
  the disk. Now consider
  the endpoints $a^+, a^-$ of $\wt{\gamma_a}$. As an oriented curve,
  $\wt{\gamma_a}$ is geodesic outside of neighborhoods of lifts of $x$, and
  inside those neighborhoods turns to the left. It then follows that
  $(ab)^+, a^+, a^-, (ba)^-$ occur in that cyclic order on
  $\partial \wt{S}$, and similarly $(ab)^-, b^-,b^+,(ba)^+$ occur
  in that cyclic order:
  \[
    \resizebox{60mm}{!}{\Huge{%% Creator: Inkscape 1.3 (0e150ed6c4, 2023-07-21), www.inkscape.org
%% Accompanies image file 'lifts.pdf' (pdf, eps, ps)
\begingroup%
  \makeatletter%
  \providecommand\color[2][]{%
    \errmessage{(Inkscape) Color is used for the text in Inkscape, but the package 'color.sty' is not loaded}%
    \renewcommand\color[2][]{}%
  }%
  \providecommand\transparent[1]{%
    \errmessage{(Inkscape) Transparency is used (non-zero) for the text in Inkscape, but the package 'transparent.sty' is not loaded}%
    \renewcommand\transparent[1]{}%
  }%
  \providecommand\rotatebox[2]{#2}%
  \newcommand*\fsize{\dimexpr\f@size pt\relax}%
  \newcommand*\lineheight[1]{\fontsize{\fsize}{#1\fsize}\selectfont}%
  \ifx\svgwidth\undefined%
    \setlength{\unitlength}{572.76495025bp}%
    \ifx\svgscale\undefined%
      \relax%
    \else%
      \setlength{\unitlength}{\unitlength * \real{\svgscale}}%
    \fi%
  \else%
    \setlength{\unitlength}{\svgwidth}%
  \fi%
  \global\let\svgwidth\undefined%
  \global\let\svgscale\undefined%
  \makeatother%
  \begin{picture}(1,0.79061173)%
    \lineheight{1}%
    \setlength\tabcolsep{0pt}%
    \put(0,0){\includegraphics[width=\unitlength,page=1]{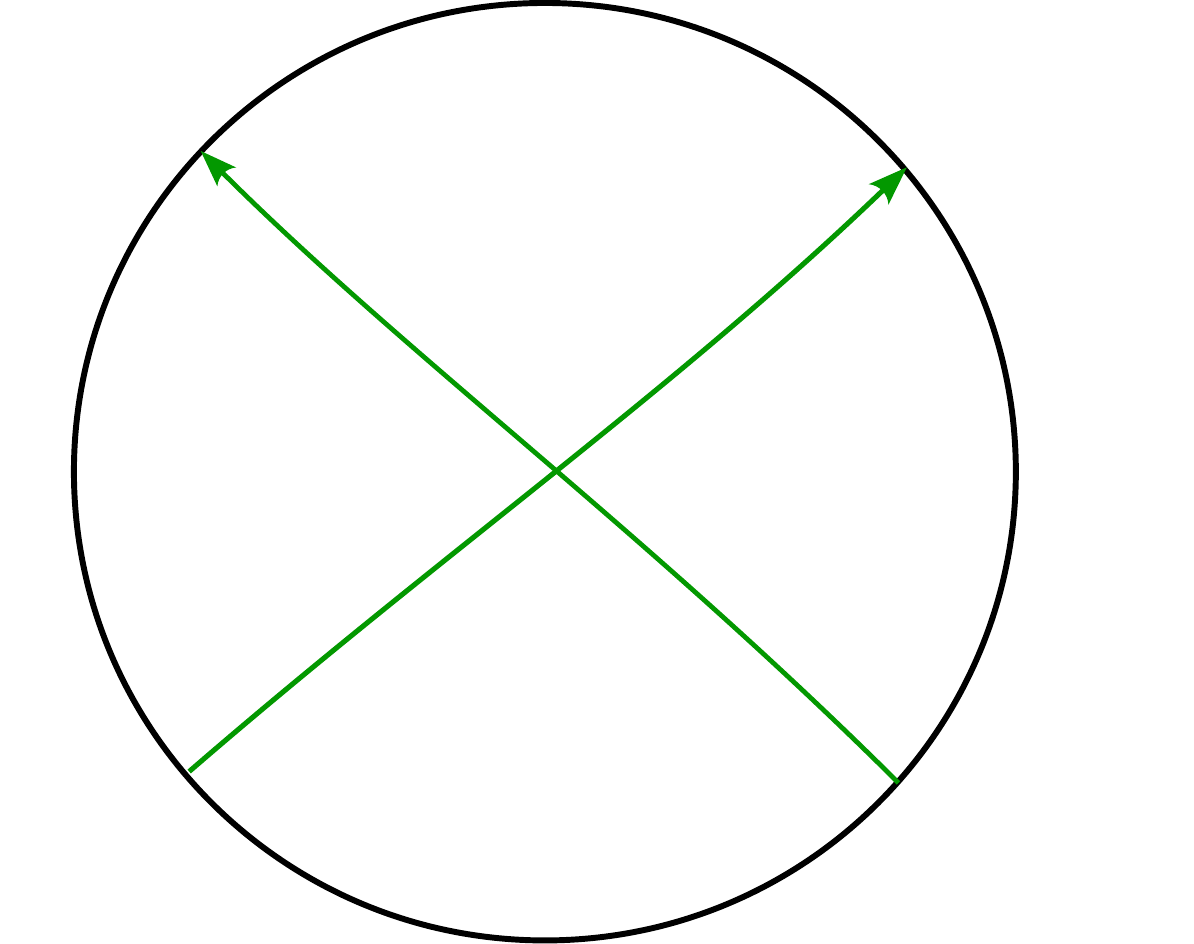}}%
    \put(0.76818664,0.6552911){\color[rgb]{0,0,0}\makebox(0,0)[lt]{\lineheight{1.25}\smash{\begin{tabular}[t]{l}$ba^+$\end{tabular}}}}%
    \put(0.06038788,0.07996839){\color[rgb]{0,0,0}\makebox(0,0)[lt]{\lineheight{1.25}\smash{\begin{tabular}[t]{l}$ba^-$\end{tabular}}}}%
    \put(0.08226172,0.66316997){\color[rgb]{0,0,0}\makebox(0,0)[lt]{\lineheight{1.25}\smash{\begin{tabular}[t]{l}$ab^+$\end{tabular}}}}%
    \put(0.75427633,0.09261365){\color[rgb]{0,0,0}\makebox(0,0)[lt]{\lineheight{1.25}\smash{\begin{tabular}[t]{l}$ab^-$\end{tabular}}}}%
    \put(0.79524099,0.15039774){\color[rgb]{0,0,0}\makebox(0,0)[lt]{\lineheight{1.25}\smash{\begin{tabular}[t]{l}$b^-$\end{tabular}}}}%
    \put(0.81867417,0.56831173){\color[rgb]{0,0,0}\makebox(0,0)[lt]{\lineheight{1.25}\smash{\begin{tabular}[t]{l}$b^+$\end{tabular}}}}%
    \put(0,0){\includegraphics[width=\unitlength,page=2]{lifts.pdf}}%
    \put(0.03202651,0.56937251){\color[rgb]{0,0,0}\makebox(0,0)[lt]{\lineheight{1.25}\smash{\begin{tabular}[t]{l}$a^+$\end{tabular}}}}%
    \put(0.02269948,0.1712114){\color[rgb]{0,0,0}\makebox(0,0)[lt]{\lineheight{1.25}\smash{\begin{tabular}[t]{l}$a^-$\end{tabular}}}}%
  \end{picture}%
\endgroup%
}}
  \]
  It therefore
  follows that $\ora{a}$ and
  $\ora{b}$ are parallel, so we are in case~(a) of the algebraic
  triality.

  Case~(b) follows in the same way, interchanging $b \leftrightarrow B$.
\end{proof}

We will also need a result on the relative position of axes of two hyperbolic elements. 

\begin{lemma}
  If $a$ and $b$ are two hyperbolic elements of $\PSL(2,\R)$ with
  crossing axes, then $ab$ and $ba$ are also hyperbolic with parallel
  axes.
  \label{lem:parallel_implies_crossing}
\end{lemma}

\begin{proof}
  We will prove that the
  endpoints of the various geodesics must be laid out as follows, up
  to orientation reversal:
  \[
    \resizebox{60mm}{!}{\Huge{%% Creator: Inkscape 1.3 (0e150ed6c4, 2023-07-21), www.inkscape.org
%% Accompanies image file 'rcrossing.pdf' (pdf, eps, ps)
\begingroup%
  \makeatletter%
  \providecommand\color[2][]{%
    \errmessage{(Inkscape) Color is used for the text in Inkscape, but the package 'color.sty' is not loaded}%
    \renewcommand\color[2][]{}%
  }%
  \providecommand\transparent[1]{%
    \errmessage{(Inkscape) Transparency is used (non-zero) for the text in Inkscape, but the package 'transparent.sty' is not loaded}%
    \renewcommand\transparent[1]{}%
  }%
  \providecommand\rotatebox[2]{#2}%
  \newcommand*\fsize{\dimexpr\f@size pt\relax}%
  \newcommand*\lineheight[1]{\fontsize{\fsize}{#1\fsize}\selectfont}%
  \ifx\svgwidth\undefined%
    \setlength{\unitlength}{601.45427554bp}%
    \ifx\svgscale\undefined%
      \relax%
    \else%
      \setlength{\unitlength}{\unitlength * \real{\svgscale}}%
    \fi%
  \else%
    \setlength{\unitlength}{\svgwidth}%
  \fi%
  \global\let\svgwidth\undefined%
  \global\let\svgscale\undefined%
  \makeatother%
  \begin{picture}(1,0.96290048)%
    \lineheight{1}%
    \setlength\tabcolsep{0pt}%
    \put(0,0){\includegraphics[width=\unitlength,page=1]{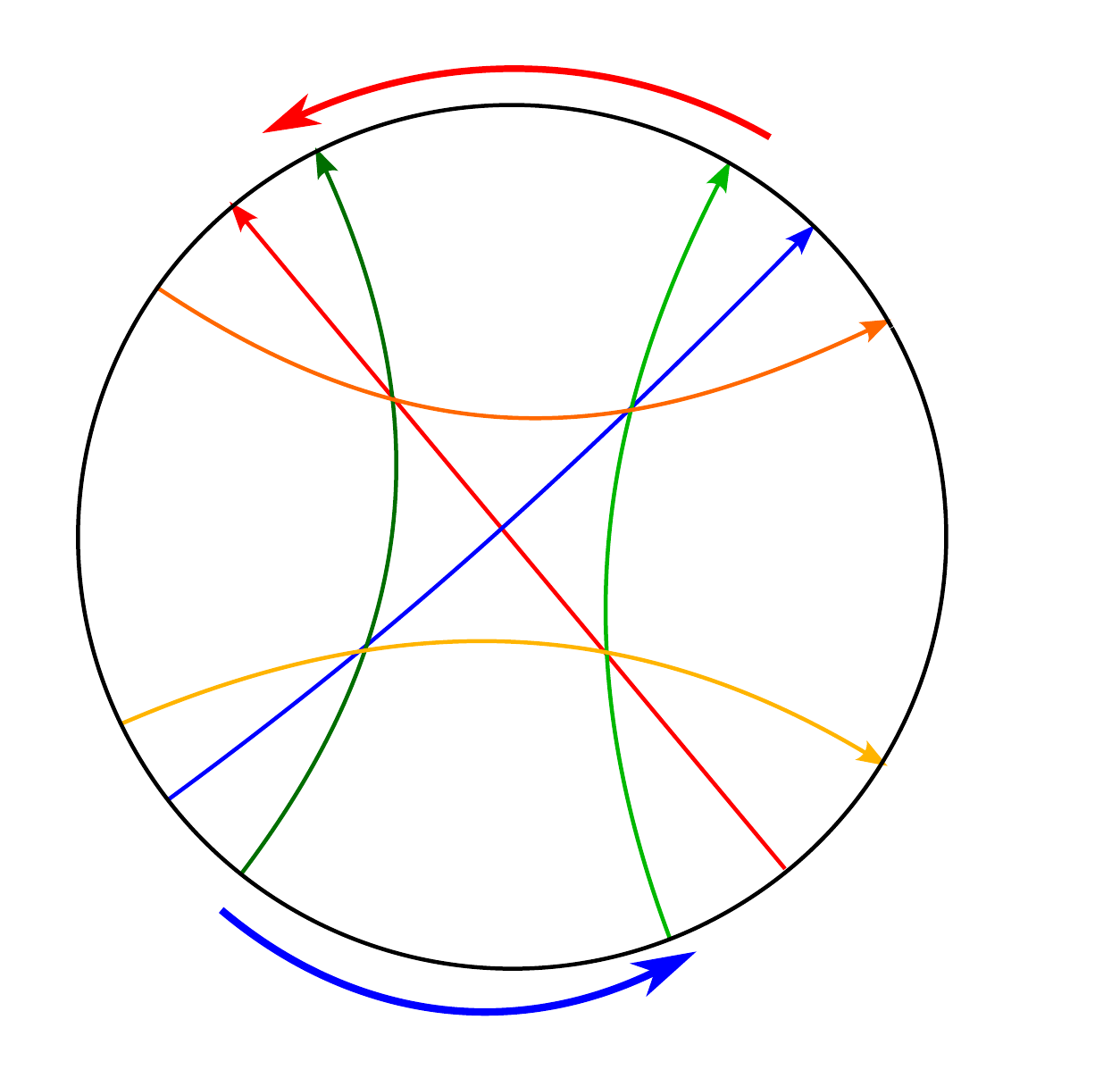}}%
    \put(0.6906423,0.84550275){\color[rgb]{0,0,0}\makebox(0,0)[lt]{\lineheight{1.25}\smash{\begin{tabular}[t]{l}$ba^+$\end{tabular}}}}%
    \put(0.60202973,0.04886308){\color[rgb]{0,0,0}\makebox(0,0)[lt]{\lineheight{1.25}\smash{\begin{tabular}[t]{l}$ba^-$\end{tabular}}}}%
    \put(0.18986065,0.85973288){\color[rgb]{0,0,0}\makebox(0,0)[lt]{\lineheight{1.25}\smash{\begin{tabular}[t]{l}$ab^+$\end{tabular}}}}%
    \put(0.40834104,0.92451707){\color[rgb]{1,0,0}\makebox(0,0)[lt]{\lineheight{1.25}\smash{\begin{tabular}[t]{l}$a$\end{tabular}}}}%
    \put(0.38069309,0.00879213){\color[rgb]{0,0,1}\makebox(0,0)[lt]{\lineheight{1.25}\smash{\begin{tabular}[t]{l}$b$\end{tabular}}}}%
    \put(0.14966573,0.08319217){\color[rgb]{0,0,0}\makebox(0,0)[lt]{\lineheight{1.25}\smash{\begin{tabular}[t]{l}$ab^-$\end{tabular}}}}%
    \put(0.07940002,0.1948772){\color[rgb]{0,0,0}\makebox(0,0)[lt]{\lineheight{1.25}\smash{\begin{tabular}[t]{l}$b^-$\end{tabular}}}}%
    \put(0.73606966,0.76237808){\color[rgb]{0,0,0}\makebox(0,0)[lt]{\lineheight{1.25}\smash{\begin{tabular}[t]{l}$b^+$\end{tabular}}}}%
    \put(0.14624663,0.78929958){\color[rgb]{0,0,0}\makebox(0,0)[lt]{\lineheight{1.25}\smash{\begin{tabular}[t]{l}$a^+$\end{tabular}}}}%
    \put(0.70880588,0.14867489){\color[rgb]{0,0,0}\makebox(0,0)[lt]{\lineheight{1.25}\smash{\begin{tabular}[t]{l}$a^-$\end{tabular}}}}%
    \put(0.01972425,0.72263571){\color[rgb]{0,0,0}\makebox(0,0)[lt]{\lineheight{1.25}\smash{\begin{tabular}[t]{l}$aB^+$\end{tabular}}}}%
    \put(0.81241815,0.67170725){\color[rgb]{0,0,0}\makebox(0,0)[lt]{\lineheight{1.25}\smash{\begin{tabular}[t]{l}$aB^-$\end{tabular}}}}%
    \put(-0.00377502,0.27707596){\color[rgb]{0,0,0}\makebox(0,0)[lt]{\lineheight{1.25}\smash{\begin{tabular}[t]{l}$Ba^+$\end{tabular}}}}%
    \put(0.80086193,0.24724933){\color[rgb]{0,0,0}\makebox(0,0)[lt]{\lineheight{1.25}\smash{\begin{tabular}[t]{l}$Ba^-$\end{tabular}}}}%
  \end{picture}%
\endgroup%
}}
  \]
  By the North-South dynamics of $a$ and $b$ on the circle, $ab$ maps the
  interval $(a^+,b^+)$ strictly inside itself, which implies that it
  is hyperbolic with attracting endpoint $(ab)^+$ in that interval;
  similarly $ab$ maps $(a^-,b^-)$ to a strict superset of itself, so
  $(ab)^- \in (a^-,b^-)$. The actions of $a$ and $b$
  interchange $(ab)^- \leftrightarrow (ba)^-$ and
  $(ab)^+ \leftrightarrow (ba)^+$, forcing the configuration shown.
  (The endpoints $(ab)^\pm,(ba)^\pm$ are necessarily on the
  portions of the circle where $a$ and $b$ move in opposite
  directions.)
\end{proof}

\begin{remark}
  A very similar argument shows that if $\ora{a}$ and $\ora{b}$ are
  parallel, then $\ora{ab}$ and $\ora{ba}$ cross. However, in the
  self-crossing case, the condition of $\ora{ab}$ and $\ora{ba}$ crossing is necessary but
  not sufficient to conclude that $\ora{a}$ and $\ora{b}$ are parallel, as the following example illustrates.

  On the pair of pants~$P$, consider
the curve~$\gamma$ shown solid below, with marked
crossing~$x$.
\[
  \mfig{curves-21}
\]
If we take the natural generators $a,b$ of $\pi_1(P)$ shown dashed,
then $\gamma$ is in the free homotopy class of $[ab]$. Note the crossing at $X$
divides $\gamma$ into two subpaths, representing
$p = bab$ and $q = B$ (adopting the convention that capital letters
mean inverse, and connecting $x$ to $*$ by the short path). Since $pq
= ba$ and $qp = ab$, the
corresponding lifts of $\gamma$ to $\wt P$ cross at infinity, since
their geodesic axes $\ora{ab}$ and $\ora{ba}$ cross. But $x$ is not an
essential crossing in any reasonable
sense: it can clearly be isotoped away, and it does not satisfy the desired smoothing: if we take lengths with
respect to any Riemannian metric,
$\ell([p Q]) = \ell([ab^3]) > \ell([ab]) = \ell([pq])$, and not
the other way around.
\end{remark}

\end{subtheorems}

We have similar notions for weighted curves.
\begin{definition}
  Let $C \in \R\Curves^+(S)$ be a weighted oriented curve, and let
  $\gamma$ be a concrete representative of the underlying unweighted
  curve. Let $(p,q) \in X(\gamma)$ be an essential crossing
  of~$\gamma$ on~$S$, and let $\gamma'$ be the smoothing as defined
  above. If the corresponding components of~$C$ have equal weight
  $w > 0$, then we can make a weighted curve $C'$ by giving every
  component in $[\gamma']$ not involved in the smoothing the same
  weight it had in~$C$, and giving the one or two new components
  weight~$w$. In this case we say that $C'$ is obtained from $C$ by a
  \emph{smoothing of weight~$w$}, and write $C \reducesto_w C'$.
\end{definition}
Using this, we define conditions on a weighted curve functional, extending Definition~\ref{def:properties}.
\begin{itemize}
\item \textbf{Weighted quasi-smoothing:}  There is a constant $R>0$
  so that, $C \reducesto_w C'$ are weighted curves related by a
  smoothing of weight~$w$,
  \[
    f(C) \ge f(C')-wR.
  \]
\item \textbf{Weighted smoothing:} Take $R=0$ in the definition of
  weighted quasi-smoothing.
\end{itemize}

See Proposition~\ref{prop:weightextend} for justification for these
definitions.

\subsection{Space of geodesics}

\begin{definition}[Boundary at infinity]
\label{def:boundaryinfty}
Endow $S$ with a complete
hyperbolic  metric~$g$; we denote the pair $(S,g)$ by~$\Sigma$. Then
we can consider the metric
universal covering $p \colon \wt{\Sigma} \to \Sigma$, with
$\wt{\Sigma}$ isometric to the hyperbolic plane.
Two quasi-geodesic rays $c,c' \colon [0,\infty) \to \wt{\Sigma}$ are
said to be \emph{asymptotic} if there exists a constant $K$ for which
$d(c(t),c'(t)) \leq K$ for all $t \geq 0$. We define
$\partial_{\infty} \Sigma$, the \emph{boundary at infinity}
of~$S$, to be the set of equivalence classes of
asymptotic quasi-geodesic rays. This boundary at infinity is
independent of the hyperbolic structure on~$S$ up to canonical
homeomorphism.
\end{definition}

\begin{definition}[Space of oriented geodesics]

\label{def:spacegeodesics}
Let $G^{+}(\Sigma)$ denote the space of \emph{oriented geodesics} in
$\wt{\Sigma}$, i.e.,
\[
G^+(\Sigma) \coloneqq\partial_\infty \Sigma  \times\partial_\infty\Sigma - \Delta.
\]
Since this is independent of the hyperbolic structure, we will also
write $G^+(S)$.
\end{definition}

\subsection{Geodesic currents}

\begin{definition}[Geodesic current definition 1]\label{def:currents-1}
We define $\GC^+(S)$, the space of oriented geodesic currents on $S$,
to be the
space of $\pi_1(S)$-invariant (positive) Radon measures on
$G^{+}(S)$.
\end{definition}

Since the action of $\pi_1(S)$ on $G^+(S)$ is not discrete, this
definition is hard to visualize. We give alternate
definitions that play a key role in our proofs. For a hyperbolic
surface~$\Sigma$, let $UT\Sigma$ be the unit tangent bundle and let
$\phi_t$ be the geodesic flow on it.

\begin{definition}[Geodesic current definition 2]\label{def:currents-2}
  We can also define oriented geodesic currents to be the space of
  (positive) finite
  Radon measures $\mu$ on
$UT \Sigma$ which are invariant under the geodesic flow, in the sense
that $(\phi_t)_*(\mu) = \mu$ for all $t\in\R$.
\end{definition}

We can also look at induced measures on cross-sections.

\begin{definition}[Geodesic current definition 3]\label{def:currents-3}
A geodesic current is a \emph{transverse invariant measure}: a family of
measures $\{\mu_{\tau}\}_{\tau}$,
where $\tau \subset UT \Sigma$ is a submanifold-with-boundary of the unit tangent
bundle of real codimension 1 transverse to the geodesic
foliation~$\mathcal{F}$, with the following invariance property:  if
$x_1 \in \tau_1, x_2 \in \tau_2$ are two points on transversal
submanifolds on the same leaf of $\mathcal{F}$, and $\phi \colon U_1
\to U_2$ a holonomy diffeomorphism between neighborhoods of $x_1$ and
$x_2$ respectively, then $\phi_{*}\mu_{\tau_1}=\mu_{\tau_2}$.
\end{definition}

The equivalence of the three definitions was known to Bonahon
\cite[Chapter~4]{Bonahon86:EndsHyperbolicManifolds}. Details can be
found in \cite[Section~3.4]{AL17:HyperbolicStructures}. Briefly, given
a measure $\mu$ on $UT\Sigma$ as in Definition~\ref{def:currents-2}
and a cross-section~$\tau$, there is an induced \emph{flux} $\mu_\tau$
on~$\tau$, as explained in Definition~\ref{def:flux}; this gives
a geodesic current in the sense of Definition~\ref{def:currents-3}.
Lifting to the universal cover then gives a geodesic current in the
sense of Definition~\ref{def:currents-1}. We can also relate
Definitions~\ref{def:currents-1} and~\ref{def:currents-2} directly by
connecting both to measures on $UT\H^2$ that are invariant under both
$\pi_1(\Sigma)$ and the geodesic flow
\cite[Proposition~8.1]{BO07:Equidistribution}.

Because Definitions~\ref{def:currents-1} and~\ref{def:currents-3} are
invariant under the mapping class
group, we will also write $G^+(S)$ and $\GC^+(S)$ in the sequel. We will also write $\pi_1(S)$. On the other hand, we will emphasize the dependence of $UT\Sigma$ on the hyperbolic structure.

\begin{remark}
  For $\Sigma,\Sigma'$ hyperbolic surfaces, any homeomorphism $\psi \colon \Sigma \to \Sigma'$, there is a homeomorphism $\hat\psi \colon UT\Sigma \to UT\Sigma'$ that is an orbit
  equivalence, and it is tempting to use this to define an induced map
  between geodesic currents in the sense of
  Definition~\ref{def:currents-2}. But this does
  not quite work: $\hat\psi_*$ does not take geodesic currents to
  geodesic currents. Orbit equivalence means that
  $\psi(\phi_t(x)) = \phi_{f(t)}(\psi(x))$ for some monotonic function
  $f \colon \R \to \R$, but this is not enough to guarantee that
  $(\phi_t)_*(\hat\psi_*\mu) = \hat\psi_*\mu$, and indeed this is
  usually false. See \cite[Theorem~3.6]{Wil14:Marked}.
\end{remark}

\subsection{Oriented vs unoriented currents}
\label{sec:unor-current}

We will be mostly working in the setting of oriented geodesic
currents, but most of our natural examples (like measured laminations)
use unoriented currents.

\begin{definition}[Unoriented geodesic currents]
\label{def:unorientedcurrents}

To define the subspace $\GC(S) \subset \GC^{+}(S)$ of
\emph{unoriented geodesics currents},
let $\sigma: G^{+}(S) \to G^{+}(S)$ be the \emph{flip map}
that switches the two factors in the definition of $G^+(S)$,
reversing the orientation of the geodesic. This induces a map
$\sigma_* \colon \GC^{+}(S)
\to\GC^{+}(S)$. Set
\[
\GC(\Sigma) \coloneqq \{ \mu \in  \GC^{+}(\Sigma) \mid \sigma_*(\mu)=\mu \}.
\]
There is a map $\Pi \colon \GC^{+}(S) \to\GC^{+}(S)$ 
given by $\Pi(\mu)\coloneqq
\frac{1}{2}(\mu + \sigma_*(\mu))$ with image the subset of unoriented currents.
\end{definition}

In the proof of the main result, we shall work with
oriented currents $\GC^+(S)$;  oriented currents are more general
and just as easy to
work with for our proof.

The maps $\sigma$ and $\Pi$ have obvious analogues for curves.

\begin{proof}[Proof of Corollary~\ref{cor:unorconvex}, assuming
  Theorem~\ref{thm:convex}]
  For a curve functional as in the statement,
  let $g \colon \R \Curves^+(S) \to \R$ be $f \circ \Pi$. Then $g$
  satisfies quasi-smoothing, with the same constant as~$f$, and thus
  by Theorem~\ref{thm:convex} extends uniquely to a continuous
  function $\bar g \colon \GC^+(S) \to \R$. The desired extension $\bar f$ is the
  restriction of $\bar g$ to the subspace of unoriented currents.
\end{proof}

\subsection{Curves as currents}
\label{sec:curves-as-currents}

For an oriented multi-curve~$C$ on a hyperbolic surface~$\Sigma$, we
can construct a geodesic current as follows.

For Definition~\ref{def:currents-1}, consider all lifts of all
non-trivial components of~$C$ to $\wt\Sigma$. Each lift gives a quasi-geodesic in
$\wt\Sigma$, and thus a unique fellow-traveling geodesic in
$G^+(S)$; we thus get an infinite countable subset of
$G^+(S)$, which is easily seen to be discrete and
$\pi_1(S)$-invariant. Define the geodesic current to be the
$\delta$-function of this subset.

For Definition~\ref{def:currents-2}, take the geodesic
representative~$\gamma$ of~$C$, and consider the canonical
lift~$\wt\gamma$ of~$\gamma$ to $UT\Sigma$; this is an orbit
of~$\phi_t$. Let $\mu_C$ be the length-normalized $\delta$-function on
this orbit. That is, for an open set~$U$ we set $\mu_C(U)$ to
be the total length of $\wt \gamma \cap U$ with respect to the natural
Riemannian metric on $UT\Sigma$.

For Definition~\ref{def:currents-3} on a cross-section~$\tau$, again
take $\wt \gamma \subset UT\Sigma$, and let $\mu_\tau$ be the
$\delta$-function on the discrete set of points $\wt\gamma \cap \tau$.
(This is compatible with the length normalization in the previous
paragraph.)

From any of these points of view the inclusion naturally extends to weighted
multi-curves.

Weighted closed (multi-)curves are dense in the
space of geodesic currents
\cite[Proposition~4.4]{Bonahon86:EndsHyperbolicManifolds}.

Geometric intersection number extends continuously to geodesic
currents \cite[Proposition~4.5]{Bonahon86:EndsHyperbolicManifolds}.
The space of measured laminations (defined in
\cite[Section~1.7]{HP92:TrainTracks}) can be characterized
\cite[Proposition~17]{Bonahon88:GeodesicCurrent}
as a subset of
(unoriented) geodesic currents:
\[
\ML(S) \coloneqq \{ \alpha \in \GC(S) \mid \sigma(\alpha)=\alpha, i(\alpha,\alpha)=0 \}.
\]
The following square of inclusions is useful to keep in
mind:
\[
  \begin{tikzpicture}[x=2.25cm,y=1.25cm]
    \node (simples) at (0,0) {$\R\Simples(S)$};
    \node (laminations) at (1,0) {$\ML(S)$};
    \node (curves) at (0,-1) {$\R\Curves(S)$};
    \node (currents) at (1,-1) {$\GC(S)$.};
    \draw[right hook->] (simples) to (curves);
    \draw[right hook->] (simples) to (laminations);
    \draw[right hook->] (curves) to (currents);
    \draw[right hook->] (laminations) to (currents);
  \end{tikzpicture}
\]
Here the horizontal inclusions have dense image:
Douady and Hubbard showed that weighted simple multi-curves are dense
in $\ML$
\cite[Theorem]{DH75:DensityStrebel}. Soon after, Masur showed that
weighted simple curves are also
dense
\cite[Theorem~1]{Mas79:JenkinsStrebelOneCyl}.

\subsection{Topology on currents and measures}
\label{subsec:topologycurrents}
Let $\Meas(X)$ denote the space of positive Borel measures on a topological space $X$. $\Meas_1(X)$ will denote the space of Borel probability measures on $X$.
The topology on $\Meas(X)$ is the \emph{weak$^*$ topology}, i.e.,
the smallest topology so that,
for all continuous, compactly-supported functions~$f$ on~$G^+(S)$, the functional
\[
\mu \mapsto \int_{G^+(S)} fd\mu
\]
is continuous. 

The topology on $\GC^+(S)$ (in
Definition~\ref{def:currents-1}) is the weak$^*$ topology as a subspace
of measures 
on $G^+(S)$, 
We could also look at the weak$^*$ topology on currents as a subspace of
measures on $UT\Sigma$ (Definition~\ref{def:currents-2}); these two
points of view give the same
topology \cite[Proposition~8.1]{BO07:Equidistribution}.
On the other
hand, if we take $\tau$ to be a closed cross-section (including the
boundary), the map $\mu
\mapsto \mu_\tau$ relating Definitions~\ref{def:currents-2}
and~\ref{def:currents-3} is \emph{not} usually continuous with
respect to the weak$^*$ topologies,
so it is delicate to use the weak$^*$
topology on $\Meas(\tau)$; see
Lemma~\ref{lem:flux-continuous} and Remark~\ref{rem:flux-not-cont}.

There are in fact two topologies on spaces of measures that are
sometimes called the weak$^*$ topology; the one above is also called
the \emph{wide topology}
\cite[\textit{inter alia}]{EoM10:Convergence}.
There is also
the \emph{narrow topology} on measures $\Meas(X)$ on a space~$X$,
defined as the smallest
topology so that,
for all continuous bounded~$f$ on~$X$, the functional
$\mu \mapsto \int_{X} fd\mu$
is continuous. (That is, replace \emph{compactly supported} with
\emph{bounded} in the functions considered.)

\begin{remark}
Some authors call the weak$^*$ topology the vague topology, and use
the term weak topology for the narrow one (for example, Bauer's
textbook \cite{Bau01:Measure}).
However, this conflicts with the notion of weak topology used for Banach spaces, and we prefer the wide/narrow usage.
\end{remark}

In general, the weak$^*$ or wide topology is
weaker than the narrow topology, but in some particular cases they are
equivalent. 

A topological space $X$ is called \emph{Polish} if its topology has a countable base and can be defined by a complete metric.
\begin{theorem}[{\cite[Theorem 31.5]{Bau01:Measure}}]
\label{thm:spacemeasuresmetrizable}
Let $X$ be a locally compact topological space. Then $X$ is Polish if
and only if $\Meas(X)$ is Polish with respect to the weak$^*$-topology.
\end{theorem}

Thus, $\GC^+(S)$ is second countable completely metrizable and second countable, so in particular
sequential continuity is the same as continuity.
Although we will be dealing with Radon measures, for Polish spaces it is equivalent to consider the a priori more general class of Borel measures.

\begin{theorem}[{\cite[Theorem 26.3]{Bau01:Measure}}]
On a Polish space, a locally finite Borel measure is a $\sigma$-finite Radon measure.
\end{theorem}

The narrow and wide topology agree in certain sequences on locally compact spaces.

\begin{theorem}[{\cite[Theorem 30.8]{Bau01:Measure}}]
Let $X$ be a locally compact topological space, and $\mu_n$ a sequence of Radon measures of uniformly bounded mass converging to a Radon measure $\mu$ in the wide topology.
Then $\mu_n$ converges to $\mu$ in the narrow topology if and only if $\lim_n \mu_n(X) \to \mu(X)$.
\end{theorem}

\begin{proposition}[{\cite[Corollary 30.9]{Bau01:Measure}}]
Let $X$ be a locally compact topological space, and $\mu$, $\mu_n$, Borel probability measures. Then
$\mu_n \to \mu$ in the wide topology if and only if $\mu_n \to \mu$ in the narrow topology.
\end{proposition}

In particular, when $X$ is Polish, the two topologies agree for the
space $\Meas_1(X)$ of Borel probability measures.

\begin{proposition}
  If $X$ is a locally compact Polish space, the weak$^*$ and narrow
  topologies agree on $\Meas_1(X)$.
\label{prop:narrowwide}
\end{proposition}

\begin{convention}
  For any topological space~$X$, we will always use the weak$^*$
  topology on $\Meas(X)$. We will also work with the dense subspace
  $\R X \subset \Meas(X)$ of finitely-supported measures on~$X$
  (also called weighted linear combinations of~$X$), with its
  inherited subspace topology. (The weights are positive, but we
  usually omit that from the notation.)
  \label{conv:finitemeasures}
\end{convention}

\begin{remark}
  If we limit to sums with at most $k$ terms in
  the linear combination (or points in the support of the measure),
  we get a further subspace temporarily denoted
  $\R_{(k)}X \subset \R X \subset \Meas(X)$. We can view $\R_{(k)} X$
  as a quotient of $(\R_{\ge 0} \times X)^k$, quotienting by the action of the
  symmetric group and other evident equivalences; as such, it inherits an
  obvious topology, which agrees with the subspace topology.
  \label{rem:boundedfinitemeasures}
\end{remark}

\begin{remark}
  Geodesic currents can also be defined more generally for finite type hyperbolic surfaces. Depending on if we consider ends as cusps or funnels, we get two different spaces, that we will call $\GC_{\mathrm{cusp}}(S)$ and $\GC_{\mathrm{open}}(S)$, respectively.
In the first case, we define the space of geodesic currents analogously to Definition~\ref{def:currents-1} for closed surfaces, i.e., as invariant measures supported on the space of geodesics of the universal cover, noting that now the space of geodesics contains arcs going from cusp to cusp. 
In the second case, we consider geodesic currents supported on geodesics projecting to the convex core of the surface.
Extending continuously curve functionals to these spaces is more delicate.

In the case of $\GC_{\mathrm{cusp}}(S)$, let $S$ be a surface with two open ends and let $\Sigma$ be a complete hyperbolic metric of finite area, with respect to which
the ends of $S$ are cusps. Let $a$ be an arc going from cusp to cusp. Let $C_n$ be the closed curve going along for some time $a$, winding $n$ times around one cusp,
going along $a$ again and winding $n$ times around the other cusp.
Observe that although $C_n \to a$ in the weak$^*$ topology, and $i(a,a)=0$, we have $i(a,C_n)=2n$, so intersection number is not a continuous function on geodesic currents.

The case of $\GC_{\mathrm{open}}(S)$ is different, since intersection number is continuous. Indeed, let $\overline{\Sigma}$ denote the convex core of the complete hyperbolic surface of infinite area, which is a compact surface with geodesic boundary.
We can consider the intersection number on the double $D(\overline{\Sigma})$
of~$\overline{\Sigma}$, which is a closed surface $\overline{\Sigma}$ embeds into. This
intersection number on $D(\overline{\Sigma})$ is continuous
\cite[Proposition~4.5]{Bonahon86:EndsHyperbolicManifolds}.
Restricting this intersection number to~$\overline{\Sigma}$, we obtain continuity
of intersection number on $\GC_{\mathrm{open}}(S)$.
However, the conditions
of Theorem~\ref{thm:convex} alone are not enough to guarantee a
continuous extension $f \colon \GC_{\mathrm{open}}(S) \to \mathbb{R}$.
Indeed, let $\ell$ be the restriction to $\overline{\Sigma}$ of the hyperbolic
length on $D(\overline{\Sigma})$, and consider the modified curve functional $\ell'$
obtained by setting
\[
  \ell'(C) \coloneqq
  \begin{cases}
    0 & \text{$C$ is a boundary curve}\\
    \ell(C) & \text{otherwise.}
  \end{cases}
\]
We note that $\ell'$ satisfies additivity, stability and homogeneity
properties because $\ell$ does. Also, it satisfies smoothing because
$\ell$ does and non-boundary curves don't intersect boundary curves.
However, $\ell'$ does not extend to a continuous function  on
$\GC_{\mathrm{open}}(S)$: let $\gamma,\beta \in \pi_1(S,p)$ be elements based at a point $p \in
S$, and denote $C=[\beta]$, $D=[\gamma]$. Assume that $C$ is a
boundary curve and $D$ is not. For each $n$, define a non-simple,
non-boundary parallel curve by $C_n \coloneqq [\gamma \beta^n]$.
Observe that the sequence $\frac{1}{n}C_n$ converges to $C$ in the
weak$^*$ topology,
but
\begin{align*}
  \ell'(C_n) &\asymp n,\\
  \intertext{whereas}
  \ell'(C)&=0,
\end{align*}
so $\ell'$ can't be a continuous function on $\GC_{\mathrm{open}}(S)$.
So additional conditions on $f$ are needed to guarantee a continuous
extension to $\GC_{\mathrm{open}}(S)$.
\label{rmk:opensurface}
\end{remark}

%%% Local Variables:
%%% mode: latex
%%% TeX-master: "Smoothings"
%%% End:

\section{Convexity and continuity}
\label{sec:convexity}

\subsection{Convexity on the reals}

The
curve functionals~$f$ we study have some convexity property as a
function of the
weights, because of the convex union and homogeneity properties.
We first review some background on convex functions and their continuity properties.

A function $f\colon \mathbb{R}^n \to \mathbb{R}$ is called \emph{$\mathbb{R}$-convex} (resp.\ \emph{$\mathbb{Q}$-convex}) if
\[
f(ax + (1-a)y) \leq a f(x) + (1-a)f(y)
\]
for all $x,y\in \mathbb{R}^n$ and $a \in [0,1]$ (resp.\ $a \in [0,1] \cap \mathbb{Q}$).
A function $f \colon \mathbb{Q}^n \to \mathbb{R}$ might also be
$\mathbb{Q}$-convex, with the same definition. We furthermore say that
$f$ is \emph{midpoint-convex} if
\[
  f\bigl(\OneHalf x + \OneHalf y \bigr) \le \OneHalf f(x) + \OneHalf f(y).
\]

\begin{proposition}\label{prop:convexcont}
The following are true:
\begin{enumerate}[(i)]
    \item A midpoint-convex function $f \colon \mathbb{Q}^n \to \mathbb{R}$ is $\mathbb{Q}$-convex. \label{item:midpoint-rat-convex}
    \item An $\mathbb{R}$-convex function $f \colon \mathbb{R}^n \to
      \mathbb{R}$ is continuous. \label{item:realconvex-cont}
    \item A $\mathbb{Q}$-convex function $f \colon \mathbb{Q}^n \to
      \mathbb{R}$ is continuous. \label{item:ratconvex-cont}
    \item Every $\mathbb{Q}$-convex function
      $f \colon \mathbb{Q}^n \to \mathbb{R}$ has a unique continuous
      extension to an $\mathbb{R}$-convex function $\bar f \colon \R^n
      \to \R$. \label{item:ratconvex-extend}
\end{enumerate}
\end{proposition}

\begin{proof}
\begin{enumerate}[(i)]
   \item
This proof is due to
 Ivan Meir \cite{Meir19:MidpointConvex}, following \cite[Page~17]{HLP88:Inequalities}.
We first prove that midpoint inequality extends
     to arbitrary means:
\[
g((x_1+\dots+x_m)/m)\leq (g(x_1)+\dots+g(x_m))/m
\]
 for any $m\in\mathbb{Z}_{\geq1}$.
 We can prove this first for $m=2^k$ by using midpoint convexity repeatedly. 
For general $m\leq 2^{i}$, we take $x_1,\dots,x_m$ plus $2^i-m$ copies
of $x'=(x_1+\dots+x_m)/m$, yielding
\[
g(x')=g\left(\frac{(2^i-m)x'+x_1+\cdots+x_m}{2^i}\right) \le\frac{(2^i-m)g(x')+g(x_1)+\dots+g(x_m)}{2^i},
\]
 which implies $g(x')=g((x_1+\dots+x_m)/m) \leq
 (g(x_1)+\dots+g(x_m))/m$.
 
To prove $\Q$-convexity, taking $a$ copies of $x$ and $b$ copies of
$y$ we obtain
\[
g\left(\frac{ax+by}{a+b}\right)\leq \frac{ag(x)+bg(y)}{a+b}=\left(\frac{a}{a+b}\right)g(x)+\left(\frac{b}{a+b}\right)g(y)
\]
for $a,b\in \mathbb{Z}_{\geq0}$ not both zero.
    \item See \cite[Theorem~7.1.1]{Ku09}.
    \item The proof of \cite[Theorem~7.1.1]{Ku09} can be adapted for
      functions on $\mathbb{Q}^n$. The proof relies on
      Bernstein-Doetsch Theorem, which works in high generality for
      topological vector spaces (see
      \cite[Theorem~B]{KK89:BernsteinDoetsch}), and the fact that
      any point $x \in \mathbb{Q}^n$ is the interior of some
      full-dimensional $\mathbb{Q}$-simplex, on which $f$ is bounded.
    \item Define the extension by
      \[
        \bar{f}(x) \coloneqq \liminf_{\substack{y \to x\\ y \in \mathbb{Q}^n}} f(y).
      \]
      By continuity of $f$ on $\Q^n$, $\bar f$ is an extension of~$f$.
      To study $\bar{f}(a x + (1-a)y)$,
      let $x_i, y_i$ be sequences in $\mathbb{Q}^n$
      with $\lim x_i=x$,
      $\lim y_i=y$, $\liminf f(x_i) = f(x)$, and $\liminf f(y_i) = y$. Let
      $a_i \in [0,1] \cap \mathbb{Q}$ be a sequence with
      $\lim a_i = a$. Then
      \begin{align*}
    \bar{f}(a x + (1-a)y) &\leq \liminf f\bigl(a_i x_i + (1-a_i)y_i\bigr) \\
        &\leq \liminf \bigl(a_i f(x_i) + (1-a_i)f(y_i)\bigr)\\
        &= a\bar{f}(x) + (1-a)\bar{f}(y).
      \end{align*}
      Thus $\bar f$ is convex and therefore continuous.\qedhere
\end{enumerate}
\end{proof}

It is not true that all $\Q$-convex functions
$f \colon \mathbb{R}^n \to \mathbb{R}$ must be
continuous, but all counterexamples are highly pathological.  In
particular, any measurable $\Q$-convex function $f \colon \mathbb{R}^n \to \mathbb{R}$ is necessarily continuous
(see \cite[Theorem~9.4.2]{Ku09}).

\begin{remark} Proposition~\ref{prop:convexcont}\ref{item:realconvex-cont} does not hold for
  infinite dimensional
  topological vector spaces: an unbounded linear functional is convex
  but not continuous.
\end{remark}

\subsection{Convexity for curve functionals}

We now apply the results above to our setting of real-valued functions
on curves.

As an immediate consequence of Proposition~\ref{prop:convexcont}\ref{item:ratconvex-extend}, for
a curve functional~$f$ satisfying convex
union and homogeneity, we can extend $f$ to a weighted curve functional that is convex and therefore
continuous
for a fixed set of components. This will play a role in the proof of
Theorem~\ref{thm:convex}, specifically in
Proposition~\ref{prop:iteratescont}.

First, for any curve functional
satisfying homogeneity, we adopt the convention that we extend $f$ to
rationally-weighted curves
$\Q\Curves^+(S)$ in the usual way by clearing denominators: set
\begin{equation}
  f\Bigl(\sum a_i C_i\Bigr) \coloneqq \frac{1}{d}f\Bigl(\sum d a_i C_i\Bigr)
  \label{eq:rationalex}
\end{equation}
for some integer~$d$ sufficiently large so all the $d a_i$ are
integers.
By homogeneity of~$f$, the extension does not depend on~$d$.

\begin{proposition}
  Let $C = (C_i)_{i=1,\dots,n}$ be a finite sequence of multi-curves, and
  consider combinations $\sum_{i=1}^n a_i C_i$. Let $f$ be a
  curve functional that satisfies homogeneity and convex union.
  Define a function $f_{C}\colon \Q^n \to \R$ by
\[
  f_{C}(a_1,\dots,a_n)\coloneqq f\left( \sum_{i=1}^n a_i C_i \right).
\]
Then $f_C$ is $\Q$-convex and thus continuous.
\label{prop:convexweights}
\end{proposition}
\begin{proof}
  It is immediate from the definitions that $f_C$ is midpoint convex.
  The result follows from
  Proposition~\ref{prop:convexcont}.
\end{proof}

\begin{corollary}
  \label{cor:convex-extend}
  If a curve functional $f$ satisfies homogeneity and convex union, 
  then there is a unique continuous homogeneous extension of~$f$ to weighted curve functional.
\end{corollary}

\begin{proposition}
  If a curve functional satisfies 
  convex union, homogeneity, and quasi-smoothing, the extension from
  Corollary~\ref{cor:convex-extend} satisfies
  weighted quasi-smoothing with the same constant.
  \label{prop:weightextend}
\end{proposition}

\begin{proof}
We first observe that $f$, as a function on integrally weighted multi-curves,
satisfies weighted quasi-smoothing.
 If $C=[\gamma]$, $\gamma \reducesto \gamma'$ and
  $k$ is an integer, then $k\gamma \reducesto_k k \gamma'$, since $k\gamma$ are $k$ disjoint parallel copies of $\gamma$.
Thus,
\[
f(kC) \geq f(kC')-kR.
\]
By the method of clearing denominators and homogeneity, as in
Equation~(\ref{eq:rationalex}), we obtain rational weighted
quasi-smoothing.

Finally, by continuity of $f$ as a function of the weights of a fixed
multi-curve, we get real weighted quasi-smoothing.
\end{proof}

  Theorems~\ref{thm:convex} and~\ref{thm:stable} as stated start from
  a curve functional of various types. Many curve
  functionals naturally come as functions on weighted curves (see Section~\ref{sec:examples}).
  On the other hand, we have seen in Proposition~\ref{prop:weightextend} that a curve functional satisfying convex union and homogeneity and stability properties yields weighted curve functional satisfying the same properties.
  
%%% Local Variables:
%%% mode: latex
%%% TeX-master: "Smoothings"
%%% End:

\section{Examples}
\label{sec:examples}

We give several examples of curve functionals that extend to
functions on currents, mostly as a consequence of our main
theorems. This includes known results, such as hyperbolic lengths and
intersection numbers, or more generally lengths for any length metric
structure, as well as new results, such as extremal lengths with
respect to a conformal structure or with respect to a graph.
In the following applications we consider unoriented curves unless
otherwise stated.

\subsection{Intersection number}
\label{sec:intersection-number}
Fix a multi-curve~$D$ and consider the curve functional
$\Curves(S)$ defined by $f(C) = i(C, D)$, where $i(C,D)$ is the 
minimal number of intersection points between representatives of $C$
and~$D$ in general position.
Then $f$ is
homogeneous, additive, and stable.

There is a simple geometric
argument, that we will use repeatedly, to see that $f$ satisfies
smoothing. Fix a minimal representative $\delta$ for~$D$. Take a
curve~$C$ with an
essential self-intersection~$x$, and a representative~$\gamma$ with
minimal intersection with~$\delta$. Then $\gamma$ has
a self-intersection point~$x'$ of the homotopy type of~$x$. If we
consider the
curve representative~$\gamma' \in C'$ obtained by smoothing at~$x'$, then,
since $i(C',D)$ is an infimum, we have
\[
  i(C',D) \le i(\gamma',\delta) = i(\gamma,\delta) = i(C,D),
\]
as desired.

By Theorem~\ref{thm:convex}, intersection
number with~$D$ extends to a
continuous function on geodesic currents
\[
  i(\cdot, D) \colon \GC(S) \to \R.
\]
We can then fix $C$ and vary $D$ to show that, for $\mu$ a geodesic
current, $i(\cdot,\mu)$ is a continuous function on $\GC(S)$.
In \cite[Proposition~4.5]{Bonahon86:EndsHyperbolicManifolds}, Bonahon
shows that the geometric intersection number $i \colon
\mathbb{R}\Curves(S) \times \mathbb{R}\Curves(S) \to \mathbb{R}_{\geq
  0}$  between two weighted multi-curves extends to a continuous
two-variable function
\[i \colon \GC(S) \times \GC(S) \to \mathbb{R}_{\geq 0}.\]

\begin{question}\label{quest:intersect-cont}
  Can the arguments in this paper be extended to give an alternate
  proof that
  geometric intersection number is a continuous two-variable function?
\end{question}

Following Example~\ref{ex:selfintersection}, proving that there is a
continuous extension of
$\sqrt{i(C,C)}$ to a function on $\GC(S)$ is equivalent to proving
continuity of $i$ as a two-variable function, by a simple polarization
argument:
\[
 i(C,D) = \frac{i(C\cup D, C \cup D) - i(C,C) - i(D,D)}{2}.
\]

\subsection{Hyperbolic length}
\label{sec:hyperbolic-length}
We continue with the original motivating example for geodesic currents
\cite[Proposition~14]{Bonahon88:GeodesicCurrent}. Fix a hyperbolic metric~$g$ on~$S$, and denote the hyperbolic structure by $\Sigma$.
 Then, for any closed curve $C$ on~$S$
(not necessarily simple), we can consider its hyperbolic length with
respect to the Riemannian metric. In terms of the holonomy
representation $\rho_g: \pi_1(S) \to \PSL_2(\R)$, this is given by
\[
  \ell_g(C) = 2\cosh^{-1}\left(\OneHalf \tr(\rho(c))\right)
\]
where $c \in \pi_1(S)$ is a representative of $C$.
We extend $\ell_g$ to a
weighted curve functional by additivity and homogeneity:
\[
  \ell_g(t_1C_1 \cup \dots \cup t_nC_n) = \sum_{i=1}^n t_i\ell_g(C_i).
\]
By definition, $\ell_g$ is additive and homogeneous.
Stability follows from properties of the trace of $2\times 2$
matrices, or geometrically from the length. Smoothing follows by the
argument for intersection number.
Thus, by Theorem~\ref{thm:convex}, $\ell_g$ extends to a continuous
function on geodesic currents.

We recall that Bonahon shows that
\[
\ell_g(C)=i(\mathcal{L}_{\Sigma},C)
\]
where $\ell_g(C)$ denotes the hyperbolic length of $C$, i.e., the
length of the $g$-geodesic representative, and $\mathcal{L}_{\Sigma}$
denotes the \emph{Liouville  current}, a geodesic current induced by
the volume form on $UT\Sigma$. (For an equivalent formulation in terms
of Definition~\ref{def:currents-1}, see
\cite[\S~2]{Bonahon88:GeodesicCurrent}.)

\subsection{Length with respect to arbitrary metrics}
\label{sec:arbitrary-metrics}

The argument from Section~\ref{sec:hyperbolic-length} applies equally
well to show that for any Riemannian or, more
generally, length metric~$g$
on~$S$, 
length~$\ell_g$ with respect to~$g$ satisfies smoothing. For
completeness and later use, we prove
that these curve functionals
are stable. (Freedman-Hass-Scott give a proof in the Riemannian case
\cite[Lemma~1.3]{FHS82:ClosedGeodesics}.)
\begin{lemma}\label{lem:length-stable}
  For any orientable surface~$S$ and length metric~$g$ on~$S$, the
  curve functional $\ell_g$ is stable: $\ell_g(C^n) = n \ell_g(C)$.
\end{lemma}
\begin{proof}
  One inequality is true in any length space: by taking the obvious
  $n$-fold representative of $C^n$, we see that
  $\ell_g(C^n) \le n \ell_g(C)$. The other inequality follows from the
  smoothing property and Proposition~\ref{prop:crossings-powers}: since $C^n \reducesto n C$, we have
  $\ell_g(C^n) \ge n \ell_g(C)$.
\end{proof}

\begin{remark}
  Lemma~\ref{lem:length-stable} is false if $S$ is not orientable. For
  instance, if $S$ is the projective plane, $g$ is any metric, and $C$
  is the non-trivial curve on~$\Sigma$, then $\ell_g(C) > 0$ but $C^2$
  is null-homotopic so $\ell_g(C^2) = 0$. We can get a similar
  inequality without torsion on a Möbius strip, by removing a small
  disk from this projective plane.
\end{remark}

\subsection{Length with respect to embedded graphs}
\label{sec:length-emb-graph}
We can generalize further beyond length metrics. Let
$\iota \colon \Gamma \hookrightarrow S$ be an embedding of a finite
graph in~$S$ that
is filling, in the sense that the complementary regions are disks, or
equivalently
$\iota_*$ is surjective on~$\pi_1$. Endow $\Gamma$
with a length metric~$g$. Then any closed multi-curve~$C$ on~$S$ can
be homotoped so that it factors through~$\Gamma$, in many different
ways. Let $\ell_\Gamma(C)$ be the length of the smallest multi-curve $D$
on~$\Gamma$ so that $\iota(D)$ is homotopic to~$C$. It is easy to see
that this length is realized and is positive.
(In fact we can see $\ell_\Gamma$ as a limit of lengths with respect
to Riemannian metrics, by fixing an embedding of~$\Gamma$ and making
the metric on the complement of a regular neighborhood of~$\Gamma$ be
very large, following Shepard~\cite{Shepard91:TopologyShortest}.)

As before, $\ell_\Gamma$ is clearly additive and homogeneous, and is
stable by the argument of Lemma~\ref{lem:length-stable}. To see that
$\ell_\Gamma$ satisfies smoothing at an essential crossing, take a
minimal-length concrete representative $\delta$ of $D$ on~$\Gamma$.
Since the image $\iota\circ\delta$ has a corresponding crossing by
Lemma~\ref{lem:essential-cross} and $\iota$ is an embedding, there is
a corresponding crossing of $\delta$ that can be smoothed and then
tightened to get the desired inequality.

As a special case, we can consider
the case when $\Gamma$ is a \emph{rose graph} with only one
vertex~$\ast$ and edges of length~$1$.
Since $\iota$ is filling, the image of the edges of~$\Gamma$
give generators for $\pi_1(S,\iota(\ast))$. Then the length
$\ell_\Gamma(C)$ of a curve~$C$
is the length
of~$C$ as a conjugacy class in $\pi_1(S)$ with respect to
these generators. This is a simple generating set in the sense of
Erlandsson \cite{Erlandsson:WordLength}, who proved this continuity
and constructed an explicit
multi-curve~$K$ so that $\ell_\Gamma(C) = i(C,K)$.

\subsection{Stable lengths}
\label{sec:stable-examp}

Generalizing the previous example, let
$\iota \colon \Gamma \to S$ be an \emph{immersion} from a finite graph
to~$S$ so that $\iota_* \colon \pi_1(\Gamma) \to \pi_1(S)$
is surjective, and again give a length metric
on~$\Gamma$. For instance, if $\Gamma$ has a single vertex and all
edges have length~$1$, this is
equivalent to giving an arbitrary generating set for $\pi_1(S)$.
We can define $\ell_\Gamma(C)$ as before, as the minimum length of any
multi-curve~$D$ on~$\Gamma$ so that $\iota_*(D) = C$.

The curve functional $\ell_\Gamma$ is still additive, but unlike the previous
examples it is not stable (see Example \ref{ex:puncturedtorus} below).
Thus we cannot hope to extend
$\ell_\Gamma$ to currents, but rather extend the stable
curve functional~$\norm{\ell_\Gamma}$ (defined in Section \ref{sec:stable}).
We do have quasi-smoothing.

\begin{lemma}
  For any connected, $\pi_1$-surjective immersion of a length
  graph~$\iota \colon \Gamma \to S$, the curve functional~$\ell_\Gamma$
  satisfies quasi-smoothing.
\end{lemma}

This is a special case of a more general result. Let $V$ be a
connected length space, with a continuous, $\pi_1$-surjective map
$\iota \colon V \to S$. Then for $C$ a multi-curve on~$S$, a \emph{lift} of
$C$ is a multi-curve $\tilde{C}$ in~$V$ so that $\iota_* \tilde{C} = C$. (Both $\tilde{C}$
and~$C$ are defined up to homotopy.) Define
$\ell_{\iota,V}(C)$ to be the infimum, over all lifts $\tilde{C}$ of~$C$, of the
length of~$\tilde{C}$ in~$V$.

\begin{proposition}\label{prop:cocompact-quasismooth}
  Let $V$ be a connected, compact length space, with a continuous,
  $\pi_1$-surjective map $\iota \colon V \to S$. Then
  $\ell_{\iota,V} \colon \Curves(S) \to \R_{+}$ satisfies quasi-smoothing.
\end{proposition}

As a corollary of Proposition~\ref{prop:cocompact-quasismooth} and
Theorems~\ref{thm:stable} and~\ref{thm:convex}, $\norm{\ell_V}$
extends continuously to a function on $\GC(S)$. This continuous
extension was first proved by Bonahon
\cite[Proposition~10]{Bon91:NegativelyCurvedGroups} in the context of
a hyperbolic group acting discretely and cocompactly on a length space
(replacing $\pi_1(S)$ acting on~$\wt V$), with an additional technical
assumption that the space is uniquely geodesic at infinity. Later, Erlandsson, Parlier, and Souto
\cite[Theorem~1.5]{EPS:CountingCurves} lifted this assumption.

We remark that given a properly discontinuous action of $\pi_1(S)$ on
$X$ a CW-complex, we can construct a $\pi_1(S)$-surjective map $\iota
\colon X/\pi_1(S) \to S$. In fact, we will construct a
$\pi_1(S)$-equivariant map $\iota \colon X \to \wt{S}$. First, we
define $\iota \colon X_0 \to \tilde{S}$ by picking a value on each $\pi_1(S)$ orbit of
the 0-skeleton arbitrarily and extending equivariantly. Similarly on
the 1-skeleton~$X_1$, for each $\pi_1(S)$ orbit on $X_1$, pick a
path in $\wt{S}$ between the images of the endpoints. Continue the construction
inductively.
This construction works because $\pi_1(S)$ acts freely---since
$\pi_1(S)$ is torsion-free and acts properly discontinuously---and
$\wt{S}$ is contractible.

In this construction, proper discontinuity of the action is crucial. For example, it was shown by Bonahon in \cite[Proposition~11]{Bon91:NegativelyCurvedGroups} that if $W$ is a finite graph which is a deformation retract of $S$, the action of $\pi_1(S)$ on the universal cover $\wt{W}$ of $W$ is cocompact but not properly discontinuous, and translation length of conjugacy classes of $\pi_1(S)$ acting on $X$ does not extend continuously to $\GC(S)$.

\begin{proof}[Proof of Proposition \ref{prop:cocompact-quasismooth}]
  Let $\wt\iota \colon \wt V \to\wt S$ be the pull-back of $\iota$
  along the universal cover~$\pi_S$ of~$S$, part of the pull-back square
  \begin{equation}
  \mathcenter{\begin{tikzpicture}[x=2.25cm,y=1.25cm]
    \node (pullback) at (0,0) {$\wt V$};
    \node (X) at (1,0) {$V$};
    \node (Z) at (0,-1.5) {$\wt{S}$};
    \node (Y) at (1,-1.5) {$S$};
    \draw[->,dashed]  (pullback) -- (X) node[midway,above,cdlabel] {\pi_V};
    \draw[->,dashed] (pullback) -- (Z) node[midway,right,cdlabel] {\tilde{\iota}};
    \draw[->] (Z) -- (Y) node[midway,above,cdlabel] {\pi_S};
    \draw[->] (X) -- (Y) node[midway,right,cdlabel] {\iota};
  \end{tikzpicture}}
  \label{eq:pullback}
  \end{equation}
 where $\wt V=V \times_S \wt{S}$.
  Since
  $\iota$~is $\pi_1$-surjective, $\wt V$ is a connected covering space
  of~$V$. For
  $\wt x \in \wt S$, let $\diam_\iota(\wt x)$ be the diameter of
  $\wt\iota^{-1}(\wt x) \subset \wt V$. (Set $\diam_\iota(\wt x) = 0$
  if $\wt x$ is not in the image of~$\wt\iota$.) Set
  \[
    \diam_\iota(V) \coloneqq \sup_{\wt x \in \wt S} \diam_\iota(\wt x).
  \]
  We wish to see that $\diam_\iota(V)$ is finite. First,
  since $\iota$ is $\pi_1$-surjective, for every $\gamma \in
  \pi_1(S,x)$ there exists $\delta_{\gamma} \in \pi_1(V,\tilde x)$ so
  that $\iota_* \delta_\gamma = \gamma$. This implies
  \[
      \delta_{\gamma} \cdot \tilde{\iota}^{-1}(\tilde{x}) = \tilde{\iota}^{-1}(\gamma \cdot \tilde{x}).
    \]
    But $\delta_{\gamma}$ is a deck transformation, and thus acts as
    an isometry on $\wt{V}$,
    so
    $\diam(\tilde{\iota}^{-1}(\gamma \cdot
    \tilde{x}))=\diam(\tilde{\iota}^{-1}(\tilde{x}))$ and so we can define
   \begin{equation}\label{eq:diam-iota}
    \diam_\iota(x) = \diam_\iota(\wt x).
  \end{equation}
  for $x \in S$ and any lift $\wt x$ of $x$. (Note that
  $\diam_\iota(x)$ is not in general the diameter of $\iota^{-1}(x)$;
  rather than looking at the length of a shortest path connecting two
  points in $\iota^{-1}(x)$, we restrict to paths that map to
  null-homotopic loops.)
     
    \begin{lemma}\label{lem:diam-bounded}
      In the above setting, the diameter $\diam_\iota(x)$ is upper
      semi-continuous as a function of~$x$.
    \end{lemma}
    \begin{proof}
      For each $x_0 \in S$, consider an evenly covered neighborhood
      $U$ of $x_0$, and fix a lift $\tilde{x}_0 \in \wt S$.
        We want to show that for all sequences $\{ x_i \} \subset U$
        with $x_i \to x$ and for all $\epsilon>0$, there exists~$i_0$
        so that for all $i \geq i_0$
    \[
     \diam(\tilde{\iota}^{-1}(\tilde{x_i}))  < \diam(\tilde{\iota}^{-1}(\tilde{x})) + \epsilon.
    \]

    Now, $\tilde{\iota}$ is a pullback of a proper map, so it is a closed map:
    \begin{lemma} [{\cite[\href{https://stacks.math.columbia.edu/tag/005R}{Theorem 005R}]{stacks-project}}]
      Let $X$ be a metric space, and $f \colon X \to Y$ a proper map. For any continuous map $g \colon Z \to Y$, the pullback map $X \times_Y Z \to Z$ is closed.
    \end{lemma}

    By definition of diameter, and since the fibers $\tilde{\iota}^{-1}(\tilde{x})$ are compact for any $\tilde{x}_i$, we can find points $p_i,q_i \in \tilde{\iota}^{-1}(\tilde{x_i})$  so that $\diam(\tilde{\iota}^{-1}(\tilde{x_i}))=d(p_i,q_i)$. Also, by closedness of $\tilde{\iota}$, a subsequence of the $p_i$ and $q_i$ converges to points $p,q \in \tilde{\iota}^{-1}(\tilde{x})$.
    Furthermore, by continuity of distance, for any $\epsilon>0$, we have, for $i$ large enough,
    \[
    \diam(\tilde{\iota}^{-1}(\tilde{x_i}))=d(p_i,q_i)   <d(p,q) + \epsilon \leq \diam(\tilde{\iota}^{-1}(\tilde{x})) + \epsilon,
    \]
    finishing the proof of Lemma~\ref{lem:diam-bounded}.
  \end{proof}

  As a result of Lemma~\ref{lem:diam-bounded}, the function
  $\diam_\iota(x)$ is bounded on~$S$; let $R(V)$ be this global bound.
  
  We now finish the proof of
  Proposition~\ref{prop:cocompact-quasismooth}. We are given a curve
  $C$ with an essential crossing~$p$ and corresponding smoothing $C'$.
  Pick a concrete curve $\tilde\gamma$ on~$V$ that comes within
  $\epsilon$ of realizing $\ell_{\iota,V}(C)$; in particular,
  $\iota \circ \tilde \gamma$ represents~$C$. By
  Lemma~\ref{lem:essential-cross}, there are points
  $x,y \in X(\tilde\gamma)$ so that
  $\iota(\tilde\gamma(x)) = \iota(\tilde\gamma(y))$ is a crossing
  corresponding to~$p$. We wish to find another curve $\tilde\gamma'$
  on~$V$, with length not too much longer, so that
  $\iota \circ \tilde\gamma'$ represents $C'$. We can do this by
  cutting $\wt\gamma$ at $x$ and~$y$, yielding endpoints $x_1$, $x_2$
  and $y_1$, $y_2$, and reconnecting $x_1$ to~$y_2$ and $y_1$ to~$x_2$
  by paths in~$V$ that project to the identity in $\pi_1(S)$.

  But the maximal length of a path connecting any two points $x,y \in V$
  with $\iota(x) = \iota(y)$ that projects to a null-homotopic path is
  exactly $\diam_\iota(\iota(x))$. We can therefore construct a
  desired representative $\tilde\gamma'$ with
  \[
    \ell_V(\tilde\gamma') \le \ell_V(\tilde\gamma) + 2R(V)
      \le \ell_{\iota,V}(C) + \epsilon + 2R(V).
  \]
  Since $\epsilon$ was arbitrary, we have proved the result with
  quasi-smoothing constant $2R(V)$.
\end{proof}

We show  now an example of a curve functional that satisfies quasi-smoothing but not strict smoothing.
\begin{example}
\label{ex:puncturedtorus}
  Consider the torus with one puncture with fundamental group
  generated by the usual horizontal loop~$a$ and vertical loop~$b$, as in
  Figure~\ref{fig:traintrack}.
  (This does not strictly speaking fit in the context of closed
  surfaces considered in this paper, but we can embed this punctured
  torus in a larger surface without essential change.)
    \begin{figure}
\centering{
\scalebox{0.4}{\Huge{%% Creator: Inkscape inkscape 0.92.4, www.inkscape.org
%% Accompanies image file 'traintrack.pdf' (pdf, eps, ps)
\begingroup%
  \makeatletter%
  \providecommand\color[2][]{%
    \errmessage{(Inkscape) Color is used for the text in Inkscape, but the package 'color.sty' is not loaded}%
    \renewcommand\color[2][]{}%
  }%
  \providecommand\transparent[1]{%
    \errmessage{(Inkscape) Transparency is used (non-zero) for the text in Inkscape, but the package 'transparent.sty' is not loaded}%
    \renewcommand\transparent[1]{}%
  }%
  \providecommand\rotatebox[2]{#2}%
  \newcommand*\fsize{\dimexpr\f@size pt\relax}%
  \newcommand*\lineheight[1]{\fontsize{\fsize}{#1\fsize}\selectfont}%
  \ifx\svgwidth\undefined%
    \setlength{\unitlength}{382.55013053bp}%
    \ifx\svgscale\undefined%
      \relax%
    \else%
      \setlength{\unitlength}{\unitlength * \real{\svgscale}}%
    \fi%
  \else%
    \setlength{\unitlength}{\svgwidth}%
  \fi%
  \global\let\svgwidth\undefined%
  \global\let\svgscale\undefined%
  \makeatother%
  \begin{picture}(1,0.87259457)%
    \lineheight{1}%
    \setlength\tabcolsep{0pt}%
    \put(0,0){\includegraphics[width=\unitlength,page=1]{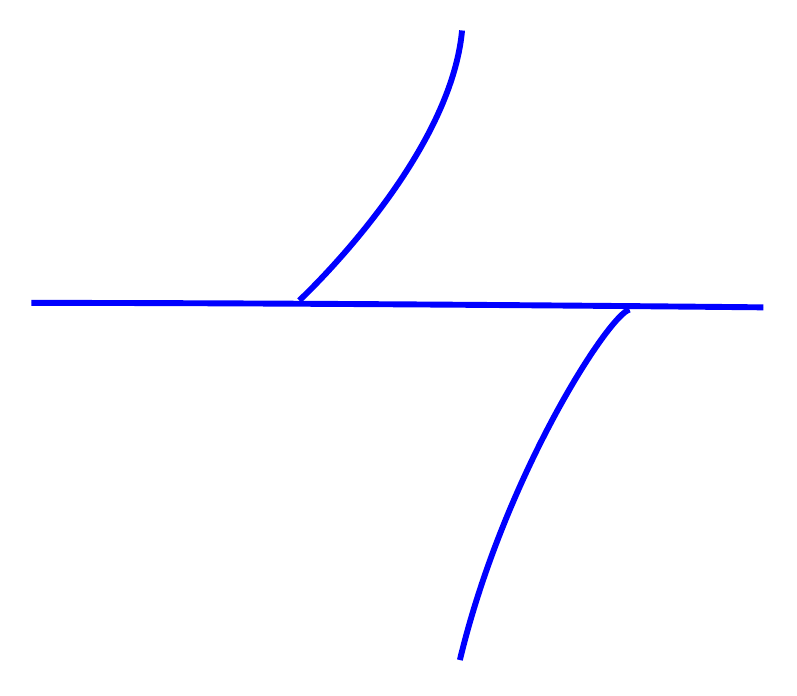}}%
    \put(0.12889388,0.41803715){\color[rgb]{0,0,0}\makebox(0,0)[lt]{\lineheight{1.25}\smash{\begin{tabular}[t]{l}$1$\end{tabular}}}}%
    \put(0.54788658,0.64657857){\color[rgb]{0,0,0}\makebox(0,0)[lt]{\lineheight{1.25}\smash{\begin{tabular}[t]{l}$x$\end{tabular}}}}%
    \put(0.42297294,0.41635672){\color[rgb]{0,0,0}\makebox(0,0)[lt]{\lineheight{1.25}\smash{\begin{tabular}[t]{l}$1-x$\end{tabular}}}}%
    \put(0.67810349,0.22708417){\color[rgb]{0,0,0}\makebox(0,0)[lt]{\lineheight{1.25}\smash{\begin{tabular}[t]{l}$x$\end{tabular}}}}%
    \put(0,0){\includegraphics[width=\unitlength,page=2]{traintrack.pdf}}%
    \put(0.25449158,0.76079829){\color[rgb]{0.7254902,0,0}\makebox(0,0)[lt]{\lineheight{1.25}\smash{\begin{tabular}[t]{l}$a$\end{tabular}}}}%
    \put(0.06104159,0.18407605){\color[rgb]{0,0.60392157,0}\makebox(0,0)[lt]{\lineheight{1.25}\smash{\begin{tabular}[t]{l}$b$\end{tabular}}}}%
  \end{picture}%
\endgroup%
}}
\caption{A punctured torus with, in blue, a train-track carrying a slice of measured laminations on the punctured torus depending on a parameter $x\in[0,1]$. In red, the parallel loop~$a$. In green, the meridian loop~$b$. }
\label{fig:traintrack}
}
\end{figure}
  Its fundamental group is the free group $F_2 = \langle a, b\rangle$.
  We will consider word-length $f$ with respect to the generating set
  $(a, a^2, b)$.
  Word-length satisfies quasi-smoothing and additive union, but not
  stability. (For instance, $f(a^2) = 1 \ne f(2a) = 2$.) The stable
  word-length $\| f\|$ satisfies stability and still satisfies quasi-smoothing,
  but it doesn't satisfy strict smoothing. We will show it behaves
  more erratically than word-length with respect to
  embedded generating sets.

  Consider for example the collection of weighted curves $C(x)$
  carried by the train-track in Figure~\ref{fig:traintrack}, with
  weights depending
  on a rational parameter $x \in [0,1]$. For instance, for $x = 2/5$,
  the curve is $1/5[aabab]$, with stable length $(1/5) \cdot 4$.
  If we plot the stable word-length of~$C(x)$ multiplied by the
  weight, we obtain the saw-tooth graph in Figure~\ref{fig:sawtooth}.
  We note the erratic behavior as a function of~$x$. In particular, it
  is far from convex.
  If $\|f\|$ satisfied the smoothing
  property, then it would be a convex function of the train-track
  weights, since if $w_1$ and $w_2$ are two rational weights on a
  train track~$T$,
  the weighted multi-curve $T(w_1) \cup T(w_2)$ can be smoothed
  to $T(w_1 + w_2)$
  (observed in \cite[Appendix~A]{Mirzakhani04:Thesis} and
  \cite[Section~3.2]{Thurston16:RubberBands}).
\end{example}

\begin{figure}
\centering{
\fontsize{12pt}{12pt}\selectfont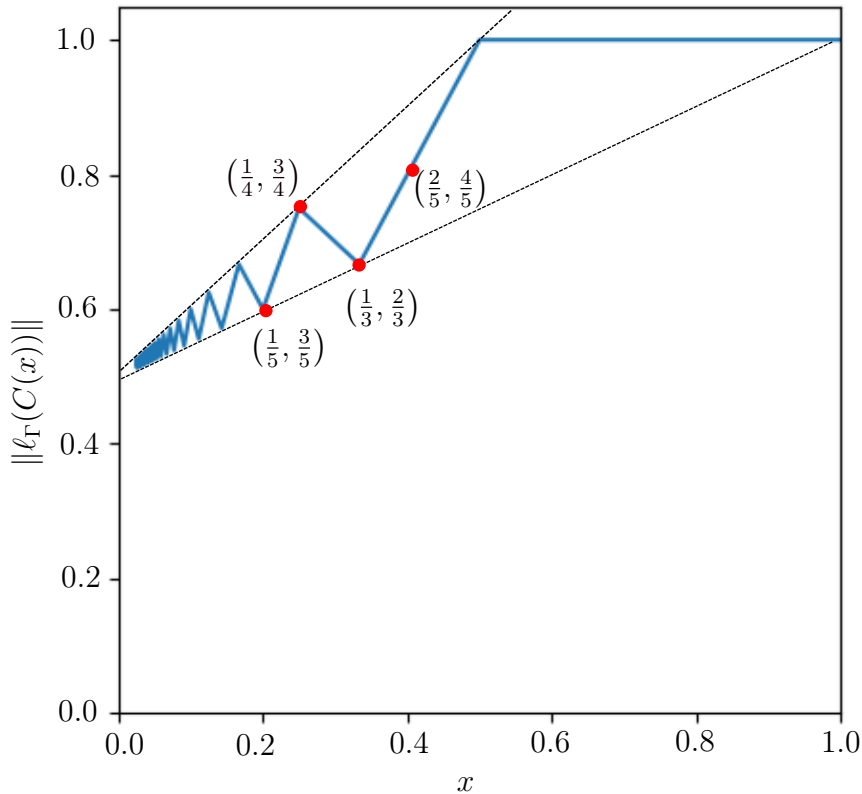
}
\caption{We consider the curve $C(x)$ carried by a train-track
  depending on a parameter $x$, and plot the stable word-length of $C(x)$
  as a function of~$x$. The graph has
  vertices at $\bigl(\frac{1}{2n+1},\frac{n+1}{2n+1}\bigr)$ and $\bigl(\frac{1}{2n},\frac{n+1}{2n}\bigr)$.}
\label{fig:sawtooth}
\end{figure}

\subsection{Asymmetric lengths}
\label{sec:asymmetric-lengths}

The arguments in Section~\ref{sec:stable-examp} apply equally well to
cases where distances may be zero or not symmetric. For instance, we
can take a directed graph $\Gamma$ with a non-negative length on each
edge, together with a map $\iota \colon \Gamma \to S$ so that the
corresponding cover $\wt \Gamma$ is strongly connected (every vertex
can be reached from any other vertex). The same arguments apply to
show that $\ell_\Gamma(\vec C)$ satisfies the \emph{oriented}
quasi-smoothing property and so its stabilization $\norm{\ell_\Gamma}$
extends to a continuous function $\GC^+(S) \to \R_{\ge 0}$.

One example would be to take a generating set for $\pi_1(S)$ as a
monoid. This corresponds to taking $\Gamma$ to be a graph with a
single vertex and one edge for each monoid generator.

\subsection{Generalized translation lengths from higher
  representations}
\label{sec:transl-lengths}

  Let $G$ be a real, connected, non-compact, semi-simple, linear Lie group. Let $K$ denote a maximal compact subgroup of $G$, 
so that $X=G/K$ is the Riemannian symmetric space of $G$. Let $[P]$ be the conjugacy class of
a parabolic subgroup $P \subset G$. 
Then there is a natural notion of $[P]$-Anosov representation $\rho\colon \pi_1(S) \to G$; see, for example, \cite[Section~4]{K18:GeometricStructures}.
When $\operatorname{rank}_{\mathbb{R}}(G)=1$ there is essentially one
class $[P]$, so we can simply refer to them as Anosov representations,
and they can be defined as those injective representations
$\rho\colon \pi_1(S) \to G$ where $\Gamma\coloneqq \rho(\pi_1(S))$
preserves and acts cocompactly on some nonempty convex
subset~$V$ of~$X$.

Rank~$1$ Anosov representations include two familiar examples.

\begin{enumerate}
\item Fuchsian representations into $G=\PSL(2,\mathbb{R})$. Here,
  $K=\SO(2)$ and $X=\mathbb{H}^2$. The convex set~$V$ in this
  case is the lift of the convex core of the hyperbolic surface.
\item  Quasi-Fuchsian representations of surface groups into
  $G=\PSL(2,\mathbb{C})$. In this case, $K=\SU(2)$, $X=\mathbb{H}^3$,
  and $V$ is the lift of the convex core of the hyperbolic
  quasi-Fuchsian manifold.
\end{enumerate}

In general, the conjugacy classes of parabolic subgroups of~$G$ correspond to
subsets~$\theta$ of the set of restricted simple roots $\Delta$ of $G$. 
For a given $[P]$-Anosov representation and each $\alpha \in \theta$,
Martone and Zhang
define~\cite[Definition~2.21]{MZ19:PositivelyRatioed}
a curve functional
\[
l_{\alpha}^{\rho}\colon \mathcal{C}(S) \to \mathbb{R}_{\geq 0}
\]
and show that for a certain subset of Anosov representations, these can be extended to geodesic currents as intersection numbers with some fixed geodesic current.

For the two rank 1 examples above, this length  $l_{\alpha}^{\rho}(C)$ corresponds to
hyperbolic length of the closed geodesic in the homotopy class $C$ in
the quotient hyperbolic manifold $\H^2/\rho(\pi_1(S))$ or $\H^3/\rho(\pi_1(S))$.

Bonahon showed the length in the Fuchsian case extends to geodesic
currents \cite[Proposition~14]{Bonahon88:GeodesicCurrent}. Length in
the quasi-Fuchsian case also extends to geodesic currents
\cite[Lemma~4.3]{BridgemanTaylor05:QCDef}.
Our techniques give another proof in this second case.

\begin{proposition} Translation length $l_{\rho}$, for $\rho\colon \pi_1(S) \to \PSL(2,\mathbb{C})$ a quasi-Fuchsian representation, extends to geodesic currents.
\end{proposition}
\begin{proof} Let $V$ to be the convex core of
  $\mathbb{H}^3/\rho(\pi_1(S))$. We have an obvious retract
  $r\colon V \to S$ (defined up to homotopy). Now use
  Proposition~\ref{prop:cocompact-quasismooth}, taking $\iota=r$.
\end{proof}

\begin{question}
  The \emph{complex} translation length of a quasi-Fuchsian
  representation also extends continuously to a function on
  geodesic currents
  \cite[Section~6]{BridgemanTaylor08:ExtensionWP}.
  Is there a version of our main
  theorem that would prove that a complex-valued curve functional like this
  extends to currents?
\end{question}

\begin{remark}
All the proposed generalizations of the definition of convex cocompact
representations for higher rank groups turn out to yield
 products of representations of rank~$1$ (see \cite[Theorem~1.3]{KL05:RigidityInvariantConvex} and \cite[Th\'eor\`eme]{Q05:GroupConvComp}) so it's not clear
that the approach using Proposition~\ref{prop:cocompact-quasismooth}
will allow one to extend curve functionals in higher
rank to geodesic currents.

 For $G = \PSL(3,\R)$, there is another cocompact action, not on a convex subset of the symmetric space, but on a convex subset of $G/P=\mathbb{R}P^2$.
 In this case, there is a natural metric on this convex subset, the Hilbert metric. One can easily show using this metric that smoothing is satisfied.
 In general, in higher rank one can construct similar cocompact actions on convex domains of $G/P$ (see \cite{GW12:AnosovDiscontinuity}), but there is not a known canonical choice of metric. 
Martone and Zhang show \cite[Theorem~2.1]{MZ19:PositivelyRatioed} that
some types of representations known as \emph{positively ratioed} can
be realized as intersection numbers with a distinguished geodesic
current.
This immediately implies this subclass of representations satisfy the smoothing property. It would be interesting to subsume their extension result under our scope. More specifically:

\begin{question}
Can we prove quasi-smoothing for the translation length for a subclass
of $[P]$-Anosov representations, as in
\cite[Definition~2.25]{MZ19:PositivelyRatioed}, directly from the
definition of translation length?
\end{question}
\end{remark}

\subsection{Extremal length}
\label{sec:extremal-length}
We now turn to curve functionals that satisfy only convex union and not
additive union, starting with the original motivation for this work,
extremal length.

\begin{definition}
  Fix $\Sigma$ a Riemann surface with a metric~$g$. Let $C = \bigcup t_iC_i$ be a
  weighted multi-curve on~$\Sigma$. For $\rho \colon \Sigma \to \R_{\ge 0}$ a
  measurable rescaling function, the \emph{area} of~$\rho$ is
  \[
    \Area(\rho g) \coloneqq \int_{x \in \Sigma} \rho(x)^2 \mu_g(x),
  \]
  where $\mu_g$ is the Lebesgue measure of~$g$.
  The \emph{length} of~$C$ is
  \[
    \ell_{\rho g}(C) \coloneqq \inf_{\gamma \in C} \sum_i t_i \int_{x\in\gamma_i} \rho(x)\,dx,
  \]
  where $dx$ is measured with respect to $g$ arc-length, and the
  infimum runs overall all representatives 
  $\gamma = \bigcup_i \gamma_i$ of~$C$, where $\gamma_i$ is a representative of the component $C_i$ of $C$.
  When $\rho$ is continuous, $\ell_\rho(C)$  is the length with respect to the
  metric~$g$ rescaled by~$\rho$. The square root of the \emph{extremal length} of~$C$ is
  \[
    \sqrt{\EL}(C) \coloneqq \sup_\rho \frac{\ell_{\rho
        g}(C)}{\sqrt{\Area(\rho g)}}.
  \]
  \label{def:EL}
\end{definition}

Observe that the supremand is unchanged under multiplying $\rho$ by a
positive constant. It is a standard result that the supremum is
realized by some generalized metric (not necessarily Riemannian) \cite[Theorem~12]{Rodin74:ExtremalLength} and that, when $C$ is a simple multi-curve, the optimum
metric $\rho g$ is the cone Euclidean metric associated to a quadratic
differential \cite{Jenkins57:OnExistenceOfExtremalLength}. Very little is known about the optimum metric
when $C$ is not simple, except in special cases
\cite{WZ94:Plumbing,Calabi96:ExtremalIsosystolic,HZ20:SwissCross,NZ19:Isosystolic}.

\begin{lemma}\label{lem:el-smooth}
  As a function of~$C$ with fixed~$\Sigma$, $\sqrt{\EL}$ satisfies homogeneity, stability, and smoothing.
\end{lemma}

\begin{proof}
  This follows since $\ell_{\rho g}$ satisfies these properties for
  each~$\rho$; here is the argument for smoothing.
  
  Let $C$ be a multi-curve with an essential crossing, and let $C'$ be
  the curve obtained by smoothing at the crossing. Then, for any
  scaling function~$\rho$,
  \[
    \frac{\ell_{\rho g}(C')}{\sqrt{\Area(\rho g)}} \le
    \frac{\ell_{\rho g}(C)}{\sqrt{\Area(\rho g)}}.
  \]
  Since $\sqrt{\EL}(C')$ and $\sqrt{\EL(C)}$ are the suprema of such
  terms, the result follows.
\end{proof}

Convex union is different, since on one side of the
inequality we have a sum of values of $\sqrt{\EL}$. (Extremal length
does not satisfy additivity.)

\begin{lemma}\label{lem:el-convex-union}
  $\sqrt{\EL}$ satisfies convex union.
\end{lemma}

\begin{proof}
  Fix a curve split as a union $C = C_1 \cup  C_2$, and let
  $\rho \colon \Sigma \to \R$ be the function realizing the supremum
  in the definition of extremal length for~$C$.
  Then
  \begin{align*}
    \sqrt{\EL}(C_1 \cup C_2)
      &= \frac{\ell_{\rho g}(C_1)}{\sqrt{\Area(\rho g)}}
        + \frac{\ell_{\rho g}(C_2)}{\sqrt{\Area(\rho g)}}\\
    &\le \sqrt{\EL}(C_1) + \sqrt{\EL}(C_2),
  \end{align*}
 where the last inequality holds by the supremum in the definition of $\EL$.
\end{proof}

Thus, by Theorem~\ref{thm:convex}, $\sqrt{\EL}$ (and $\EL$) extend
uniquely to continuous functions
on geodesic currents. 
With this extension,
we propose the following conjecture.

\begin{conjecture}For some universal constant $C$,
\[\EL_{\Sigma}(\mathcal{L}_{\Sigma})=C\Area(\Sigma).\]

where $\mathcal{L}_{\Sigma}$ is the Liouville current (compare Subsection~\ref{sec:hyperbolic-length}).
\end{conjecture}

\begin{remark}
  For the extremal length \emph{without} the square root, we instead have
  inequalities
  \[
    \EL(C_1) + \EL(C_2) \le \EL(C_1 \cup C_2) \le 2(\EL(C_1) + \EL(C_2)).
  \]
  The second inequality is a simple consequence of
  Lemma~\ref{lem:el-convex-union}. To see the first inequality, take
  optimal rescaling functions~$\rho_i$ for $\EL(C_i)$, normalized so
  that $\ell_{\rho_i g}(C_i) = \Area(\rho_i g) = \EL(C_i)$. Then using
  $\rho_1 + \rho_2$ as the test function for $\EL(C_1 \cup C_2)$ gives
  the desired inequality after elementary manipulations.
\end{remark}

\subsection{Extremal length with respect to elastic graphs}
\label{sec:el-graph}

There is a parallel notion of extremal length with respect to elastic
graphs \cite{Thurston19:Elastic}, just as there is for ordinary
lengths (Section~\ref{sec:length-emb-graph}).

An \emph{elastic graph}  $(\Gamma,\alpha)$ is a 1-dimensional CW
complex~$\Gamma$ (i.e., allowing multiple edges and loops) together
with an assignment of
positive real numbers $\alpha(e)$ for each $e \in
\operatorname{Edge}(\Gamma)$, where the \emph{edges} are the
1-dimensional cells of~$\Gamma$.

By a concrete \emph{multi-curve $\gamma$ on $\Gamma$} we mean a
1-manifold $X(\gamma)$ and a PL map $\gamma\colon X(\gamma) \to
\Gamma$. 
Given a scaling function $\rho\colon \operatorname{Edge}(\Gamma) \to
\mathbb{R}_{\geq 0}$, the length metric $\rho \alpha$ on $\Gamma$
gives edge $e$ the length $\rho(e)\alpha(e)$.
We define the \emph{length} of $\gamma$ as
\[
 \ell_{\rho \alpha}(\gamma) \coloneqq \sum_{e \in \operatorname{Edge}(\Gamma)} n_\gamma(e)\rho(e)\alpha(e),
\]
where $n_\gamma(e)$ is the weighted number of times that $\gamma$ runs over~$e$.
We can likewise define the length of a multi-curve~$D$ on $\Gamma$ as
the infimum over of concrete multi-curves in~$D$.

The \emph{area} of $\Gamma$ with respect to $\rho \alpha$ is defined to  be
\[
 \Area_{\rho}(\Gamma,\alpha) \coloneqq \sum_{e \in \operatorname{Edge}(\Gamma)} \rho(e)^2 \alpha(e).
\]
Intuitively, each edge is turned into a rectangle of width $\rho(e)$,
aspect ratio $\alpha(e)$, and thus area $\rho(e)^2\alpha(e)$.

As for extremal length for surfaces, we define the square root of extremal
length of a multi-curve on~$\Gamma$ by
\begin{equation}\label{eq:el-graph}
\sqrt{\EL}(D; \Gamma,\alpha) \coloneqq \sup_{ \rho : \operatorname{Edge}(\Gamma) \to \mathbb{R}_{\geq 0}} \frac{\ell_{\rho \alpha}(D)}{\sqrt{\Area_{\rho}(\Gamma,\alpha)}}.
\end{equation}
It is easy to do this optimization. We get a more interesting quantity
by incorporating a filling
embedding $\iota \colon \Gamma \hookrightarrow S$ of $\Gamma$ in a
surface~$S$, i.e., an embedding $\iota$ that is $\pi_1$-surjective. Then, for a
multi-curve~$C$ on~$S$ and scaling~$\rho$, the length is defined as in
Section~\ref{sec:length-emb-graph}:
\[
  \ell_{\rho\alpha;\iota}(C) \coloneqq \inf_{\substack{D\textrm{ on }\Gamma\\\iota_*(D) = C}} \ell_{\rho\alpha}(D).
\]
(The ``filling'' condition guarantees that there are such
multi-curves~$D$ with $\iota_*(D) = C$.) We can then define a version of
extremal length, following Equation~\eqref{eq:el-graph}, with respect
to~$\iota$:
\begin{equation}\label{eq:el-graph-iota}
  \sqrt{\EL}(C; \Gamma,\alpha,\iota) \coloneqq
  \sup_{ \rho : \operatorname{Edge}(\Gamma) \to \mathbb{R}_{\geq 0}} \frac{\ell_{\rho \alpha;\iota}(D)}{\sqrt{\Area_{\rho}(\Gamma,\alpha)}}.
\end{equation}

\begin{proposition}\label{prop:EL-embed}
  For $\iota \colon \Gamma \to \Sigma$ a filling embedding,
  $\sqrt{\EL}(C; \Gamma,\alpha,\iota)$ satisfies the convex
  union, stability, homogeneity, and smoothing properties.
\end{proposition}

\begin{proof}
  As in Lemma~\ref{lem:el-smooth},
  smoothing, stability, and homogeneity follow because
  $\ell_{\rho \alpha}(C)$ satisfies them for any~$\rho$. Convex union
  follows as in Lemma~\ref{lem:el-convex-union}.
\end{proof}

We can also consider extremal length with respect to an immersion
$\iota$ (rather than an embedding), defined in the same way.
\begin{proposition}\label{prop:EL-immerse}
  For $\iota \colon \Gamma \to \Sigma$ a $\pi_1$-surjective immersion,
  $\sqrt{\EL}(C; \Gamma, \alpha, \iota)$ satisfies the convex union,
  homogeneity, and quasi-smoothing properties.
\end{proposition}
\begin{proof}
  Convex union and homogeneity still hold by the same argument. We
  need an extra argument for quasi-smoothing. Instead of taking the
  supremum over all~$\rho$, rewrite Equation~\eqref{eq:el-graph-iota} as
  \[
    \sqrt{\EL}(C; \Gamma, \alpha, \iota) =
    \sup_{\substack{\rho : \operatorname{Edge}(\Gamma) \to \R_{\ge 0}\\
        \operatorname{Area}_\rho(\Gamma,\alpha) = 1}} \ell_{\rho\alpha}(C).
  \]
  The immersed graph $\iota(\Gamma)$ has finitely many
  self-intersections. For each self-intersection~$x$,
  take the supremum over the compact set of metrics
  $\{\rho \mid \Area_\rho(\Gamma,\alpha) = 1\}$ of the diameter
  $\diam_\iota(x)$ defined in Equation~\eqref{eq:diam-iota}. By the
  replacement argument in Proposition~\ref{prop:cocompact-quasismooth}, all
  $\ell_{\rho\alpha}$ for $\rho$ in this set satisfy quasi-smoothing with a
  uniform quasi-smoothing constant. It follows that
  $\EL(C; \Gamma,\alpha,\iota)$ also satisfies quasi-smoothing with
  the same constant.
\end{proof}

\begin{remark}
  We can relate extremal length for elastic graphs and surfaces by
  making a choice of a ribbon structure to $\Gamma$. Given any
  $\epsilon>0$, a ribbon elastic graph $G$ can be thickened into a
  conformal surface with boundary $N_{\epsilon}(G)$ by replacing each
  edge $e$ of $G$ by a rectangle of size $\alpha(e) \times \epsilon$
  and gluing the rectangles at the vertices by using the given ribbon
  structure. There are then inequalities relating
  $\EL(C; \Gamma,\alpha)$ and $\epsilon \EL(C;N_{\epsilon}(G))$, to
  within a multiplicative factor \cite[Props.\ 4.8 and~4.9]{Thurston20:Characterize}.
  For graphs immersed or embedded in a surface, the situation is less
  clear. By suitably choosing the elastic weights on an embedded
  graph, it appears that one can approximate extremal length well;
  Palmer gives one approach
  \cite{Palmer15:harmonic-thesis}. We are not aware of precise
  theorems.
\end{remark}

\subsection{$p$-extremal-length with respect to immersed graphs}
Extremal length fits into a family of energies for graphs
\cite[Appendix~A]{Thurston19:Elastic}. For $\Gamma$ a metric graph
with metric~$g$, a constant~$p$ with $1 \le p \le \infty$, and $C$ a curve on~$\Gamma$,
define
\begin{equation}\label{eq:Ep-graph}
E_p(C;\Gamma,g) \coloneqq\! \sup_{\sigma:\operatorname{Edge}(\Gamma) \to \R_{\ge 0}} \frac{\ell(C; \sigma g)}{\norm{\sigma}_p}
\end{equation}
where the $L^p$ norm $\norm{\sigma}_p$ is taken with respect to the metric~$g$. As
in the previous section, we can also consider a $\pi_1$-surjective immersion
$\iota \colon \Gamma \to S$, and consider $C$ to be a curve on~$S$ rather
than on~$\Gamma$.

For $p=\infty$, $E_\infty(C)$ is in fact just the length with respect
to~$g$ (as in Section~\ref{sec:stable-examp}).
Indeed, let $\sigma$ be any scaling factor, and let $\gamma$ be the
shortest representative of~$C$ on~$\Gamma$ with respect to~$g$
(not with respect to~$\sigma g$). Then
\[
\ell(C; \sigma g)
  \le \ell(\gamma; \sigma g)
  \le \norm{\sigma}_\infty \ell(\gamma; g)
\]
from which the result easily follows.

\begin{proposition}
  For any $\pi_1$-surjective immersion $\iota \colon \Gamma \to S$,
  the curve functional
  $E_p(\cdot;\allowbreak \Gamma, g, \iota)$ satisfies convex union, homogeneity, and
  quasi-smoothing, and thus its stabilization extends continuously to
  a function on geodesic currents. If $\iota$ is a filling embedding,
  then $E_p$ in addition satisfies stability and smoothing.
\end{proposition}
\begin{proof}
  This follows as in Propositions~\ref{prop:EL-embed}
  and~\ref{prop:EL-immerse}. To prove quasi-smoothing in the immersed
  case, we restrict to those functions $\sigma$ on
  $\operatorname{Edge}(\Gamma)$ where $\norm{\sigma}_p = 1$; as in the
  proof of Proposition~\pageref{prop:EL-immerse}, this set is
  compact.
\end{proof}

%%% Local Variables:
%%% mode: latex
%%% TeX-master: "Smoothings"
%%% End:

\section{Counting problems}\label{sec:counting}

One direct application of Theorem~\ref{thm:convex} is to obtain new
counting results for curves on surfaces of a given topological type.

A \emph{filling current} is a geodesic current
$\alpha \in \mathcal{G}\mathcal{C}(\Sigma)$ so that $i(\alpha,\mu)>0$
for all $\mu \in \mathcal{G}\mathcal{C}(\Sigma)\backslash \{ 0 \}$.
One example is a filling multi-curve, one whose complement in $S$
consists of disks.

Rafi and Souto proved the following.

\begin{definition}
A function $f$ on currents is \emph{positive} if
$f(\mu) > 0$ for all $\mu \ne 0$.
\label{def:positive}
\end{definition}

For a fixed continuous, homogeneous, and positive function $f\colon
\GC \to \R_{+}$, $\alpha$ a current, and $L$ a positive real number, let 
\[N(f,\alpha,L) \coloneqq \# \{ \phi \in \MCG \mid f(\phi(\alpha)) \leq L\}.\]

\begin{theorem}[Rafi-Souto {\cite[Main~Theorem]{RS19:CountingProblems}}]
For a fixed continuous, homogeneous, and positive function $f\colon
\GC \to \R_{+}$ and for a fixed filling current~$\alpha \in
\mathcal{G}\mathcal{C}(\Sigma)$, the
limit
\[
\lim_{L \to \infty} \frac{N(f,\alpha,L)}{L^{6g-6}}
\]
exists and is equal to
\[ 
\frac{m(f)m(\alpha)}{\mathfrak{m}_g}
\]
where $m(f)$, $m(\alpha)$, and $\mathfrak{m}_g$ are constants
depending only on $f$, $\alpha$, and the genus~$g$ respectively:
\begin{align}
  m(f)&=\mu_{\mathrm{Thu}}(\{\lambda \in \ML \mid f(\lambda) \leq 1 \})\label{eq:m-f}\\
  m(\alpha)&=\mu_{\mathrm{Thu}}(\{\lambda \in \ML \mid i(\alpha,\lambda) \leq 1 \})\\
  \mathfrak{m}_g &= \int_{\mathcal{M}_g}\!m(Y)\,d\omega_{\mathrm{WP}}(Y).
\end{align}
Here $\mu_{\mathrm{Thu}}$ is the Thurston volume measure on $\ML$
induced by the symplectic pairing on $\ML$.%
\footnote{There are at least two natural normalizations for
  volume on measured foliations; one coming from the Lebesgue measure on train-track charts, the other
  induced from the symplectic structure. We use the former here.}
\label{thm:RafiSouto17}
\end{theorem}

Dumas communicates a proof of the following theorem (attributed to
Mirzakhani).

\begin{theorem}[Dumas-Mirzakhani {\cite[Theorem~5.10]{Dumas15:SkinningMaps}}]
The function $\Lambda\colon \mathcal{M}_g \to \mathbb{R}_{\geq 0}$ given by
$\Sigma \mapsto m(\operatorname{EL}_{\Sigma})$ is constant, where
$m(\EL_\Sigma)$ is defined by Eq.~(\ref{eq:m-f}).
\label{thm:DumasMirzakhani15}
\end{theorem}

Let $\sqrt{\operatorname{EL}_{\Sigma}}\colon \GC \to \R_{\geq 0}$ be the continuous extension of square-root of extremal length to currents provided by Theorem~\ref{thm:convex}.
In order to be able to apply Theorem~\ref{thm:RafiSouto17}, it remains
to check that square root of extremal length is non-zero.

\begin{lemma}
For any $\Sigma \in \Teich(S)$, the functional
$\sqrt{\operatorname{EL}_{\Sigma}}$ is positive on $\GC(\Sigma)$.
\label{lem:ELpos}
\end{lemma}
\begin{proof}
Let $A = \sqrt{-2\pi\chi(S)}$. By Definition~\ref{def:EL} applied to
the hyperbolic metric,
$\ell_{\Sigma}(C)/A \leq
\sqrt{\operatorname{EL}_{\Sigma}}(C)$ for all curves~$C$ and all $\Sigma \in \Teich(S)$.
For any $\mu \in \GC(\Sigma)$, there exists a sequence $(\lambda_i C_i)_{i \in \mathbb{N}}$ of weighted curves so that $\lambda_i C_i \to \mu$ in the weak$^*$ sense.
It thus follows that $\ell_{\Sigma}(\mu)/A\leq\sqrt{\EL_\Sigma}(\mu)$. Since $\ell_{\Sigma}$ is a positive
function on currents, $\sqrt{\EL_\Sigma}$ is as well.
\end{proof}

We thus get solutions to counting problems for extremal length.
\begin{corollary}
For any filling current $\alpha$ and ${\Sigma} \in \Teich(S)$, the limit
 \[
\lim_{L \to \infty} \frac{N(\sqrt{\EL_{\Sigma}},\alpha,L)}{L^{6g-6}}
\]
exists, is independent of~$\Sigma$, and is equal to
\begin{equation}\label{eq:counting-formula}
\frac{\Lambda \cdot m(\alpha)}{\mathfrak{m}_g}.
\end{equation}
\label{cor:extremallengthcounting}
\end{corollary}
\begin{proof}
Using Theorem~\ref{thm:convex} and the results in Subsection~\ref{sec:extremal-length}, we can extend $\sqrt{EL_X}$ as a continuous, real-homogeneous functional on geodesic currents.
Furthermore, $\sqrt{EL_X}$ is positive on currents, by Proposition~\ref{lem:ELpos}.
Thus, by Theorem~\ref{thm:RafiSouto17}, the result follows.
Independence of~$\Sigma$ follows from
Theorem~\ref{thm:DumasMirzakhani15}.
\end{proof}

A similar counting result is true for $\alpha$ a simple multi-curve by
Mirzakhani's work, but one has to count slightly differently.
For $f$ a curve functional, $\alpha$ a simple curve, and $L > 0$, set
\[n(f,\alpha,L) \coloneqq \# \{ \phi(\alpha)\in \MCG \mid f(\phi(\alpha)) \leq L\}.\]
In general, note that
$n(f,\alpha,L) \neq N(f,\alpha,L)$; in fact, $N$ will be infinite if
$\alpha$ is not filling. Even for $\alpha$ filling, $N$ will be bigger
than $n$ if $\alpha$ has non-trivial stabilizer in the  mapping class
group.

We state the corresponding result for simple multi-curves.

\begin{proposition}
\label{prop:fsimple}
For any simple multi-curve $\alpha$, there is a constant $c(\alpha)$
so that, for any
$f \colon \GC^+(S) \to \mathbb{R}$ continuous, positive and
real-homogeneous function,
the limit
\begin{equation}
\lim_{L \to \infty} \frac{n(f,\alpha,L)}{L^{6g-6}}
\end{equation}
exists and is equal to
\begin{equation}
\frac{m(f)c(\alpha)}{\mathfrak{m}_g}.
\end{equation}
\end{proposition}

\begin{proof}
First note that the set $\{
\lambda \in \ML(\Sigma) \mid
f(\lambda) \leq 1 \}$ is
compact because
$f$ is positive on non-zero measured laminations.
Let $A \coloneqq \{ \lambda \in \ML \mid f(\lambda) = 1
\})$;
then $\mu_{\mathrm{Thu}}(A)=0$, as proved by Rafi and Souto
\cite[p.~879]{RS19:CountingProblems}.
Finally, we apply Mirzakhani's counting result
\cite[Theorem~1.3]{Mirzakhani08:SimpleClosed} and the Portmanteau
theorem
(see \cite[Theorem~30.12]{Bau01:Measure}) to
conclude the limit exists.
\end{proof}

\begin{corollary}
\label{cor:ELsimple}
For any simple multi-curve $\alpha$ and ${\Sigma} \in \Teich(S)$, the limit
\[
\lim_{L \to \infty} \frac{n(\sqrt{\EL_{\Sigma}},\alpha,L)}{L^{6g-6}}
\]
exists, is independent of $\Sigma$, and is equal to
\begin{equation}
\frac{\Lambda \cdot c(\alpha)}{\mathfrak{m}_g}.
\end{equation}
\end{corollary}

\begin{remark}
The constant $c(\alpha)$ in Proposition~\ref{cor:ELsimple} is not the
same as the constant $m(\alpha)$ in Theorem~\ref{thm:RafiSouto17}. For
details on how $c(\alpha)$ is defined, see
\cite[Equation~(1.2)]{Mirzakhani08:SimpleClosed}.
This is related to the fact that in Proposition~\ref{cor:ELsimple} we
count multi-curves instead of mapping classes because the stabilizer of a simple multi-curve under the mapping class group
is infinite.

The counting problem  $n(\ell_{\Sigma},\alpha,L)$  with $\alpha$ an
arbitrary essential multi-curve (possibly neither simple
nor filling) and $f=\ell_{\Sigma}$ a hyperbolic length, is established
in more recent work of Mirzakhani
\cite[Theorem~1.1]{Mir16:Counting}. Relying on her work,
Erlandsson-Parlier-Souto \cite[Theorem~1.6]{EPS:CountingCurves} give
the corresponding  result where $f$ is allowed to be intersection
number with other filling currents (not just a hyperbolic Liouville current). From Mirzakhani's work and work of
Erlandsson-Souto \cite[Corollary~4.4]{ES16:Counting}, one also can get the corresponding
counting problems where $\alpha$ is allowed to be a current (not just
a multi-curve). In fact, from Erlandsson-Souto's work one can also
allow $f$ to be any continuous, positive and real-homogeneous function
on currents, although they don't explicitly state this in their paper.
 This is done in Rafi-Souto's work (Theorem~\ref{thm:RafiSouto17}
 above) which also 
gives the expression~\eqref{eq:counting-formula} for the limit of the
counting problem. Rafi-Souto also
relies on
Mirzakhani's work \cite{Mir16:Counting}.

Finally, recently
Erlandsson-Souto~\cite[Theorem~8.1]{ES20:GeodesicCount} have given an
independent proof of the counting argument in \cite{Mir16:Counting}
illuminating the connection between counting problems for
simple and non-simple multi-curves.
Corollary \ref{cor:extremallengthcounting} also follows from 
Proposition \ref{cor:ELsimple} and forthcoming work of Erlandsson and
Souto~\cite{ES19:MirzakhaniCurveCount,ES20:GeodesicCount}.

We remark that Erlandsson-Souto's work
shows that if one knows a counting result for simple closed curves,
then one can obtain a counting result for non-simple closed curves
(for curve functionals extending continuously to currents).
The connection between these two types of counting problems is perhaps
that, in
some sense, the simple closed curves are the extremal points of the
space of currents, in the sense of convex sets
\cite{Rockafellar70:ConvexAnalysis}. For instance, the systole of
positive curve functional
satisfying smoothing and convex union is always a simple curve.
(Here by \emph{systole} 
we mean a weight~$1$, non-trivial multi-curve~$C$ with a minimal
value of $f(C)$.)
\end{remark}

%%% Local Variables:
%%% mode: latex
%%% TeX-master: "Smoothings"
%%% End:

\section{Proof outline}
\label{sec:proof}

  In this section, we prove the core theorem of the paper, Theorem~\ref{thm:weightconvex}, giving a continuous extension to geodesic currents of a functional~$f$ on weighted multi-curves satisfying convex union, stability, homogeneity and weighted quasi-smoothing.
  
\begin{theorem}\label{thm:weightconvex}
   Let $f$ be a weighted curve
   functional defined on weighted oriented multi-curves satisfying the
   weighted quasi-smoothing, convex union, stability, and homogeneity
  properties. Then there is a unique continuous homogeneous function
  $\bar{f} \colon \GC^+(S) \to \R_{\ge 0}$ that extends~$f$.
\end{theorem}

\begin{proof}[Proof of Theorem~\ref{thm:weightconvex}]
  The proof proceeds by studying the geodesic flow on the unit tangent
  bundle to~$S$ (with respect to an arbitrary hyperbolic metric),
  picking a suitable global cross-section with boundary~$\tau$,
  and looking at a ``smeared first return map'' to~$\tau$.
  
The proof breaks up into the following steps.
\begin{itemize}
\item \textbf{Step 1:} In Section \ref{sec:extension}, we define the
  cross-sections we consider, though we delay proving existence. We
  introduce \emph{bump functions}
  and the associated
  \emph{smeared first return map} (Def.~\ref{def:smeared-return}); indeed,
  there are several varieties of return maps
  (Table~\ref{tab:return-notation}). The main advantage of smeared
  return maps is that they are continuous (Prop.~\ref{prop:smeared-return-cont}).
  We use these smeared returns to define our
  purported extension~$f_\tau$ to geodesic currents as a
  limit (Def.~\ref{def:extension}), assuming a suitable global
  cross-section~$\tau$ exists.
    \item \textbf{Step 2:} In Section~\ref{sec:cross-section}, we find
      a suitable ``good''~$\tau$, defined in
      Definition~\ref{def:goodcross}, by considering certain
      \emph{wedge subsets} (Def.~\ref{def:wedge})
      of $UT\Sigma$,
      based at a collection of geodesic sub-segments of a closed
      geodesic~$\delta$ (see Prop.~\ref{prop:transverse-approx}).
    \item \textbf{Step 3:} In Section \ref{sec:join}, Proposition
      \ref{prop:limitexists}, we show the limit defining $f_\tau$ exists, by using convex union, quasi-smoothing, stability and homogeneity assumptions on $f$, and applying a version of  Fekete's Lemma (Lemma~\ref{lem:fekete}).
    \item \textbf{Step 4:} In Section \ref{sec:continuity},
      Proposition \ref{prop:continuity}, we show that $f_{\tau}$ is a
      continuous function on the space of oriented geodesic currents.
      We do this by showing that the iterates $f^k_{\tau}$ used in the
      definition of $f_{\tau}$ are continuous for each $k$, using the
      continuity of
      the smeared homotopy return map from Step 1, and $f$
      is a continuous function of the weights of a fixed multi-curve
      (Proposition~\ref{cor:convex-extend}).
      Propositions~\ref{prop:subadditivity}
      and~\ref{prop:superadditivity} give enough control to
      ensure uniform convergence of the iterates $f^k_{\tau}$ to
      $f_{\tau}$~(see Lemma~\ref{lem:uniformapprox}), and thus
      continuity of $f_{\tau}$ follows.
    \item \textbf{Step 5:} In Section \ref{sec:extends}, Proposition
      \ref{prop:extends}  we show that $f_{\tau}$ extends $f$ for
      oriented curves, by analyzing the image of the smeared return map in that case, and mixing-and-matching the components of the multi-curve.
\end{itemize}

The definition of $f_{\tau}$ depends on many choices: the
hyperbolic metric on~$S$, a choice of global
cross-section~$\tau$, and in fact
nested cross-sections $\tau_0 \subset \tau \subset \tau'$ and a choice
of bump function $\psi$ on~$\tau$ (see
Definition~\ref{sec:smeared-return}).
Different choices yield, a priori, different extensions $f_{\tau}$.
But we have proved that $f_\tau$ is a continuous function on the
space of geodesic currents, and moreover it restricts to $f$ on
multi-curves. Since weighted multi-curves are a dense subset
of the space of geodesic
currents (see Section~\ref{subsec:topologycurrents}), the extension
doesn't depend on these choices.
This proves Theorem \ref{thm:weightconvex}.
\end{proof}

Now, we prove Theorem~\ref{thm:convex}.

\begin{proof}[Proof of Theorem~\ref{thm:convex}]
By Proposition~\ref{prop:weightextend}, $f$ extends uniquely to a weighted curve functional satisfying convex union, homogeneity, stability, and weighted quasi-smoothing with the same constant.
Then, by Theorem~\ref{thm:weightconvex}, the theorem follows.
\end{proof}

%%% Local Variables:
%%% mode: latex
%%% TeX-master: "Smoothings"
%%% End:

\section{Defining the extension}
\label{sec:extension}

We now turn to the proof of Theorem~\ref{thm:convex}. As mentioned
above, we will fix a hyperbolic structure~$\Sigma$ on $S$ (with no
relation to the curve functional~$f$) and use the geodesic
flow~$\phi_t$ on
the unit tangent bundle to define the extension to currents.

In this section we will deal with \emph{return maps} for this flow. After
some generalities about return maps for cross-sections with boundary,
we introduce a \emph{smeared return map} that is continuous. We also
define a \emph{homotopy return map} (and a smeared version of it) that
keeps track of homotopy classes of closures of trajectories.
Then we use the smeared homotopy return map to define the
extension~$f_\tau$ of~$f$. The several variants of the return map are
summarized in
Table~\ref{tab:return-notation}.
\begin{table}
  \begin{tabular}{@{}cr@{$\mskip\thickmuskip\to\mskip\thickmuskip$}ll@{}}
    \toprule
    Notation & \multicolumn{2}{c}{Type} & Meaning \\ \midrule
    $p$ & $\tau$ & $\tau$ & Ordinary first return \\
    $P$ & $\tau$ & $\R_1\tau$ & Smeared first return \\
    $q$ & $\tau$ & $\tau\times\pi_1(M)$ & Complete homotopy return \\
    $Q$ & $\tau$ & $\R_1(\tau\times\pi_1(M))$
                              & Smeared homotopy return \\
    $m$ & $\tau$ & $\pi_1(M)$ & Return curve (projection of $q$)\\
    $[m]$ & $\tau$ & $\Curves^+(S)$ & Conjugacy class of projection of
                                      $m$ to~$S$\\
    $M$ & $\tau$ & $\R_1\pi_1(M)$
                       & Smeared return curve (projection of $Q$)\\
    $[M]$ & $\tau$ & $\R_1\Curves^+(S)$ \\
    $R$ & $\GC^+(S)$ & $\R\Curves^+(S)$
                    & Integral of $[M]$\\
    \bottomrule\addlinespace
  \end{tabular}
  \caption{Various types of return maps. $M = UT\Sigma$ is the
    domain of the flow $\phi_t$.}
  \label{tab:return-notation}
\end{table}
\subsection{Smeared first return map}
\label{sec:smeared-return}
Let $Y$ be a smooth closed manifold with a smooth flow~$\phi_t$.
For us, a \emph{cross-section} is a compact smooth codimension~$1$
submanifold-with-boundary~$\tau$ that
is smoothly transverse to the foliation of $Y$ given by~$\phi_t$.
A \emph{global cross-section} is a cross-section $\tau$ so that its
interior $\tau^\circ$
intersects all forward and backward orbits: for
all $x \in Y$, there exists $s<0$ and $t > 0$ with $\phi_s(x),
\phi_t(x) \in \tau^\circ$.
Any flow on a compact manifold has global cross-section consisting of
a union of finitely many disks, although not all flows on compact
manifolds admit a global cross-section consisting of a single
connected component. (For instance, the Reeb foliation on
$\mathbb{T}^2$ has no connected global cross-section.) By the implicit
function theorem, for any cross-section~$\tau$ there is a larger cross
section~$\tau'$ with $\partial \tau
\subset \tau'{}^\circ$, which we also write  $\tau \Subset \tau'$.

Let $t_\tau \colon Y \to \mathbb{R}$ be the first return time defined by
$t_{\tau}(x)\coloneqq\min \{ t>0 \mid \phi_t(x) \in \tau \}$, and let
$p_{\tau}(x)\coloneqq\phi_{t_{\tau}(x)}(x)$. Then $p_{\tau}$
(restricted to~$\tau$) is the \emph{first return map} associated to
the cross-section~$\tau$. We will omit the subscript on $p_\tau$ if it
is clear from context. We also have the first return time to the
interior, denoted~$t_\tau^\circ$. (Recall we assume $t_\tau^\circ(x)$ is finite.)

If $\tau$ has no boundary, then $p$ is a
homeomorphism.
On the other hand, if the cross-section has a non-invariant boundary
(i.e., $p(\partial \tau) \ne \partial \tau$) then 
$t_\tau$ and $p$ will have discontinuities. This necessarily happens
for the geodesic flow on the unit tangent bundle of a hyperbolic
surface. See \cite[Sec.~1]{DC16:BirkhoffSections} for a
justification and examples of global cross-sections with boundary for
this flow; we construct our own cross-section in
Section~\ref{sec:cross-section}.
However, by the continuity of~$\phi_t$ with respect to initial
parameters, we have the following ``local continuity''
claim.

\begin{lemma}
  Let $\tau_1, \tau_2$ be cross-sections (not necessarily global)
  of~$\phi$. Let $x_1 \in \tau_1^\circ$, and suppose we are given
  $t> 0$ so $x_2=\phi_t(x_1) \in \tau_2^\circ$. Then there exists
  a neighborhood $U_1$ of $x_1$ in~$\tau_1$, a neighborhood $U_2$ of $x_2$
  in~$\tau_2$, and a continuous function $t^1_2\colon U_1 \to \R_{>0}$ so
  that for $x \in U_1$, $\phi_{t^1_2}(x) \in U_2$. Furthermore,
  $\phi_{t^1_2}\colon U_1 \to U_2$ is a diffeomorphism. If $t$ is the
  first return to~$\tau_2$, then we can choose $U_1$ and~$U_2$ so that
  $t^1_2(x) = t_{\tau_2}(x)$ is also the first return time.
\label{lem:localhomeo}
\end{lemma}

This is presumably standard (Basener gives this as ``a useful
technical lemma, the proof of which is trivial''
\cite[Lem.~1]{Basener02:Global}), but we give a proof for
completeness.

\begin{proof}
  Pick an initial neighborhood $U_1'$ of $x_1$ in~$\tau$, and let
  $\epsilon$ be small enough so that
  $V_1 \coloneqq \phi_{(-\epsilon,\epsilon)}(U_1')$ is a 3-dimensional
  flow-box
  neighborhood of $x_1$ in~$Y$. Then the restriction of~$\phi_t$
  to~$V_1$ is a
  homeomorphism to a neighborhood $V_2$ of $x_2$ in~$Y$.
  Set $U_2'\coloneqq V_2 \cap \tau_2$. Now consider
  the composition
  $\psi \coloneqq \pi_1 \circ \phi_{-t} \circ \iota_2 \colon U_2' \to U_1'$,
  where $\iota_i \colon U_i' \hookrightarrow V_i$ is the inclusion and
  $\pi_1 \colon V_1 \to U_1$ is the flow projection:
  \[
    \begin{tikzpicture}[x=1.75cm,y=1.25cm]
%      \node(x1) at (-0.6,0) {$x_1$}; \node at (-0.3,0) {$\in$};
      \node(U1) at (0,0) {$U_1'$};
      \node(V1) at (1,0) {$V_1$};
      \node(V2) at (2,0) {$V_2$};
      \node(U2) at (3,0) {$U_2'$};
%      \node(x2) at (3.6,0) {$x_2$}; \node at (3.3,0) {$\ni$};
      \node(tau1) at (0,-1) {$\tau_1$};
      \node (tau2) at (3,-1) {$\tau_2$.};
      \draw[right hook->,bend right=15] (U1) to node[below,cdlabel]{\iota_1} (V1);
      \draw[->>,bend right=15] (V1) to node[above,cdlabel]{\pi_1} (U1);
      \draw[->] (V1) to node[above,cdlabel]{\phi_t} node[below,cdlabel]{\cong} (V2);
      \draw[left hook->] (U2) to node[below,cdlabel]{\iota_2} (V2);
      \draw[right hook->] (U1) to (tau1);
      \draw[right hook->] (U2) to (tau2);
      \draw[->] (U2) to[out=140,in=0] (1.5,0.7) node[above,cdlabel]{\psi} to[out=180,in=40] (U1);
    \end{tikzpicture}
  \]
  Then $\psi$ is a map from $U_2'$ to $U_1'$, taking $x_2$ to~$x_1$.
  By transversality of $\tau_1$ and~$\tau_2$, the differential
  of~$\psi$ at~$x_2$ is invertible. Thus by the inverse function
  theorem, there is a neighborhood $U_2$ of~$x_2$ and $U_1$ of~$x_1$
  so that the restriction of $\psi$ is a diffeomorphism from $U_2$
  to~$U_1$. For $x \in U_1$, set
  $t^1_2(x) \coloneqq t + \pi_t(\phi_{-t}(\psi^{-1}(x)))$, where $\pi_t
  \colon V_1 \to (-\epsilon,\epsilon)$ is the projection onto the time
  coordinate of the flow box. We have
  $\phi_{t^1_2}(x) = \psi^{-1}(x) \in U_1$, as desired for the first claim.
  
  For the second claim, by hypothesis, the compact sets
  $\phi_{[0,t]}(x_1)$ and $\tau_2$ do not intersect, so
  $\phi_{[0,t]}(x_1)$ has an open neighborhood that does not intersect
  $\tau_2$. It follows that we can shrink $U_1$ and $U_2$ so that
  $\psi^{-1}$ restricted to~$U_1$ is the first return map to~$\tau_2$.
\end{proof}

\begin{lemma}\label{lem:contint}
Let $x \in \tau$. If $p(x) \in \tau^\circ$, then
$p_\tau$ and $t_\tau$ are continuous in a neighborhood of $x$ in~$\tau$.
\end{lemma}
\begin{proof}
  Let $\tau'\Supset \tau$ be a slightly enlarged global cross-section
  (to cover cases when $x \in \partial\tau$).
  By Lemma~\ref{lem:localhomeo}, there exists a neighborhood $U$ of
  $x$ in $\tau'$  such
  that $p_{\tau}(U) \subset \tau^\circ$. By
  taking $V\coloneqq U \cap \tau$, we get the desired neighborhood
  in~$\tau$.
\end{proof}

\begin{lemma}\label{lem:return-semi-cont}
  Let $\tau$ be a global cross-section. Then, on~$\tau$,
  $t_\tau$ is lower
  semi-continuous and $t_\tau^\circ$ is
  upper semi-continuous. There are thus positive global upper and lower bounds
  on~$t_\tau$.
\end{lemma}
\begin{proof}
  Fix $x \in \tau$. If $p_\tau(x) \in \tau^\circ$, then $t_\tau$
  is continuous at~$x$ by Lemma~\ref{lem:contint}. Otherwise, find a
  cross-section~$\tau'$ with $\tau \Subset \tau'$. Then $t_{\tau'}$
  is continuous at~$x$, and since $t_\tau(y) \ge t_{\tau'}(y)$ we have
  proved that $t_\tau$ is lower semi-continuous.

  On the other hand, for any $x_1 \in \tau$, we can set $x_2=p_\tau^\circ(x_1)$ and
  find a neighborhood $U_1$ of $x_1$ in $\tau$
  with a function $t^1_2$
  as in Lemma~\ref{lem:localhomeo}. But then $t^\circ_\tau(x) \le
  t^1_2(x)$ for $x \in U_1$. (Note that $x_2$ is the first return point
  to~$\tau^\circ$, not to~$\tau$, so we cannot conclude that
  $t_\tau^\circ$ is continuous.)
\end{proof}

In the $C^1$ setting, Basener showed that a global cross-section can
be perturbed
slightly so that the first return map is piecewise continuous with a
cellular structure~\cite{04TransverseDisk}. However, we want a
\emph{continuous} version of the first return map, so we proceed in a
different direction.

\begin{definition}\label{def:smeared-return}
  Fix a nested pair of global cross-sections $\tau_0 \Subset \tau$.
  A \emph{bump function}~$\psi$ for this pair is a continuous function
  $\psi \colon \tau \to [0,1]$ so that $\psi$ is $1$ on $\tau_0$ and
  $0$ on an open neighborhood of $\partial \tau$. Set
  $\bar{\psi}(x) = 1-\psi(x)$. Let $p\colon \tau \to \tau$
  be the first return map with respect to~$\tau$. Then the
  \emph{smeared first return map} of $\psi$ is a function
  $P_\psi \colon \tau \to \R_1\tau$ defined by
\[
  P_{\psi}(x)\coloneqq
  \begin{cases}
    p(x) & p(x) \in \tau_0\\
    \psi(p(x))\cdot p(x) + \bar\psi(p(x))\cdot P_{\psi}(p(x))
      & p(x) \in \tau - \tau_0.
  \end{cases}
\]
Here, $\R_1\tau \subset \Meas_1(\tau)$ is the subspace of
  measures with finite support and total mass~$1$; see
  Convention~\ref{conv:finitemeasures}.
\end{definition}

The convention here is that a \emph{smeared map} takes values in
finite linear combinations of the target space (or maybe in measures).
We use capital letters for smeared maps.

Intuitively, we iterate $x$ forward, stopping at each iterate
with probability given by~$\psi$. More visually, imagine the original
cross-section as a disk. As we look along the flow lines, we see an
overlapping set of disks, with hard edges between them. To find the
smeared first return map, we ``feather'' the edges by giving the disks
partially-transparent boundaries made out of cellophane. If we
continuously increase the transparency towards the boundary, the
resulting image will have soft edges. See
Figure~\ref{fig:smearedfirstreturn}.

We will usually omit~$\psi$ from the notation and denote
the smeared first return map by~$P$.

Since $\tau_0$ is a global cross-section, in the definition of
$P$ we eventually take the first choice, and so $P(x)$ is a
finite sum of elements of~$\tau_1$ as claimed. A little more is true.
\begin{lemma}\label{lem:return-bound}
  If $\tau_0 \Subset \tau_1$ is a nested pair of global
  cross-sections, there is an $N>0$ so that for any
  $x \in \tau_1$, there is an integer $k < N$ so that
  $p_1^k(x) \in \tau_0$, where $p_1$ denotes the first return map for $\tau_1$
\end{lemma}
\begin{proof}
  By Lemma~\ref{lem:return-semi-cont}, there is an upper bound on the
  return time from~$\tau_1$ to~$\tau_0$, and thus an upper bound on
  the number of intersections of the return path to~$\tau_0$ with the
  compact set~$\tau_1$.
\end{proof}

As a consequence of Lemma~\ref{lem:return-bound}, we can rewrite~$P$
directly. Let $N$ be the bound from Lemma~\ref{lem:return-bound}. Then
\begin{equation}
  \begin{aligned}
  \label{eq:P-expr}
  P(x) &= \psi(p(x)) \cdot p(x) + \bar\psi(p(x))\psi(p^2(x)) \cdot p^2(x) + \cdots \\
    &= \sum_{k=1}^N \bar\psi(p(x))\cdots\bar\psi(p^{k-1}(x)) \psi(p^k(x)) \cdot p^k(x).
  \end{aligned}
  \end{equation}
If we extend the upper limit of the sum beyond~$N$, the additional
terms will be~$0$.

\begin{proposition}\label{prop:smeared-return-cont}
  For any nested global cross-sections $\tau_0 \Subset \tau$ and bump
  function~$\psi$, the smeared first return map~$P$
  is continuous. 
\end{proposition}
\begin{proof}
  We wish to show that $P$ is continuous at $x \in \tau$. There is
  some first $n > 0$ such that $p^n(x) \in \tau_0^\circ$. If
  there is any $i$ between $1$ and~$n$ so that $p^i(x) \in \partial \tau$, we first find
  a smaller cross-section~$\tau'$ with
  $\tau_0 \Subset \tau' \Subset \tau$ without this problem, as
  follows. Recall that we assumed that $\psi$ vanishes in a
  neighborhood of $\partial\tau$. Since there are finitely many points
  $p^i(x)$ in the open set $\tau \setminus \supp(\psi)$, we can pick
  $\tau'$ containing $\supp(\psi)$ so that its boundary avoids those
  finitely many~$p^i(x)$. Since $\psi$ vanishes on
  $\tau \setminus \tau'$, the smeared first return map defined with
  respect to~$\tau'$ agrees with that defined with respect to~$\tau$.
  By replacing $\tau$ by~$\tau'$, we may thus assume that
  $p^i(x) \notin \partial \tau$. Similarly shrink $\tau_0$ so that
  $p^i(x) \notin \partial \tau_0$ for $0 < i < n$.

  We proceed by induction on~$n$. If $n=1$, that is, if
  $p(x) \in \tau_0^\circ$, then $p$ is continuous at~$x$ by
  Lemma~\ref{lem:contint}, and $P$ is continuous since the map taking
  a point $x$ to the delta function $\delta_x$ is continuous.
  Otherwise, note that $p_1$ (the return map for $\tau$) is continuous
  at~$x$ (again by Lemma~\ref{lem:contint}, since
  $p_1(x) \in \tau^\circ$). By induction $P$ is continuous at
  $p_1(x)$, and therefore
  \[
    P(x) = \psi(p_1(x)) \cdot p_1(x) + \bar\psi(p_1(x)) \cdot P(p_1(x))
  \]
  is continuous at~$x$.
\end{proof} 

\begin{figure}
\centering{
\resizebox{80mm}{!}{\Huge{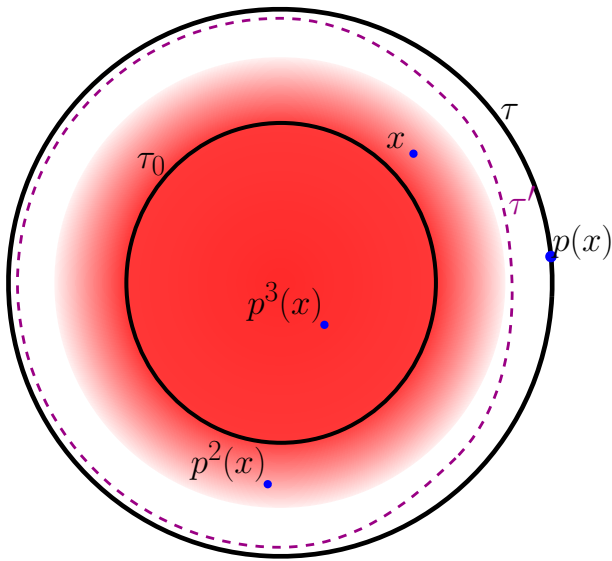}}
\caption{Smeared first return map, illustrating the proof of
  continuity in the case $n=3$ (before shrinking $\tau$). The bump
  function $\psi$ is indicated by the density of red.}
\label{fig:smearedfirstreturn}
}
\end{figure}

To define iterates of~$P$, we first extend $P$ and other functions to
act on measures.

\begin{definition}\label{def:smeared-iterate}
  When $X, Y$ are measure spaces and $f \colon X \to Y$ is a
  measurable function, by convention we extend $f$ to a function
  $\Meas_1(X) \to \Meas_1(Y)$
  acting on measures, denoted~$f_*$ (or simply~$f$),
  by setting
  \[
    f_*(\mu)(S) \coloneqq \mu(f^{-1}(S))
  \]
  for $\mu \in \Meas_1(X)$ and $S \subset Y$ a measurable set. If $f$ is continuous, then this
  extension is continuous with respect to the weak$^*$ topology on
  $\Meas_1(X)$ and $\Meas_1(Y)$.
  (This uses Proposition~\ref{prop:narrowwide} and the fact that if $g
  \colon Y \to \R$ is bounded, then $g \circ f \colon X \to \R$ is
  also bounded.)
  If $f$ is invertible and $\psi\colon X \to \R_{\ge 0}$ is a scaling
  factor, then
  \begin{equation}
    \label{eq:lift-scale}
    f(\psi \cdot \mu) = (\psi \circ f^{-1}) \cdot f(\mu).
  \end{equation}
  In practice, we will often be
  interested in the subspace of finitely supported measures, in which
  case the extension $f\colon \R_1 X \to \R_1 Y$ is given by
  \[
    f\Bigl(\sum a_i x_i\Bigr) \coloneqq \sum a_i f(x_i).
  \]
  For $F: X \to \Meas_1(Y)$ a smeared function,
  we extend $F$ to a function $\Meas_1(X) \to \Meas_1(Y)$ from measures to
  measures, denoted $\tilde F$ (or
  simply~$F$), by setting, for any measurable function $\varphi\colon Y \to
  \mathbb{R}_{\geq 0}$,
  \begin{align*}
    F_\varphi(x) &\coloneqq \int_{y \in Y}\varphi(y) \bigl(F(x)\bigr)(y) \\
    \int_{y \in Y} \varphi(y)\tilde{F}(\mu)(y)
              &\coloneqq \int_{x \in X} F_\varphi(x)\mu(x),
  \end{align*}
  where $F_\varphi \colon X \to \R_{\ge 0}$ is an auxiliary function.
  See also Equations~(\ref{eq:pushforward-P})
  and~(\ref{eq:pushforward-finite}) below.
    \label{def:smearedextension}
\end{definition}
    
  \begin{proposition}
    If $F\colon X \to \Meas_1(Y)$ is continuous, then the extension
    $\tilde{F}\colon \Meas_1(X) \to \Meas_1(Y)$ is continuous.
  \end{proposition}
  \begin{proof}
  Let $(x_i)_{i=0}^\infty$ be a sequence approaching~$x\in X$. By
  assumption, $F(x_i)$ approaches  $F(x)$ in the weak$^*$ topology. By Proposition~\ref{prop:narrowwide},
  this is equivalent to saying that for all continuous
  bounded
  functions $\varphi\colon Y \to \mathbb{R}_{\geq 0}$ the function
  $F_{\varphi}$ above is
  continuous.
  Furthermore, $F_{\varphi}$ is bounded since $F$ takes values in probability measures.
  We now show that $\tilde{F}$ is continuous. Let $\mu_i \to \mu \in \Meas(X)$.
  We want to show that
  $\tilde F(\mu_i) \to \tilde F(\mu)\in \Meas_1(Y)$, i.e., for any
  continuous bounded function
  $\varphi\colon Y \to \mathbb{R}_{\geq 0}$,
   \[
  \int_{x \in X} F_{\varphi}(x)\mu_i(x) \to \int_{x \in X} F_{\varphi}(x)\mu(x).
  \] 
   This is true by definition of the weak$^*$ topology in
   $\Meas_1(X)$ and Proposition~\ref{prop:narrowwide}, since $F_{\varphi}$ is
   continuous and bounded.
\end{proof}

In our applications, $F$ takes values in finitely-supported measures,
with a bound on the size of the support.
  Concretely, if $F: X \to \R Y$ can be written
  as a finite sum
  \[
    F(x) = \sum_{i=1}^N \psi_i(x) f_i(x)
  \]
  for real-valued functions $\psi_i$ and invertible $Y$-valued
  functions~$f_i$, then, by Equation~(\ref{eq:lift-scale}), the extension is defined by
  \begin{equation}\label{eq:pushforward-P}
    F(\mu) = \sum_{i=1}^N f_i(\psi_i\cdot\mu) =
    \sum_{i=1}^N (\psi_i\circ f_i^{-1}) \cdot f_i(\mu)
  \end{equation}
  where the middle expression is a sum of pushforwards of scaled
  measures, and in the last expression we have pulled the scaling
  factors out. If $\mu$ is also finitely supported, we have
  \begin{equation}\label{eq:pushforward-finite}
    F\Bigl(\sum_i a_i x_i \Bigr) = \sum_{i,j} a_i \psi_j(x_i) f_j(x_i).
  \end{equation}
  
\begin{definition}
  With the above extension of notation, the iterates of the smeared return
  map~$P$ are defined by
  \begin{equation}\label{eq:iterate-P-def}
  \begin{aligned}
    P^0(x) &\coloneqq x \\
    P^n(x) &\coloneqq P(P^{n-1}(x)).
  \end{aligned}
  \end{equation}
\end{definition}
  
\begin{definition}
A measure $\nu$
on~$Y$ that is invariant under the flow~$\phi_t$ induces a \emph{flux}
$\mu = \nu_\tau$ on a global cross-section $\tau$ that is invariant
under the first return
map $p$ \cite[Section 3.4.2]{VO16:FoundationsErgodic}. Concretely,
pick $\epsilon > 0$ small enough so that the map
$b \colon [0,\epsilon] \times \tau \to Y$ defined by
$b(t,x) = \phi_t(x)$ is an embedded flow box. Then for $S \subset \tau$,
define $\mu(S) \coloneqq \nu\bigl(\phi_{[0,\epsilon]}(S)\bigr)/\epsilon$. (Since
our cross-sections are compact manifolds-with-boundary, we can always
find such an~$\epsilon$.)
\label{def:flux}
\end{definition}

We have the following invariance property.

\begin{proposition}\label{prop:smeared-invariance}
  If $\nu$ is a measure on $UT\Sigma$ that is invariant under
  $\phi_t$, then $\nu_\tau$ is
  invariant under $p$ and $\psi\nu_\tau$ is invariant under~$P$.
\end{proposition}

For motivation for the factor of~$\psi$ in the proposition statement,
think about extending the
definitions to
allow $\psi$ to be the (non-continuous) characteristic function of $\tau_0 \Subset \tau$; then $P$
is the ordinary first return map to~$\tau_0$, and $\psi \nu_\tau = \nu_{\tau_0}$ is
the flux of~$\tau_0$. See also Example~\ref{ex:smearedreturn}.

\begin{proof}
  The first part is standard. For the second part, let $\mu =
  \nu_\tau$. Then we have
  \begin{align*}
    P\bigl(\psi\cdot\mu\bigr)
    &= p\bigl(\psi\cdot(\psi\circ p)\cdot\mu\bigr) +
      p^2\bigl(\psi\cdot (\bar\psi\circ p)\cdot (\psi\circ p^2)\cdot\mu\bigr) + \dots\\
    &= \sum_{k=1}^N p^k\bigl(\psi\cdot(\bar\psi \circ p)\cdots
      (\bar\psi\circ p^{k-1})\cdot(\psi \circ p^k)\cdot \mu\bigr)\\
    &= \sum_{k=1}^N (\psi\circ p^{-k})\cdot(\bar\psi \circ p^{-k+1})\cdots
      (\bar\psi \circ p^{-1}) \cdot \psi\cdot p^k(\mu)\\
    &= \sum_{k=1}^N (\psi\circ p^{-k})\cdot(\bar\psi \circ p^{-k+1})\cdots
      (\bar\psi \circ p^{-1}) \cdot \psi\cdot \mu.
  \end{align*}
  using the definition of~$P$ (in the form of
  Equation~\eqref{eq:P-expr}); rewriting as a sum;
  Equation~\eqref{eq:pushforward-P}; and invariance of~$\mu$ under~$p$.
  Since $\psi + \bar \psi = 1$, this sum telescopes, and the result
  is~$\psi\mu$.
\end{proof}

We will sometimes blur the distinction between a geodesic current and
its flux, and write, for instance, $\nu(\tau)$ for the total mass of
$\nu_\tau$ on~$\tau$, or $\nu(\psi\tau)$ for $\int_{x \in \tau} \psi(x) \nu_\tau(x)$.

\subsection{Homotopy type of return}
\label{sec:return-homotopy}
We will additionally need to track \emph{how} a point returns to the
cross-section. For this, we suppose that we have a global
cross-section~$\tau$ contained in a simply-connected
cross-section~$\tau'$. (For a $C^1$ flow on a manifold of
dimension at least 3, there is always a simply-connected
cross-section~\cite{04TransverseDisk}.)
For such a cross-section, from the first return for $x \in \tau$, we can
extract another piece of information: the homotopy class of the return
trajectory.

\begin{definition}
  Let $\phi_t$ be a flow on a manifold~$Y$ and $\tau$ be a
  global
  cross-section, contained in a larger compact simply-connected
  cross-section~$\tau'$. Fix a basepoint $* \in \tau'$. For $x\in
  \tau$, define
  the \emph{return trajectory} $m(x) \in \pi_1(Y, *)$ by taking the
  homotopy class of a path that runs in~$\tau'$ from~$*$ to~$x$, along
  the flow trajectory from $x$ to $p_\tau(x)$, and then in~$\tau'$ from
  $p_\tau(x)$ back to~$*$. Since $\tau'$ is simply-connected, $m(x)$
  is independent of the choice of path.
  \label{def:homotopyreturn}
\end{definition}

\begin{lemma}\label{lem:homotopy-finite}
  Let $Y$ be a compact manifold with flow $\phi_t$ and $\tau$ be a
  global
  cross-section.
  As $x$ varies in~$\tau$, the return trajectory $m(x)$ takes on only
  finitely many values.
\end{lemma}
\begin{proof}
  Since there are upper bounds on the return time
  (Lemma~\ref{lem:return-semi-cont}), on the speed
  of~$\phi_t$ with respect to a Riemannian metric on~$Y$, and on the
  diameter of~$\tau'$, the length
  of the path representing $m(x)$ is bounded.
  On a compact manifold, there are only finitely many elements
  of~$\pi_1(Y,*)$ that have representatives of bounded length.
\end{proof}

To get the return map for iterates, we
also incorporate the point of first return.
\begin{definition}\label{def:homotopy-return}
  The \emph{homotopy return map} is the map $q \colon \tau \to \tau
  \times \pi_1(Y,*)$ defined by
  \[
    q(x) \coloneqq (p(x), m(x)).
  \]
  We can iterate $q$ by inductively defining $q^{n+1}$ to be the
  composition
  \[
    \tau \overset{q^n}{\longrightarrow} \tau \times \pi_1(Y)
      \xrightarrow{q \times \id}
        \tau \times \pi_1(Y) \times \pi_1(Y)
      \xrightarrow{(x,g,h)\mapsto(x,hg)}
        \tau \times \pi_1(Y).
  \]
  Define $m^n(x)\in \pi_1(Y,*)$ to be the second component of
  $q^n(x)$.
\end{definition}

\begin{remark}
  An alternative approach to defining homotopy types of return
  trajectories is to pick a cross-section in the universal cover
  \cite[Section~3.2]{EPS:CountingCurves}.
\end{remark}

\begin{definition}
  For $\tau$ a global cross-section with basepoint~$*$,
  $\tau_0 \Subset \tau$ a smaller global cross-section, $\psi$ a
  bump function for this pair, and $\tau' \supset \tau$ a simply-connected
  cross-section, the \emph{smeared homotopy return map}
  $Q\colon \tau \to \R_1(\tau \times \pi_1(Y,*))$ is defined by
  \begin{align*}
    Q(x) &\coloneqq
    \begin{cases}
      q(x) & p(x) \in \tau_0 \\
      \psi(p(x)) \cdot q(x) + \bar\psi(p(x))\cdot L_{m(x)}Q(p(x))
        & p(x) \in \tau - \tau_0
      \end{cases}
  \end{align*}
  where $L_g$ is left translation by~$g \in \pi_1(Y,*)$:
  \[
    L_{g}\Bigl(\sum_ia_i(x_i,h_i)\Bigr) \coloneqq \sum_i a_i (x_i, gh_i).
  \]
  There is once again a natural notion of iteration, defined by
  inductively setting $Q^{n+1}$ to be the composition
  \[
    \tau \overset{Q^n}{\longrightarrow} \R_1(\tau \times \pi_1(Y))
      \xrightarrow{\R_1(Q \times \id)} \R_1(\R_1(\tau \times \pi_1)\times\pi_1(Y))
     \xrightarrow{\operatorname{join}}  \R_1(\tau \times \pi_1(Y)).
  \]
  where $\operatorname{join}$ is the somewhat more involved
  operation
  \[
    \operatorname{join}\Biggl(\sum_i a_i \Bigl(\Bigl(\sum_j b_{ij} (x_{ij},
      g_{ij})\Bigr), h_i\Bigr)\Biggr)
    \coloneqq \sum_{i,j} a_i b_{ij}(x_{ij},h_ig_{ij}).
  \]
  (The terminology comes from the theory of monads
  \cite{Moggi91:Monads,Wadler92:MonadsFP}. See
  Equation~(\ref{eq:pushforward-finite}).)
\end{definition}
  
\begin{definition}
  We define the \emph{smeared $n$-th return trajectory}
  $M^n \colon \tau \to \R_1 \pi_1(Y)$ to be the composition
  \[
    \tau \overset{Q^n}{\longrightarrow}\R_1(\tau \times \pi_1(Y))
    \longrightarrow \R_1 \pi_1(Y)
  \]
  where at the second step we lift the projection on the second
  component to act on weighted objects as in
  Definition~\ref{def:smeared-iterate}.
  \label{def:smearedhomotopyreturn}
\end{definition}

Let $\Lambda(n,\tau)$ be the set of curves that appear with non-zero
coefficient in $M^n(x)$ for some $x \in \tau$.
\begin{lemma}\label{lem:smeared-iterate-finite}
  $\Lambda(n,\tau)$ is finite.
\end{lemma}
\begin{proof}
  Immediate from Lemmas~\ref{lem:return-bound}
  and~\ref{lem:homotopy-finite}.
\end{proof}

\begin{lemma}\label{lem:smeared-homotopy-cont}
  The maps $Q^k$ and $M^k$ are continuous.
\end{lemma}
\begin{proof}
  The proof of Proposition~\ref{prop:smeared-return-cont} also proves that
  $Q$ is continuous. It then follows that $Q^k$ and $M^k$ are continuous.
\end{proof}

\subsection{Return maps for the geodesic flow}
We now turn to the specifics of our situation. Let $\Sigma$ be the
surface~$S$ endowed with an arbitrary
hyperbolic Riemannian metric~$g$.
Points in $UT\Sigma$ will be denoted~$\vec x$, meaning a pair of a
point $x \in \Sigma$ and a unit tangent
vector at~$x$.
Let
$\phi_t \colon UT \Sigma \to UT \Sigma$ be the geodesic flow associated
to~$g$.

We pick nested global cross-sections $\tau_0 \Subset \tau \subset
\tau'$, with $\tau'$ simply-connected,
and a bump function~$\psi$ for the pair $(\tau_0,\tau)$. We thus get a
smeared $n$-th return
trajectory $M^n \colon \tau \to \R_1 \pi_1(UT\Sigma, *)$. We want to work
with curves in~$\Sigma$ rather than its unit tangent bundle, so
compose with the projection
$\pi_\Sigma \colon UT\Sigma \to \Sigma$ to get a linear combination of
elements of $\pi_1(S, \pi_\Sigma(*))$. Then take conjugacy classes (to pass
to unbased curves) to get
an element of $\R_1\Curves^+(S)$. We call the resulting map~$[M^n]$, which has type
\[
  [M^n] \colon \tau \to \R_1 \Curves^+(S).
\]
From Lemmas~\ref{lem:smeared-iterate-finite}
and~\ref{lem:smeared-homotopy-cont}, $[M^n]$ is a continuous function
with values in the finite
dimensional subspace
$\R^{[\Lambda(n,\tau)]} \subset \R \Curves^+(S)$, where
$[\Lambda(n,\tau)]$ is
the projection of $\Lambda(n,\tau)$.

\subsection{Definition of the extension}
\label{subsec:extension}
Now, we will use the above return map $[M^n]$ to define the extension of $f$ to geodesic currents in Theorem \ref{thm:convex}.
With $\tau_0 \Subset \tau \subset \tau'$ as above, we define

\begin{gather*}
  R^n \colon \GC^+(S) \to \R\Curves^{+}(S)\\
  R^n(\mu) \coloneqq \int_{\tau} [M^n(\vec x)] \psi(\vec x) d\mu_{\tau}(\vec x).
\end{gather*}
Observe that, for fixed~$n$, $R^n(\mu)$ is a weighted multi-curve with a fixed set of
possible connected components, but with weights depending on~$\mu$. As
we will explain in Section~\ref{sec:continuity}, because
$[M^n(\vec x)]$ is continuous on~$\tau$, $R^n$ is continuous with respect to the weak$^*$
topology on $\GC^+(S)$.

We can now finally define our extension of~$f$.
\begin{definition}\label{def:extension}
Let $\tau$ be a good cross-section and $f$ a weighted curve functional satisfying stability,
  homogeneity, weighted quasi-smoothing, and convex union. We define
\begin{align}
  f^n_{\tau}(\mu)&\coloneqq f(R^n(\mu))\label{eq:fntau}  \\
  f_{\tau}(\mu)&\coloneqq\lim_{n \to \infty} \frac{f^n_{\tau}(\mu)}{n}.\label{eq:ftau}
\end{align}
\end{definition}

We will prove that the limit exists (at least for our 
cross-section) in
Proposition~\ref{prop:limitexists}.

\begin{warning}
  We work with weighted linear combinations of objects (or, more
  generally, measures) at many places in the paper. Some functions
  (like~$R^n$) are by definition additive under linear combinations,
  and in Definition~\ref{def:smeared-iterate}, we also silently
  extend other functions (like~$p$) to apply additively to
  linear combinations of points or measures.
  But the main curve functional~$f$ we are interested is \emph{not}
  necessarily additive. (We only assume that $f$ satisfies convex
  union in the main theorems.)
\end{warning}

%%% Local Variables:
%%% mode: latex
%%% TeX-master: "Smoothings"
%%% End:

\section{Constructing global cross-sections}
\label{sec:cross-section}

Next we define the specific global cross-section we use. We make
choices that are convenient for guaranteeing that certain crossings
are essential.

\begin{definition}
\label{def:wedge}
  For $\Sigma$ a hyperbolic surface, $c$ an oriented geodesic segment
  on~$\Sigma$, and $0 < \theta < \pi/2$ an angle, the \emph{wedge set}
  $W(c,\theta) \subset UT\Sigma$ is the set of vectors that cross~$c$
  nearly perpendicularly:
  \[
    W(c,\theta) \coloneqq \bigl\{\vec{x}=(x,v)\bigm\vert x\in c,\,
      \abs{\ang(T_x c,v)-\pi/2} \le \theta\,\bigr\}.
  \]
  (Angles $\ang(v,w)$ are measured by the
  counterclockwise rotation from $v$ to~$w$.) We can likewise define
  the wedge set $W(\{c_i\},\theta)$ for a collection of geodesic
  segments $\{c_i\}_{i=1}^k$.
\end{definition}

We wish to find a wedge set $W(\{c_i\},\theta)$ that is an embedded global
cross-section for the geodesic flow $\phi_t$.

Fix $\theta=\pi/6$. For any geodesic arc $c$, the wedge set
$W(c,\theta)$ intersects any geodesic that passes through a non-empty open
set. Thus 
by compactness of $UT\Sigma$, there exist a finite collection of
immersed arcs $(c_i)_{i=1}^n$ so that $\bigcup_{i=1}^n W(\{c_i\},\theta)_{i=1}^{n}$ is a
disconnected, not necessarily embedded, global cross-section of the
geodesic flow.
We will produce an embedded global cross-section from it.
Immersed points come from intersection points between the geodesic segments $c_i$, but not all of them produce immersed points of the global cross-section.

Indeed, suppose $\ang(c_i,c_j)=\varphi$.
There are two good cases:
\begin{enumerate}
    \item If $2\theta < \abs{\varphi}$, the wedge sets don't intersect, as
      shown in Figure~\ref{fig:cut-and-flow2}.
    \begin{figure}
\centering{
\resizebox{80mm}{!}{\fontsize{12pt}{12pt}\selectfont%% Creator: Inkscape inkscape 0.92.4, www.inkscape.org
%% Accompanies image file 'cut-and-flow2.pdf' (pdf, eps, ps)
\begingroup%
  \makeatletter%
  \providecommand\color[2][]{%
    \errmessage{(Inkscape) Color is used for the text in Inkscape, but the package 'color.sty' is not loaded}%
    \renewcommand\color[2][]{}%
  }%
  \providecommand\transparent[1]{%
    \errmessage{(Inkscape) Transparency is used (non-zero) for the text in Inkscape, but the package 'transparent.sty' is not loaded}%
    \renewcommand\transparent[1]{}%
  }%
  \providecommand\rotatebox[2]{#2}%
  \newcommand*\fsize{\dimexpr\f@size pt\relax}%
  \newcommand*\lineheight[1]{\fontsize{\fsize}{#1\fsize}\selectfont}%
  \ifx\svgwidth\undefined%
    \setlength{\unitlength}{223.90879748bp}%
    \ifx\svgscale\undefined%
      \relax%
    \else%
      \setlength{\unitlength}{\unitlength * \real{\svgscale}}%
    \fi%
  \else%
    \setlength{\unitlength}{\svgwidth}%
  \fi%
  \global\let\svgwidth\undefined%
  \global\let\svgscale\undefined%
  \makeatother%
  \begin{picture}(1,0.77681533)%
    \lineheight{1}%
    \setlength\tabcolsep{0pt}%
    \put(0,0){\includegraphics[width=\unitlength,page=1]{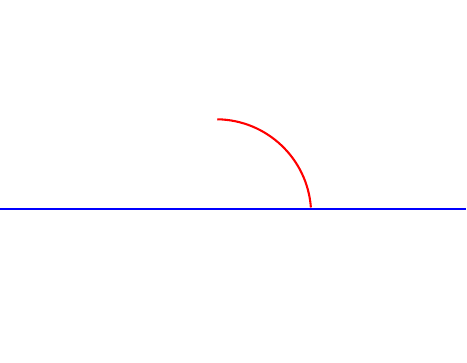}}%
    \put(0.4496151,0.50758296){\color[rgb]{0,0,1}\makebox(0,0)[rt]{\lineheight{1.25}\smash{\begin{tabular}[t]{r}$\theta$\end{tabular}}}}%
    \put(0.60300728,0.48266068){\color[rgb]{1,0,0}\makebox(0,0)[lt]{\lineheight{1.25}\smash{\begin{tabular}[t]{l}$\varphi$\end{tabular}}}}%
    \put(0,0){\includegraphics[width=\unitlength,page=2]{cut-and-flow2.pdf}}%
    \put(0.49309963,0.14838242){\color[rgb]{0,0.69803922,0}\makebox(0,0)[lt]{\lineheight{1.25}\smash{\begin{tabular}[t]{l}$c_j$\end{tabular}}}}%
    \put(0.01767887,0.28048561){\color[rgb]{0,0,1}\makebox(0,0)[lt]{\lineheight{1.25}\smash{\begin{tabular}[t]{l}$c_i$\end{tabular}}}}%
    \put(0,0){\includegraphics[width=\unitlength,page=3]{cut-and-flow2.pdf}}%
  \end{picture}%
\endgroup%
}
\caption{Wedge sets not intersecting, in the case $2\theta < \varphi$.}
\label{fig:cut-and-flow2}
}
\end{figure}
\item
  If $\theta < \frac{\pi}{2}-\abs{\varphi}$, then the corresponding
  wedge sets do intersect, but we can perturb the $c_j$ slightly to
  avoid the intersection.
Given a small interval $[a,b]$ of the geodesic segment $c_j$, containing one intersection point between $c_i$ and $c_j$,
we consider the wedge $W([a,b],\theta)$. By pushing the endpoints
$a,b$ forward along the extremal angles, and removing $[a,b]$ from $c_j$, as shown in
Figure~\ref{fig:cut-and-flow},
we obtain new interval with endpoints $a',b'$, and a new wedge set
$W([a',b'],\theta')$, for some $\theta' > \theta$, so that
$W([a',b'],\theta')$ is disjoint from $c_i$ and
$W([a',b'],\theta)$ intersects every geodesic that $W([a,b],\theta)$
does (so we still have a global cross-section).

\begin{figure}
\centering{
\resizebox{80mm}{!}{\fontsize{12pt}{12pt}\selectfont%% Creator: Inkscape inkscape 0.92.4, www.inkscape.org
%% Accompanies image file 'cut-and-flow.pdf' (pdf, eps, ps)
\begingroup%
  \makeatletter%
  \providecommand\color[2][]{%
    \errmessage{(Inkscape) Color is used for the text in Inkscape, but the package 'color.sty' is not loaded}%
    \renewcommand\color[2][]{}%
  }%
  \providecommand\transparent[1]{%
    \errmessage{(Inkscape) Transparency is used (non-zero) for the text in Inkscape, but the package 'transparent.sty' is not loaded}%
    \renewcommand\transparent[1]{}%
  }%
  \providecommand\rotatebox[2]{#2}%
  \newcommand*\fsize{\dimexpr\f@size pt\relax}%
  \newcommand*\lineheight[1]{\fontsize{\fsize}{#1\fsize}\selectfont}%
  \ifx\svgwidth\undefined%
    \setlength{\unitlength}{225.93812805bp}%
    \ifx\svgscale\undefined%
      \relax%
    \else%
      \setlength{\unitlength}{\unitlength * \real{\svgscale}}%
    \fi%
  \else%
    \setlength{\unitlength}{\svgwidth}%
  \fi%
  \global\let\svgwidth\undefined%
  \global\let\svgscale\undefined%
  \makeatother%
  \begin{picture}(1,0.27350157)%
    \lineheight{1}%
    \setlength\tabcolsep{0pt}%
    \put(0,0){\includegraphics[width=\unitlength,page=1]{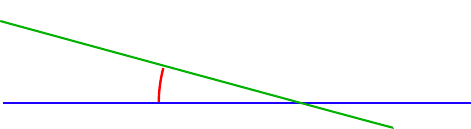}}%
    \put(0.783732,0.18441088){\color[rgb]{0.11372549,0,1}\makebox(0,0)[lt]{\lineheight{1.25}\smash{\begin{tabular}[t]{l}$\theta$\end{tabular}}}}%
    \put(0.32698362,0.08144041){\color[rgb]{1,0,0}\makebox(0,0)[rt]{\lineheight{1.25}\smash{\begin{tabular}[t]{r}$\varphi$\end{tabular}}}}%
    \put(0,0){\includegraphics[width=\unitlength,page=2]{cut-and-flow.pdf}}%
    \put(0.08014293,0.01835524){\color[rgb]{0.11372549,0,1}\makebox(0,0)[t]{\lineheight{1.25}\smash{\begin{tabular}[t]{c}$c_j$\end{tabular}}}}%
    \put(0.11928592,0.16361802){\color[rgb]{0,0.69803922,0}\makebox(0,0)[rt]{\lineheight{1.25}\smash{\begin{tabular}[t]{r}$c_i$\end{tabular}}}}%
    \put(0,0){\includegraphics[width=\unitlength,page=3]{cut-and-flow.pdf}}%
  \end{picture}%
\endgroup%
}
\caption{Cut and flow operation when the angle of intersection is
  small, $\theta < \pi/2 - \varphi$.}
\label{fig:cut-and-flow}
}
\end{figure}

\end{enumerate}
\begin{remark}
Note that the inequality must be strict in case (2).
Indeed, if $\theta = \pi/2-|\phi|$, the
wedge sets do intersect, but this is not a good case: we cannot
guarantee to make the wedge sets locally disjoint by flowing a segment
of $c_j$ forward slightly.
\end{remark}

Our choice $\theta=\pi/6$ guarantees that one of these two cases happens.

Next, we will construct an immersed connected global
cross-section~$\tau'$ containing this disconnected wedge set.

\begin{proposition}\label{prop:transverse-approx}
For any wedge set $W(\{ c_i \},\theta)$ and any $\epsilon>0$,
there exists a closed geodesic~$\delta$ so that for each~$i$,
there is a sub-segment $\delta_i \subset \delta$ so that
$c_i \subset B_{\epsilon}(\delta_i)$ and every geodesic that intersects
$W(c_i,\theta)$ also intersects $W(\delta_i,\theta + \epsilon)$.
\end{proposition}
\begin{proof}
  Use \cite[Theorem 2.4]{BPS17:FillingGeodesics} to construct the
  closed geodesic $\delta$ with $c_i \subset B_{\epsilon}(\delta)$. If
  $\epsilon$ is small enough, by
  following the geodesic flow from $W(c_i,\theta)$ we hit $\delta$ in
  a geodesic segment~$\delta_i$, so that every geodesic intersecting
  $W(c_i,\theta)$ also intersects
  $W(\delta_i,\theta + \epsilon')$ for some~$\epsilon'$. Since
  $\epsilon'$ goes to~$0$ as $\epsilon$ goes to~$0$, the result follows.
\end{proof}

Observe that if $\epsilon$ is small enough and the $c_i$ in
Proposition~\ref{prop:transverse-approx} are disjoint, then the
$\delta_i$ will be disjoint as well. Thus
combining the above propositions (and redefining $\theta$ to be
$\theta + \epsilon$), we have found a closed geodesic
$\delta$ and disjoint geodesic segments $\delta_i' \subset \delta$
so that we have the following global cross-sections:
\begin{itemize}
    \item A wedge set $\tau_0 \coloneqq W(\{ \delta_i \},\theta)$ giving a
      disconnected embedded global cross-section.
    \item A global cross-section $W(\delta,\theta)$
      containing the previous one which is connected but not embedded
      (as $\delta$ will self-intersect).
    \end{itemize}
We would like to use the second global cross-section
$W(\delta,\theta)$ to close up the homotopy return trajectories of the
smeared return map. However, $W(\delta,\theta)$ is not
simply-connected, so the homotopy return map will depend on which path
along the cross-section
we choose.
This can be easily fixed by setting, for some small open interval $I \subset
\delta \backslash \bigcup_i \{ \delta_i \}$,
\[ \tau' \coloneqq W(\delta - I, \theta) \]
so that $\tau_0 \subset \tau'$.
Strictly speaking, $\tau'$ is not simply-connected as a subset of
$UT\Sigma$; rather it is the image of an immersed disk. Since
$\tau_0 \subset \tau'$ lies in a portion where the immersion is injective,
there is no ambiguity about how to connect up the return paths to~$\tau_0$
within~$\tau'$.

\begin{definition}
A \emph{good} cross-section is the data of cross-sections $\tau_0,\tau,\tau'$ and bump function $\psi$, where $\tau$ is a slight enlargement of the embedded cross-section $\tau_0$ so that
$\tau_0 \Subset \tau \subset \tau'$, for $\tau'$ a cross-section as above. For simplicity, we will refer to a good cross-section just as~$\tau$.
\label{def:goodcross}
\end{definition}

A good cross-section gives the complete setup of Section~\ref{subsec:extension}. 

%%% Local Variables:
%%% mode: latex
%%% TeX-master: "Smoothings"
%%% End:

\section{Join lemma}
\label{sec:join}

We now turn to the heart of the proof, proving \emph{join lemmas} to
show that we can smooth essential crossings to relate the return maps
of order~$k$, order~$\ell$, and order $k+\ell$. We chose the global
cross-sections $\tau_0 \Subset \tau \subset \tau'$ in
Section~\ref{sec:cross-section}
to be wedge sets in order to connect to hyperbolic geometry and
prove the necessary crossings are essential.
Recall that we refer to the data of the nested
cross-sections from wedge sets (including the bump function~$\psi$, when
relevant) as a good cross-section (Definition~\ref{def:goodcross}), which we refer to as~$\tau$.

\begin{lemma}[Classical join lemma]
  \label{lem:join-classical}
  Let $\tau$ be a good cross section.
  There is a curve~$K_\tau$ and integer $w_\tau$ so that for large enough
  $k,\ell \geq 0$, we have, for all $\vec{x} \in \tau$,
  \begin{enumerate}[(a)]
      \item $[m^k(\vec{x})] \cup [m^\ell(p^k(\vec{x}))] \cup K_\tau
        \reducesto_{w_\tau} [m^{k+\ell}(\vec{x})]$
        \label{item:join-join}
      \item $[m^{k+\ell}(\vec{x})] \cup K_\tau \reducesto_{w_\tau}
        [m^k(\vec{x})] \cup [m^\ell(p^k(\vec{x}))]$.
        \label{item:join-split}
      \end{enumerate}
\end{lemma}

As a corollary, we will prove a corresponding join lemma for the
smeared return map.

\begin{lemma}[Smeared join lemma]
  \label{lem:join}
  Let $\tau$ be a good cross section.
  There is a curve~$K_\tau$ and weight $w_\tau$ so that for large enough
  $k,\ell \geq 0$, we have, for all $\vec{x} \in \tau$,
  \begin{enumerate}[(a)]
      \item $[M^k(\vec{x})] \cup [M^\ell(P^k(\vec{x}))] \cup K_\tau \reducesto_{w_\tau} [M^{k+\ell}(x)]$
      \item $[M^{k+\ell}(\vec{x})] \cup K_\tau \reducesto_{w_\tau} [M^k(\vec{x})] \cup [M^\ell(P^k(\vec{x}))]$.
  \end{enumerate}
\end{lemma}

\begin{example}
\label{ex:smearedreturn}
As an example of smeared first return map and to illustrate how the
join lemma is applied, consider the case when the geodesic current
$\mu$ is $\delta_{\gamma}$ for $\gamma$ a closed curve whose lift to the
unit tangent bundle intersects the global cross-section $\tau$ at two
points $\vec{x_0}$ and $\vec{x_1}$. We assume further that $\vec{x_0} \notin \tau_0$,
$\psi(\vec{x_0}) = t \in (0,1)$, and $\vec{x_1} \in \tau_0$.
Then, as illustrated in Figure~\ref{fig:first-return-examp},
$[M^1(\vec{x_0})]$ is a curve $C_{0,1}$ with weight~$1$, since
$p_{\tau}(\vec{x_0})=\vec{x_1} \in \tau_0$.
On the other hand, $[M^1(\vec{x_1})]$ consists of a weighted multi-curve with
two components $C_{1,2}$ and $C_{1,3}$ starting from $x_1$ and landing
at $\vec{x_2}=\vec{x_0}$ and $\vec{x_3}=\vec{x_1}$, with weights $t$ and $1-t$, respectively.
Then (with $\mu = \delta_\gamma$) we have
\begin{align*}
  \psi \mu_\tau &= t \delta_{\vec{x_0}} + \delta_{\vec{x_1}}\\
  R^1(\mu) &= tC_{0,1} + tC_{1,2} + (1-t)C_{1,3}.
\end{align*}

\begin{figure}
\centering{
\resizebox{80mm}{!}{\Huge{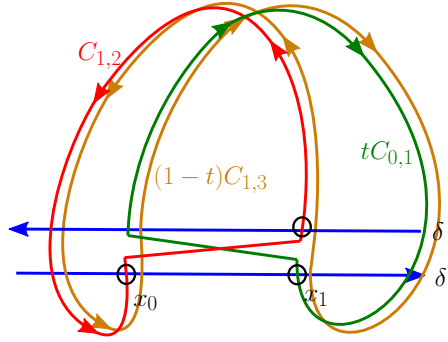}}
}
\caption{Example of a first iteration of the smeared first return map,
  i.e., $[M^1]$ on a geodesic current corresponding to a closed curve
  intersecting the cross-section~$\tau$ twice at points
  $\vec{x_0},\vec{x_1}$ and the cross-section~$\tau_0$ once at point
  $\vec{x_1}$. The weight of the bump function $\psi$ at $\vec{x_0}$
  is $t$. We obtain three weighted curves. $C_{0,1}$ consisting of the
  geodesic trajectory that goes from $\vec{x_0}$ to $\vec{x_1}$ and closes
  off by following the cross-section in some coherent way. $C_{0,1}$ has
  weight 1 by definition of smeared return map, since $\vec{x_1} \in
  \tau_0$. $C_{1,2}$ has weight $t$, whereas $C_{1,3}$ has weight
  $1-t$. }
\label{fig:first-return-examp}
\end{figure}

Now, the join lemma asserts that we can
to join the curves $C_{0,1}$ and $C_{1,2}$, together with an extra
curve $K$, to get $C_{0,2}$. Assuming
all the relevant intersections are essential, we can do it
with $K$ being two copies of~$\delta$, one oriented in each direction,
in the
steps shown in Figure~\ref{fig:join-example}.
\begin{figure}
  \begin{align*}
    \mathcenter{\resizebox{40mm}{!}{\Huge{%% Creator: Inkscape inkscape 0.92.4, www.inkscape.org
%% Accompanies image file 'matchjoin2.pdf' (pdf, eps, ps)
\begingroup%
  \makeatletter%
  \providecommand\color[2][]{%
    \errmessage{(Inkscape) Color is used for the text in Inkscape, but the package 'color.sty' is not loaded}%
    \renewcommand\color[2][]{}%
  }%
  \providecommand\transparent[1]{%
    \errmessage{(Inkscape) Transparency is used (non-zero) for the text in Inkscape, but the package 'transparent.sty' is not loaded}%
    \renewcommand\transparent[1]{}%
  }%
  \providecommand\rotatebox[2]{#2}%
  \newcommand*\fsize{\dimexpr\f@size pt\relax}%
  \newcommand*\lineheight[1]{\fontsize{\fsize}{#1\fsize}\selectfont}%
  \ifx\svgwidth\undefined%
    \setlength{\unitlength}{414.20528862bp}%
    \ifx\svgscale\undefined%
      \relax%
    \else%
      \setlength{\unitlength}{\unitlength * \real{\svgscale}}%
    \fi%
  \else%
    \setlength{\unitlength}{\svgwidth}%
  \fi%
  \global\let\svgwidth\undefined%
  \global\let\svgscale\undefined%
  \makeatother%
  \begin{picture}(1,0.77285676)%
    \lineheight{1}%
    \setlength\tabcolsep{0pt}%
    \put(0,0){\includegraphics[width=\unitlength,page=1]{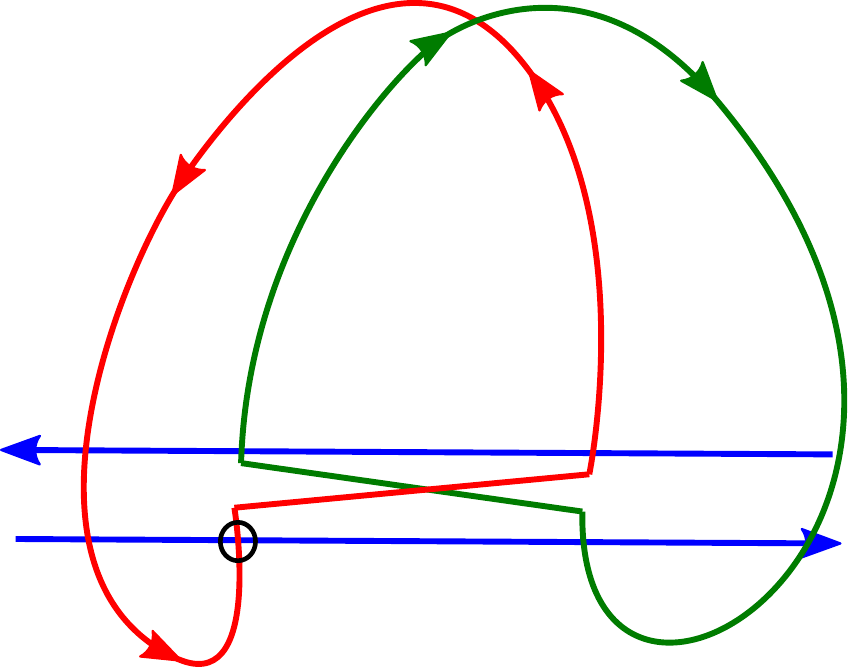}}%
    \put(0.26874801,0.65274748){\color[rgb]{1,0,0}\makebox(0,0)[rt]{\lineheight{1.25}\smash{\begin{tabular}[t]{r}$C_{1,2}$\end{tabular}}}}%
    \put(0.85211048,0.65098466){\color[rgb]{0,0.48627451,0}\makebox(0,0)[lt]{\lineheight{1.25}\smash{\begin{tabular}[t]{l}$C_{0,1}$\end{tabular}}}}%
  \end{picture}%
\endgroup%
}}}\quad
    \searrow \quad\mathcenter{\resizebox{40mm}{!}{\Huge{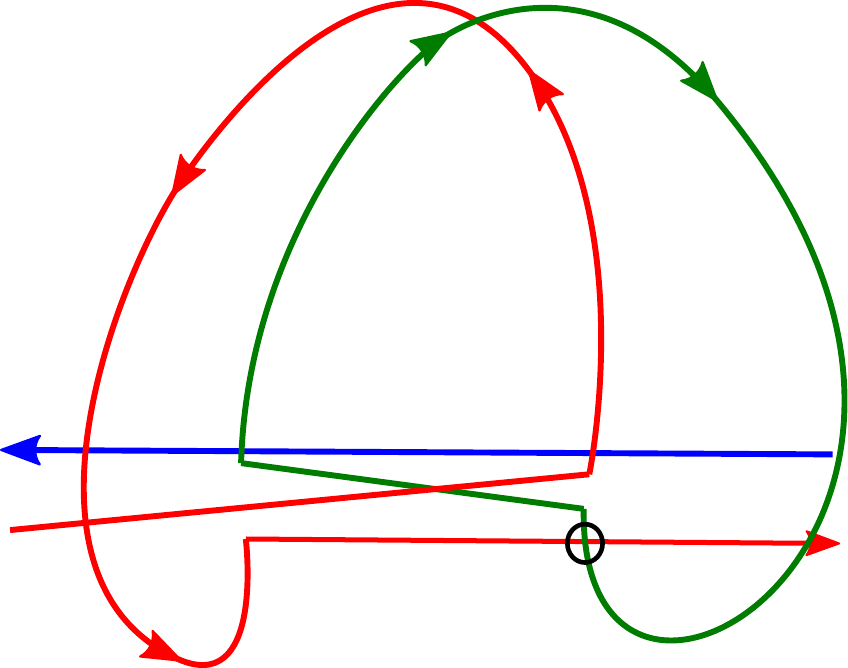}}}\quad
    \searrow \quad\mathcenter{\resizebox{40mm}{!}{\Huge{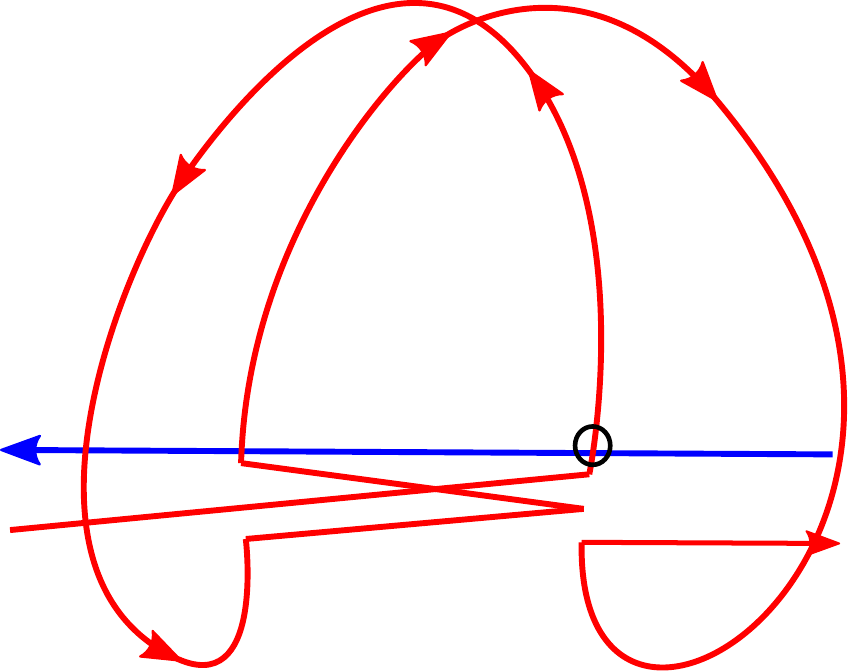}}}\\
    \searrow \quad\mathcenter{\resizebox{40mm}{!}{\Huge{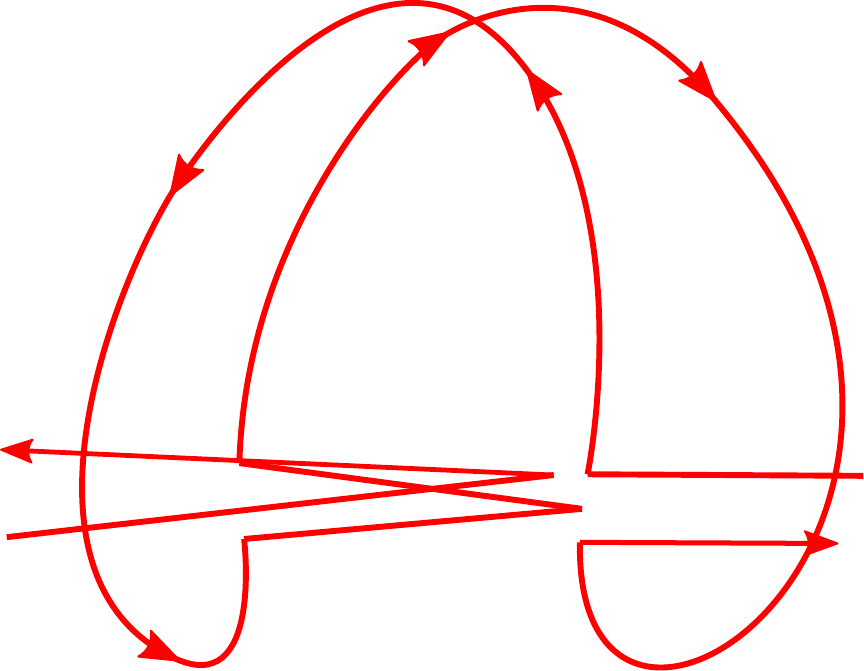}}}\quad
    \simeq \quad\mathcenter{\resizebox{35mm}{!}{\Huge{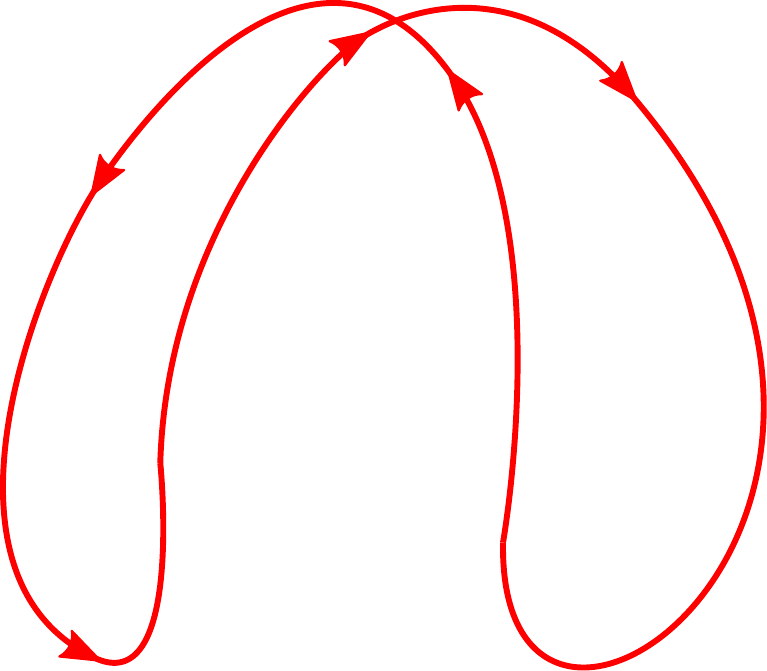}}}
  \end{align*}
  \caption{Applying the Join Lemma in
    Example~\ref{ex:smearedreturn}. At each step, we smooth at the
    circled crossing.}
  \label{fig:join-example}
\end{figure}
In the full proof, to guarantee the analogous intersections are
essential we will add more copies of~$\delta$.
\end{example}

\begin{proof}[Proof of Lemma~\ref{lem:join-classical}]
  First we look at the geometry of a return trajectory $m^k(\vec{x})$
  when $k$ is large, as
  a concrete curve either
  on~$\Sigma$ or lifted to the universal cover.
  Since there is a lower bound on
  the first return time (Lemma~\ref{lem:return-semi-cont}), the $n$-th
  return time grows at least
  linearly in~$k$. We may therefore assume that the portion of
  $m^k(\vec{x})$ that follows $\phi_t(\vec{x})$ is very long. The lift of
  $m^k(\vec{x})$ to the universal cover is thus a \emph{broken path}:
  for some large~$L$, it
  alternates between
  \begin{enumerate}[(a)]
  \item \emph{long} segments of length at least~$L$ following $\phi_t(\vec{x})$ and
  \item \emph{short} segments of some length following~$\delta$,
  \end{enumerate}
  with turns between them that are within $\epsilon$ of a right angle,
  alternating left and right.
  (See Definition~\ref{def:brokenpath} for a precise definition.)

  We study the geometry of broken paths in
  Section~\ref{sec:hyperbolic}. In particular, we prove several lemmas
  there guaranteeing that broken paths intersect essentially in
  certain circumstances. If $L$ is large enough, we have the following
  results.
  \begin{itemize}
  \item A broken path and a lift of~$\delta$ intersect essentially
    (Lemma \ref{lem:brokenpathshort}).
  \item Broken paths with short segments that are different enough in
    length intersect
    essentially (Lemma~\ref{lem:brokenpathslinked}).
  \end{itemize}
  
  We first prove part~\ref{item:join-join} in the Lemma statement. We
  will use the following steps. By convention, $\delta$ is oriented to
  the right, and $\delta^{-1}$ is the same curve oriented to the left.
  Let $\alpha$, $\beta$, and~$\gamma$ be $m^k(\vec x)$ and
  $m^\ell(p^k(\vec x))$, respectively.
  \begin{enumerate}
  \item We start by smoothing $[\alpha]$
    with a large number $N$ of copies of $[\delta]$. Each one of
    these intersections is essential by
    Lemma~\ref{lem:brokenpathshort}. This yields a new curve~$[\alpha_1]$ with
    lift a broken path with an lengthened short segment.
  \item We then smooth $[\alpha_1]$ against~$[\beta]$. The corresponding
    lifts are broken paths with short segments of different enough
    lengths, so
    Lemma~\ref{lem:brokenpathslinked} guarantees that the crossing
    is essential, yielding a new curve~$[\gamma_1]$.
  \item Finally, we smooth $[\gamma_1]$ against $N[\delta^{-1}]$. This returns 
    to the correct homotopy class, again using
    Lemma~\ref{lem:brokenpathshort} to
    guarantee that the crossings are essential. The result
    is~$[\gamma]$, as desired.
  \end{enumerate}

  We need to use a large enough number~$N$ of copies of $\delta$ that
  guarantees that the crossing in the second step above is essential.
  Let $\epsilon$ be the angle of the wedge set, let $\ell$ be the
  length of~$\delta$, and let $\kappa(\epsilon)$ be the constant from
  Lemma~\ref{lem:brokenpathslinked}. Then we claim that it suffices to
  take $N = 2M$ with $M = 1+\lceil\kappa(\epsilon)/\ell \rceil$, so that
  overall constants in the statement are
  \begin{align*}
    K_\tau &= N\bigl([\delta] + [\delta^{-1}]\bigr)\\
    w_\tau &= 2N+1 = 3 + 4\lceil\kappa(\epsilon)/\ell\rceil.
  \end{align*}

  In order to be explicit about how to apply
  Lemmas~\ref{lem:brokenpathshort} and~\ref{lem:brokenpathslinked}, we
  will work
  with concrete lifts of our curves to broken paths in the universal
  cover; to pick out a lift, we work with elements of~$\pi_1$, and so
  pick a basepoint. For concreteness, choose the basepoint $*$ to be
  at the far left end of the segment on $\delta$ defining~$\tau'$.
  (Recall that we removed a short interval to make $\tau'$
  simply-connected.).  We are particularly
  interested in the short segments on the lift of~$\delta$; for that
  purpose, parameterize the lift of~$\delta$ by length in~$\R$,
  with $0$ at the lift of the
  basepoint~$*$ and $\delta$ oriented in the positive direction so
  that $\delta$ itself lifts to a curve ending at $\ell$.
  Given $\gamma \in \pi_1(\Sigma, *)$, we denote by $\ora{\gamma}$ its hyperbolic axis.
  
  Now we state precisely the sequence of smoothings that we will
  perform, illustrating them with slightly schematic figures of both
  the curves on the
  surface and of the corresponding broken paths realizing the lifts in
  the universal cover. At each step we circle the crossings that we
  smooth at the next step.
  \begin{enumerate}[(i)]
  \item Let $\delta, \alpha$ and $\beta$ be elements of $\pi_1(S)$ representing the transversal curve, $m^k(\vec{x})$ and $m^\ell(p^k(\vec{x}))$, respectively:
  \[\mathcenter{\fontsize{12pt}{12pt}\selectfont%% Creator: Inkscape inkscape 0.92.4, www.inkscape.org
%% Accompanies image file 'sm1.pdf' (pdf, eps, ps)
\begingroup%
  \makeatletter%
  \providecommand\color[2][]{%
    \errmessage{(Inkscape) Color is used for the text in Inkscape, but the package 'color.sty' is not loaded}%
    \renewcommand\color[2][]{}%
  }%
  \providecommand\transparent[1]{%
    \errmessage{(Inkscape) Transparency is used (non-zero) for the text in Inkscape, but the package 'transparent.sty' is not loaded}%
    \renewcommand\transparent[1]{}%
  }%
  \providecommand\rotatebox[2]{#2}%
  \newcommand*\fsize{\dimexpr\f@size pt\relax}%
  \newcommand*\lineheight[1]{\fontsize{\fsize}{#1\fsize}\selectfont}%
  \ifx\svgwidth\undefined%
    \setlength{\unitlength}{188.43586126bp}%
    \ifx\svgscale\undefined%
      \relax%
    \else%
      \setlength{\unitlength}{\unitlength * \real{\svgscale}}%
    \fi%
  \else%
    \setlength{\unitlength}{\svgwidth}%
  \fi%
  \global\let\svgwidth\undefined%
  \global\let\svgscale\undefined%
  \makeatother%
  \begin{picture}(1,0.78689092)%
    \lineheight{1}%
    \setlength\tabcolsep{0pt}%
    \put(0,0){\includegraphics[width=\unitlength,page=1]{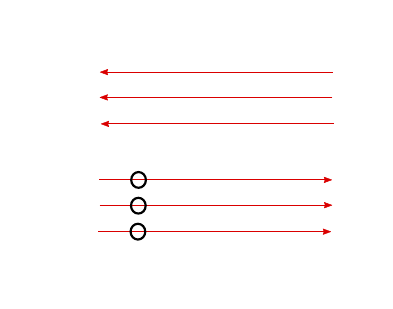}}%
    \put(0.45266109,0.76078764){\color[rgb]{0,0.50588235,0}\makebox(0,0)[t]{\lineheight{1.25}\smash{\begin{tabular}[t]{c}$\alpha$\end{tabular}}}}%
    \put(0.70557713,0.75975979){\color[rgb]{0,0,0.88235294}\makebox(0,0)[t]{\lineheight{1.25}\smash{\begin{tabular}[t]{c}$\beta$\end{tabular}}}}%
    \put(0.86251923,0.24658387){\color[rgb]{0.85490196,0,0}\transparent{0.90980393}\makebox(0,0)[lt]{\lineheight{1.25}\smash{\begin{tabular}[t]{l}$\delta$\end{tabular}}}}%
    \put(0.24276169,0.52093746){\color[rgb]{0.85490196,0,0}\makebox(0,0)[rt]{\lineheight{1.25}\smash{\begin{tabular}[t]{r}$\delta^{-1}$\end{tabular}}}}%
    \put(0,0){\includegraphics[width=\unitlength,page=2]{sm1.pdf}}%
  \end{picture}%
\endgroup%
}
  \qquad
  \mathcenter{\input{hyperbolic-draws/join-01.pgf}}\]
  The endpoints of the central short segments of lifts of~$\alpha$ and~$\beta$ are both
  in $(0,\ell)$, in the parameterization above.
\item Smooth $[\alpha]$ with $N[\delta]$ a total of $N$ times to get
    $[\alpha_1]$ with $\alpha_1 = \delta^M\alpha\delta^M$. The
    crossings are essential by
    Lemma~\ref{lem:brokenpathshort}.
    \[
      \mathcenter{\fontsize{12pt}{12pt}\selectfont%% Creator: Inkscape inkscape 0.92.4, www.inkscape.org
%% Accompanies image file 'sm2.pdf' (pdf, eps, ps)
\begingroup%
  \makeatletter%
  \providecommand\color[2][]{%
    \errmessage{(Inkscape) Color is used for the text in Inkscape, but the package 'color.sty' is not loaded}%
    \renewcommand\color[2][]{}%
  }%
  \providecommand\transparent[1]{%
    \errmessage{(Inkscape) Transparency is used (non-zero) for the text in Inkscape, but the package 'transparent.sty' is not loaded}%
    \renewcommand\transparent[1]{}%
  }%
  \providecommand\rotatebox[2]{#2}%
  \newcommand*\fsize{\dimexpr\f@size pt\relax}%
  \newcommand*\lineheight[1]{\fontsize{\fsize}{#1\fsize}\selectfont}%
  \ifx\svgwidth\undefined%
    \setlength{\unitlength}{119.70628389bp}%
    \ifx\svgscale\undefined%
      \relax%
    \else%
      \setlength{\unitlength}{\unitlength * \real{\svgscale}}%
    \fi%
  \else%
    \setlength{\unitlength}{\svgwidth}%
  \fi%
  \global\let\svgwidth\undefined%
  \global\let\svgscale\undefined%
  \makeatother%
  \begin{picture}(1,1.25306674)%
    \lineheight{1}%
    \setlength\tabcolsep{0pt}%
    \put(0,0){\includegraphics[width=\unitlength,page=1]{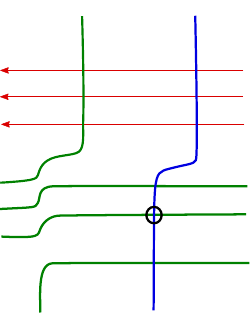}}%
    \put(0.32589353,1.20818034){\color[rgb]{0,0.50588235,0}\makebox(0,0)[t]{\lineheight{1.25}\smash{\begin{tabular}[t]{c}$\alpha\delta^4$\end{tabular}}}}%
    \put(0.77891396,1.21028428){\color[rgb]{0,0,0.88235294}\makebox(0,0)[t]{\lineheight{1.25}\smash{\begin{tabular}[t]{c}$\beta$\end{tabular}}}}%
    \put(0,0){\includegraphics[width=\unitlength,page=2]{sm2.pdf}}%
  \end{picture}%
\endgroup%
}
      \qquad
      \mathcenter{\input{hyperbolic-draws/join-02.pgf}}
    \]
    The endpoints of the central short segment of the lift of~$\alpha_1$ are
    in $(-M\ell,-(M-1)\ell)$ and in $(M\ell,(M+1)\ell)$.
  \item Smooth $[\alpha \delta^N] \cup [\beta]$ at a middle crossing
    to make $[\gamma_1] = [\alpha \delta^{M} \beta \delta^{M}]$.
    Since $(M-1)\ell \ge \kappa(\epsilon)$, the crossing is
    essential by Lemma~\ref{lem:brokenpathslinked}.
  \[\mathcenter{\fontsize{12pt}{12pt}\selectfont%% Creator: Inkscape inkscape 0.92.4, www.inkscape.org
%% Accompanies image file 'sm3.pdf' (pdf, eps, ps)
\begingroup%
  \makeatletter%
  \providecommand\color[2][]{%
    \errmessage{(Inkscape) Color is used for the text in Inkscape, but the package 'color.sty' is not loaded}%
    \renewcommand\color[2][]{}%
  }%
  \providecommand\transparent[1]{%
    \errmessage{(Inkscape) Transparency is used (non-zero) for the text in Inkscape, but the package 'transparent.sty' is not loaded}%
    \renewcommand\transparent[1]{}%
  }%
  \providecommand\rotatebox[2]{#2}%
  \newcommand*\fsize{\dimexpr\f@size pt\relax}%
  \newcommand*\lineheight[1]{\fontsize{\fsize}{#1\fsize}\selectfont}%
  \ifx\svgwidth\undefined%
    \setlength{\unitlength}{126.606022bp}%
    \ifx\svgscale\undefined%
      \relax%
    \else%
      \setlength{\unitlength}{\unitlength * \real{\svgscale}}%
    \fi%
  \else%
    \setlength{\unitlength}{\svgwidth}%
  \fi%
  \global\let\svgwidth\undefined%
  \global\let\svgscale\undefined%
  \makeatother%
  \begin{picture}(1,1.18261632)%
    \lineheight{1}%
    \setlength\tabcolsep{0pt}%
    \put(0,0){\includegraphics[width=\unitlength,page=1]{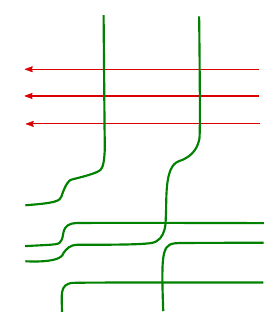}}%
    \put(0.3935756,1.14160613){\color[rgb]{0,0.50588235,0}\transparent{0.85882354}\makebox(0,0)[t]{\lineheight{1.25}\smash{\begin{tabular}[t]{c}$\alpha\delta^2\beta\delta^2$\end{tabular}}}}%
    \put(0,0){\includegraphics[width=\unitlength,page=2]{sm3.pdf}}%
  \end{picture}%
\endgroup%
}
    \qquad
    \mathcenter{\input{hyperbolic-draws/join-03.pgf}}\]
  Here, in the picture in
  the universal cover, two different lifts of $\gamma_1$ are shown
  (one dashed), to make it clearer what happened in the smoothing;
  these are the lifts of $\delta^{M} \alpha \delta^{M}
  \beta$ (solid) and  $\beta \delta^{M} \alpha
  \delta^{M}$ (dashed).
\item Smooth $[\gamma_1]$ with
  $N[\delta^{-1}]$ a total of
  $N$ times at appropriate crossings to make
  $[\alpha\beta] = m^{k+\ell}(\vec x)$. The crossings are essential by
  Lemma~\ref{lem:brokenpathshort}.
    \[\mathcenter{\fontsize{12pt}{12pt}\selectfont%% Creator: Inkscape inkscape 0.92.4, www.inkscape.org
%% Accompanies image file 'sm4.pdf' (pdf, eps, ps)
\begingroup%
  \makeatletter%
  \providecommand\color[2][]{%
    \errmessage{(Inkscape) Color is used for the text in Inkscape, but the package 'color.sty' is not loaded}%
    \renewcommand\color[2][]{}%
  }%
  \providecommand\transparent[1]{%
    \errmessage{(Inkscape) Transparency is used (non-zero) for the text in Inkscape, but the package 'transparent.sty' is not loaded}%
    \renewcommand\transparent[1]{}%
  }%
  \providecommand\rotatebox[2]{#2}%
  \newcommand*\fsize{\dimexpr\f@size pt\relax}%
  \newcommand*\lineheight[1]{\fontsize{\fsize}{#1\fsize}\selectfont}%
  \ifx\svgwidth\undefined%
    \setlength{\unitlength}{120.79714479bp}%
    \ifx\svgscale\undefined%
      \relax%
    \else%
      \setlength{\unitlength}{\unitlength * \real{\svgscale}}%
    \fi%
  \else%
    \setlength{\unitlength}{\svgwidth}%
  \fi%
  \global\let\svgwidth\undefined%
  \global\let\svgscale\undefined%
  \makeatother%
  \begin{picture}(1,1.21495285)%
    \lineheight{1}%
    \setlength\tabcolsep{0pt}%
    \put(0,0){\includegraphics[width=\unitlength,page=1]{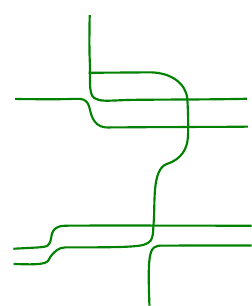}}%
    \put(0.35247498,1.17201646){\color[rgb]{0,0.50588235,0}\makebox(0,0)[t]{\lineheight{1.25}\smash{\begin{tabular}[t]{c}$\alpha\beta$\end{tabular}}}}%
    \put(0,0){\includegraphics[width=\unitlength,page=2]{sm4.pdf}}%
  \end{picture}%
\endgroup%
}\qquad\mathcenter{\input{hyperbolic-draws/join-04.pgf}}\]
  \item The result is $[\alpha \beta]$ as desired.
    \[
      \fontsize{12pt}{12pt}\selectfont%% Creator: Inkscape inkscape 0.92.4, www.inkscape.org
%% Accompanies image file 'sm5.pdf' (pdf, eps, ps)
\begingroup%
  \makeatletter%
  \providecommand\color[2][]{%
    \errmessage{(Inkscape) Color is used for the text in Inkscape, but the package 'color.sty' is not loaded}%
    \renewcommand\color[2][]{}%
  }%
  \providecommand\transparent[1]{%
    \errmessage{(Inkscape) Transparency is used (non-zero) for the text in Inkscape, but the package 'transparent.sty' is not loaded}%
    \renewcommand\transparent[1]{}%
  }%
  \providecommand\rotatebox[2]{#2}%
  \newcommand*\fsize{\dimexpr\f@size pt\relax}%
  \newcommand*\lineheight[1]{\fontsize{\fsize}{#1\fsize}\selectfont}%
  \ifx\svgwidth\undefined%
    \setlength{\unitlength}{68.15916584bp}%
    \ifx\svgscale\undefined%
      \relax%
    \else%
      \setlength{\unitlength}{\unitlength * \real{\svgscale}}%
    \fi%
  \else%
    \setlength{\unitlength}{\svgwidth}%
  \fi%
  \global\let\svgwidth\undefined%
  \global\let\svgscale\undefined%
  \makeatother%
  \begin{picture}(1,2.18198822)%
    \lineheight{1}%
    \setlength\tabcolsep{0pt}%
    \put(0,0){\includegraphics[width=\unitlength,page=1]{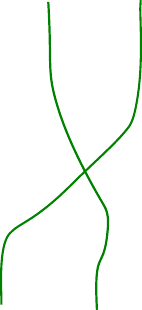}}%
  \end{picture}%
\endgroup%

    \]
  \end{enumerate}

  This completes the proof of
  part~\ref{item:join-join} of the statement. 
  Part~\ref{item:join-split} is very similar. Precisely, we do the
  following steps.
  \begin{enumerate}[(i$^\prime$)]
  \item Let $\delta = \tau$, $\alpha = m^k(\vec x)$, and $\beta =
    m^\ell(p^k(\vec x))$ be as before, so that we start with
    $[\alpha\beta] = [m^{k+\ell}(\vec x)]$.
  \item Use Lemma~\ref{lem:brokenpathshort} to smooth with $N[\delta]$
    a total of $N$ times to get $[\gamma_1]$ with
    $\gamma_1 = \delta^M\alpha\beta\delta^M$. If we set
    $\gamma_2 = \beta\delta^N\alpha$, then $[\gamma_1] = [\gamma_2]$,
    but these two group elements give different canonical lifts to the
    universal cover $\ora{\gamma_1}$ and $\ora{\gamma_2}$: the endpoints of the primary ``short segment''
    of $\ora{\gamma_1}$ are in $(-M\ell,-(M-1)\ell)$ and $(M\ell,(M+1)\ell)$ on
    the lift of~$\delta$, while the endpoints of the (zero-length) primary
    ``short segment'' of $\ora{\gamma_2}$ are at the same point in~$(0,\ell)$.  Therefore, $\ora{\gamma_1}$ and $\ora{\gamma_2}$ cross by
    Lemma~\ref{lem:brokenpathslinked}.
  \item Let $\alpha_1 = \delta^M\alpha$ and $\beta_1 =
    \beta\delta^M$ so $\gamma_1 = \alpha_1\beta_1$. We will now verify that the lifts $\ora{\alpha_1}$ and $\ora{\beta_1}$ are
parallel in the terminology of Proposition~\ref{prop:crossings-triality}. 
By Lemma~\ref{lem:brokenpathshort},
$\ora{\alpha_1}$ and $\ora{\beta_1}$ both cross $\ora{\delta}$ in the
same sense, so they cannot be anti-parallel. Recall that
$\ora{\gamma_1}$ and $\ora{\gamma_2}$ cross. Hence, by Lemma~\ref{lem:parallel_implies_crossing}, $\ora{\alpha_1}$ and $\ora{\beta_1}$ cannot cross and so must be parallel. By Proposition~\ref{prop:crossings-triality},
  smoothing $[\gamma_1]$ yields $[\alpha_1] \cup
    [\beta_1]$.

  \item Smooth $[\alpha_1]$ and $[\beta_1]$ each $M$ times with
    $M[\delta^{-1}]$ to make $[\alpha]$ and $[\beta]$, respectively, using
    Lemma~\ref{lem:brokenpathshort}.
  \item The result is $[\alpha] \cup [\beta]$ as desired. \qedhere
  \end{enumerate}
\end{proof}

\begin{proof}[Proof of Lemma~\ref{lem:join}]
  By definition of $M^k$ and $P^k$, we have non-negative
  constants $K$, $a_i$, $L$, and~$b_{i,j}$ so that
  \begin{align*}
    M^k(\vec{x}) &= \sum_{i=k}^K a_i m^i(\vec{x}) &
      P^k(\vec{x}) &= \sum_{i=k}^K a_i \vec{y_i}\\
    M^\ell(\vec{y_i}) &= \sum_{j=\ell}^L b_{i,j} m^j(\vec{y_i}) &
    M^{k+\ell}(\vec{x}) &= \sum_{i=k}^K \sum_{j=l}^L a_i b_{i,j} m^i(\vec{x}) m^j(\vec{y_i}).
  \end{align*}
  Furthermore, $\sum_i a_i = 1$ and, for fixed~$i$, $\sum_j b_{i,j} = 1$.
  The result follows by distributing and applying
  Lemma~\ref{lem:join-classical} repeatedly.
\end{proof}

As an immediate consequence, we have the following.
\begin{proposition}\label{prop:subadditivity} For a fixed good cross-section~$\tau$ as constructed above and for every curve
  functional~$f$ satisfying quasi-smoothing and convex union,
  there is a constant $\kappa(\tau)$ so that, for sufficiently large
  $k,\ell$ and every geodesic current~$\mu$, we have
  \[
    f^{k+\ell}_{\tau}(\mu)
      \le f^k_{\tau}(\mu) + f^\ell_{\tau}(\mu) + \kappa(\tau)\mu(\psi\tau).
  \]
\end{proposition}

\begin{proof}
  We will prove this with $\kappa(\tau) = f(K_\tau) + R w_\tau$, where
  $K_\tau$ and $w_\tau$ are from Lemma~\ref{lem:join} and $R$ is the
  quasi-smoothing constant from Equation~\eqref{eq:quasismoothing}.

  We have
  \begin{align*}
    f^{k+\ell}_\tau(\mu)
      &= f\left(\int_\tau M^{k+\ell}(\vec{x}) \psi(\vec{x})\mu(\vec{x})\right)\\
      &\le f\left(\int_\tau \bigl(M^k(\vec{x}) + M^\ell(P^k(\vec{x})) + K_\tau\bigr) \psi(x)\mu(x)\right)
        + \int_\tau R w_\tau \psi(\vec{x}) \mu(\vec{x})\\
      &\le f\left(\int_\tau M^k(x)\psi(\vec{x})\mu(\vec{x})\right)
        + f\left(\int_\tau M^\ell(P^k(\vec{x}))\psi(\vec{x})\mu(\vec{x})\right)
        + \kappa(\tau)\int_\tau \psi(\vec{x})\mu(\vec{x})\\
      &= f^k(\mu) + f\left(\int_\tau M^\ell(\vec{x})P^k_*(\psi(\vec{x})\mu(\vec{x}))\right) +
          \kappa(\tau)\mu(\psi\tau)\\
      &= f^k(\mu) + f^\ell(\mu) + \kappa(\tau)\mu(\psi\tau),
  \end{align*}
  where we use, successively:
  \begin{itemize}
  \item the definition of $f^{k+\ell}$;
  \item  Lemma~\ref{lem:join} and the quasi-smoothing property of~$f$;
  \item the convex union property of~$f$ and the definition of $\kappa(\tau)$;
  \item change of variables and the definitions of $f^k$ and $\mu(\psi\tau)$; and
  \item Proposition~\ref{prop:smeared-invariance} and the
    definition of~$f^\ell$.\qedhere
  \end{itemize}
\end{proof}

We recall a slight variation of Fekete's lemma, which follows from
standard versions \cite[Theorem~22]{BE52:Fekete}.

\begin{lemma}[Fekete's lemma]
  \label{lem:fekete}
  Let $(a_n)_{n=1}^\infty$ be a sequence of real numbers and suppose there exists $N$ such that for all $m,n \geq N$, 
  $a_{n+m} \le a_n + a_m$. Then
  \[
    \lim_{n \to \infty}\frac{a_n}{n} \quad\text{exists and is equal
      to}\quad
    \inf_{n\ge N} \frac{a_n}{n}.
  \]
\end{lemma}

Finally, we can show that the limit defining the extension of $f$ exists.

\begin{proposition}
For any curve functional~$f$ satisfying quasi-smoothing and convex union,
the limit defining~$f_\tau$ in
Equation~\eqref{eq:ftau} exists.
\label{prop:limitexists}
\end{proposition}
\begin{proof}
Use Proposition \ref{prop:subadditivity} and apply Lemma
\ref{lem:fekete} to the sequence
$f_\tau^k(\mu) + \kappa(\tau)\mu(\psi\tau)$. 
\end{proof}

%%% Local Variables:
%%% mode: latex
%%% TeX-master: "Smoothings"
%%% End:

\section{Continuity of the extension}
\label{sec:continuity}

In order to prove continuity of the extension, we will prove continuity of
$f^k_\tau$, and then get upper and lower
bounds on the limit $f_\tau(\mu)$ in terms of $f^k_\tau(\mu)$.
Proposition~\ref{prop:subadditivity} lets us use Fekete's lemma to get upper
bounds. To get lower bounds, we have the following.

\begin{proposition}\label{prop:superadditivity}
  For a fixed good cross-section~$\tau$ and any weighted curve functional~$f$ satisfying
  homogeneity, weighted quasi-smoothing, and convex union,
  there is a constant $K(\mu) = \kappa(\tau)\mu(\psi\tau)$ so that
  for all sufficiently large~$k$ and every geodesic
  current~$\mu$ we have
  \[
    2f^{k}_{\tau}(\mu) \leq f^{2k}_{\tau}(\mu) + K(\mu).
  \]
\end{proposition}

\begin{proof}
  By Lemma~\ref{lem:join}(b) in the case $k=\ell$,
  \[
    M^{2k}(\vec{x}) \cup K_\tau \reducesto_{w_\tau} M^k(\vec{x}) \cup M^k(P^k(\vec{x})).
  \]
  Integrating this statement with respect to the measure~$\psi\mu$
  (which is invariant under~$P^k$), we
  find that
  \[
    M^{2k}(\psi\mu) \cup K_\tau\cdot \mu(\psi\tau) \reducesto_{w_\tau \mu(\psi\tau)} 2M^k(\psi\mu).
  \]
  Applying $f$ to both sides and using homogeneity of~$f$ gives the
  desired result.
\end{proof}

Since $f$ is not in general additive, by comparison to
Proposition~\ref{prop:subadditivity},
Proposition~\ref{prop:superadditivity} is more restrictive,
requiring $k=\ell$. This still suffices to show that
the $f^k_\tau$ approximate $f_\tau$ well.

\begin{lemma} Let $\tau$ be fixed good cross-section and $f$ be a weighted curve functional satisfying
  homogeneity, weighted quasi-smoothing, and convex union. For any sufficiently large~$k$,
  \[
    \abs*{f_{\tau}(\mu)-\frac{f^{k}_{\tau}(\mu)}{k}} \leq \frac{K(\mu)}{k}
  \]
  where $K(\mu) = \kappa(\tau)\mu(\psi\tau)$ is the constant from
  Propositions~\ref{prop:subadditivity} and~\ref{prop:superadditivity}.
\label{lem:uniformapprox}
\end{lemma}
\begin{proof}
From Propositions~\ref{prop:subadditivity}
and~\ref{prop:superadditivity}, for large enough~$k$ we have
\[
  \Bigl| \frac{f^{2k}_{\tau}(\mu)}{2k}-\frac{f^{k}_{\tau}(\mu)}{k}\Bigr| \leq \frac{K(\mu)}{2k}.
\]
We also have
\begin{multline*}
  \frac{f^{k}_{\tau}(\mu)}{k} +  \Big(\frac{f^{2k}_{\tau}(\mu)}{2k} - \frac{f^{k}_{\tau}(\mu)}{k} \Big) + \Big(\frac{f^{4k}_{\tau}(\mu)}{4k} - \frac{f^{2k}_{\tau}(\mu)}{2k} \Big) + \Big(\frac{f^{8k}_{\tau}(\mu)}{8k} - \frac{f^{4k}_{\tau}(\mu)}{4k} \Big) + \cdots \\
  = \lim_{n \to \infty} \frac{f^{2^nk}_{\tau}(\mu)}{2^nk}
  = \lim_{n \to \infty} \frac{f^{nk}_{\tau}(\mu)}{nk}
  =f_{\tau}(\mu),
\end{multline*}
where the first equality follows by telescoping, and the second one because we have already proved that the limit exists.
We can then give bounds:
\begin{align*}
  \abs*{f_{\tau}(\mu)-\frac{f^{k}_{\tau}(\mu)}{k}}
   &\leq \abs*{\frac{f^{2k}_{\tau}(\mu)}{2k} -\frac{f^{k}_{\tau}(\mu)}{k}}
     + \abs*{\frac{f^{4k}_{\tau}(\mu)}{4k} -\frac{f^{2k}_{\tau}(\mu)}{2k}}
     + \abs*{\frac{f^{8k}_{\tau}(\mu)}{8k} -\frac{f^{4k}_{\tau}(\mu)}{4k}}
     + \cdots\\
  &\leq \frac{K(\mu)}{2k} + \frac{K(\mu)}{4k} + \frac{K(\mu)}{8k} + \cdots \\
   &= \frac{K(\mu)}{k}. \qedhere
\end{align*}
\end{proof}

We next prove that the $f^k_\tau$ are continuous.
\begin{proposition} Let $\tau$ be fixed good cross-section and $f$ be
  a weighted curve functional satisfying
  homogeneity, weighted quasi-smoothing, and convex union. Then the
  functions $f^{k}_{\tau}\colon \GC^+(S) \to \R$ are continuous for
  every~$k$.
\label{prop:iteratescont}
\end{proposition}

We will break the proof into lemmas.

\begin{lemma}\label{lem:flux-continuous}
  For $\tau$ a (closed) global cross-section with interior
  $\tau^\circ$, the map $\mu \mapsto \mu_{\tau^\circ}$ from $\GC^+(S)$
  to $\Meas(\tau^\circ)$ is continuous.
\end{lemma}

\begin{proof}
  We first adjust the definition of the flux~$\mu_{\tau^\circ}$. Let
  $\epsilon$ be small enough so that the corresponding flow box is
  embedded. Pick a non-zero continuous function
  $\omega \colon [0,\epsilon] \to \R_{\ge 0}$ so that
  $\omega(0) = \omega(\epsilon) = 0$ and
  $\int_0^\epsilon \omega(t)\,dt = 1$. Then, for any
  measurable function~$r$ on~$\tau^\circ$, the flux~$\mu_{\tau^\circ}$
  satisfies
  \begin{equation}\label{eq:smeared-flux}
    \int_{\vec{x} \in \tau^\circ} r(\vec{x})\,\mu_{\tau^\circ}(\vec{x})
      = \int_{t=0}^\epsilon \int_{\vec{x} \in \tau^\circ}
          \omega(t) r(\vec{x})\,\mu\bigl(\phi_t(\vec{x})\bigr).
  \end{equation}
  Now suppose that we have a sequence of measures $\mu_i$ approaching
  $\mu$ in the weak$^*$ topology, and let $r$ be a continuous function
  on~$\tau^\circ$ with compact support. By
  Theorem~\ref{thm:spacemeasuresmetrizable}, it suffices to show that
  $\int_{\vec{x} \in\tau^\circ} r(\vec{x}) \, \mu_{i,\tau^\circ}(\vec{x})$ converges to
  $\int_{\vec{x}\in\tau^\circ} r(\vec{x}) \,\mu_{\tau^\circ}(\vec{x})$.

  Consider the function~$s$ on~$UT\Sigma$ defined by
  \[
    s(\vec{y}) =
    \begin{cases}
      r(\vec{x})\omega(t) & \text{if $\vec{y} = \phi_t(\vec{x})$ for $\vec{x}\in\tau$, $t \in [0,\epsilon]$}\\
      0 & \text{otherwise}.
    \end{cases}
  \]
  Then $s$ is continuous, since $r$ and $\omega$ vanish on the
  boundaries of their domains of definition, so
  $\int_{\vec{y} \in UT\Sigma} s(\vec{y}) \mu_i(\vec{y})$ converges to
  $\int_{\vec{y} \in UT\Sigma} s(\vec{y}) \mu(\vec{y})$. The result follows from
  Equation~\eqref{eq:smeared-flux}.
\end{proof}

\begin{remark}\label{rem:flux-not-cont}
  The map $\mu \mapsto \mu_\tau$, from $\GC^+(S)$ to positive measures
  $\Meas(\tau)$ on the
  closed cross-section is \emph{not} continuous with
  respect to the weak$^*$ topology at points where
  $\mu_\tau(\partial \tau) \neq 0$. Indeed, let $\mu$ be the geodesic
  current corresponding to a closed curve $[a]$, let $[b]$ be another
  closed curve intersecting~$[a]$, and let $\mu_n$ be the geodesic
  current corresponding to $\frac{1}{n} [a^nb]$ so that
  $\lim_{n\to\infty} \mu_n = \mu$, as shown in
  Figure~\ref{fig:flux-not-cont}. Take a (non-complete)
  transversal~$\tau$ that intersects $\supp(\mu)$ only once on
  $\partial\tau$. Then
  (for appropriate choices, as shown) the total mass of $\mu_n$, i.e.
  $(\mu_n)_\tau(\tau)$, is approximately $1/2$, while
  $\mu_\tau(\tau) = 1$.

  Similarly, for the open transversals~$\tau^\circ$, we have $\mu_{\tau^\circ}(\tau^\circ) = 0$ while
  $(\mu_n)_{\tau^\circ}(\tau^\circ)$ is approximately $1/2$ for large~$n$. This does
  not contradict Lemma~\ref{lem:flux-continuous}; it just says that
  total mass is not a
  continuous function in the weak$^*$ topology on a non-compact space.
\end{remark}

\begin{figure}
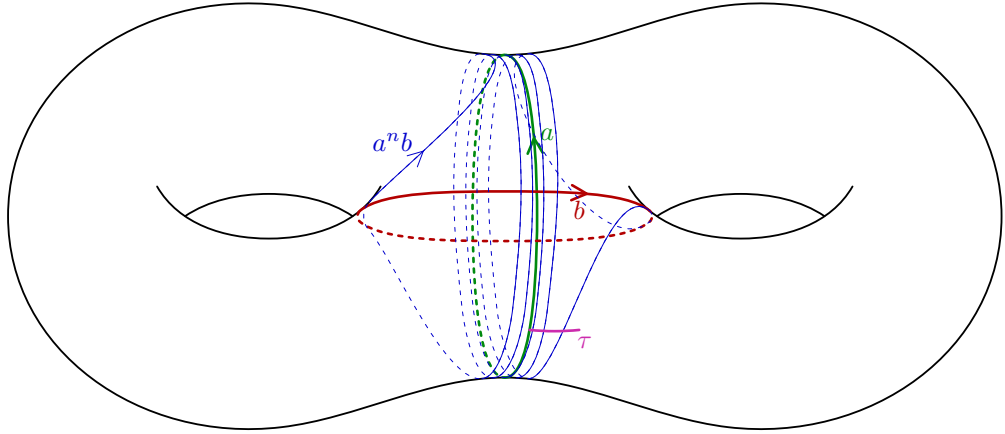

  \begin{gather*}
    \mfig{surface-0}
  \end{gather*}
  \caption{An example showing that the flux map $\mu \mapsto \mu_\tau$
    is not continuous. The sequence of curves $[a^nb]/n$
    approaches~$[a]$, but they have very different intersections
    with~$\tau$.}\label{fig:flux-not-cont}
\end{figure}

\begin{lemma}\label{lem:smeared-homotopy-meas-cont}
  The extension of $[M^k]$ to measures, as a
  map from $\Meas(\tau)$ to $\R\Curves(S)$, is continuous.
\end{lemma}
\begin{proof}
  Recall from Lemma~\ref{lem:smeared-iterate-finite} that
  $[M^k(\mu)]$ takes
  values in the finite-dimensional subspace
  $\R{[\Lambda(k,\tau,\tau_0)]} \subset \R\Curves(S)$. By continuity of
  $[M^k(\vec{x})]$ (Lemma~\ref{lem:smeared-homotopy-cont}), we can
  write
  \[
    [M^k(\vec{x})] = \sum_{C \in [\Lambda(k,\tau,\tau_0)]} a_C(\vec{x}) \cdot C
  \] 
  where $a_C$ is a continuous function on~$\tau$. (Recall that a
  function to a finite-dimensional vector space is continuous iff each
  of the coordinate functions is continuous; see Remark~\ref{rem:boundedfinitemeasures}.) But then
  \[
    [M^k(\mu)] = \sum_{C \in[\Lambda(k,\tau,\tau_0)]} \left(\int_{\vec{x} \in \tau} a_C(\vec{x})\,\mu(\vec{x})\right) \cdot C.
  \]
  The integrals are continuous functions of~$\mu$ by definition of the
  weak$^*$ topology on~$\Meas(\tau)$.
\end{proof}

\begin{proof}[Proof of Proposition \ref{prop:iteratescont} ]
  $f^k_\tau$ is the composition of maps
  \[
    \GC^+(S) \xrightarrow{\mu \mapsto \mu_{\tau^\circ}} \Meas(\tau^\circ)
      \xrightarrow{\cdot \psi} \Meas(\tau)
      \xrightarrow{[M^k]} \R^{\Lambda(k,\tau,\tau_0)}
      \overset{f}{\longrightarrow} \R.
  \]
  The component maps are continuous by, respectively,
  Lemma~\ref{lem:flux-continuous}; the fact that $\psi$ vanishes on a neighborhood of $\partial\tau$;
  Lemma~\ref{lem:smeared-homotopy-meas-cont}; and
  Proposition~\ref{prop:convexweights}.
\end{proof}

\begin{proposition}
Let $\tau$ be fixed good cross-section and $f$ be a curve functional
satisfying
  homogeneity, weighted quasi-smoothing, and convex union. Then $f_{\tau}\colon \GC^+(S) \to \mathbb{R}_{ \geq 0}$ is a continuous function.
\label{prop:continuity}
\end{proposition}

\begin{proof}
  By Proposition~\ref{prop:iteratescont}, it suffices to show that
  $f_\tau$ is a uniform limit of $f^k_\tau$. The constant $K(\mu)$ in
  Lemma~\ref{lem:uniformapprox} does depend on~$\mu$; however, if we
  bound $\mu$ within a ball so that $\int_\tau \psi\mu_\tau$ is
  bounded, the constant in the approximation becomes uniform and tends
  to~$0$ as $k \to \infty$.
\end{proof}

%%% Local Variables:
%%% mode: latex
%%% TeX-master: "Smoothings"
%%% End:

\section{The extension extends}
\label{sec:extends}

In this section we prove that when restricted to weighted curves,
the purported extension $f_{\tau}$ coincides with the original curve functional~$f$.
More precisely,
let $\gamma \in C$ be the geodesic representative of an oriented closed
curve with corresponding geodesic current~$\mu_C$, and let $\tau$ be a
good cross-section
of the geodesic flow (Definition~\ref{def:goodcross}). We wish to
show that $f_\tau(\mu_C) = f(C)$.

Let $\wt \gamma$ be the canonical lift of~$\gamma$ to the unit tangent bundle,
and let $n$ be the number of times that $\wt \gamma$
intersects~$\tau$, with intersections at
$\vec x_0, \vec x_1,\dots, \vec x_{n-1}$ in order (so $p(\vec x_i)=\vec x_{i+1}$). Then
$(\mu_C)_{\tau}=\sum_{i=0}^{n-1} \vec x_i$. Let $a_i=\psi(\vec x_i)$, so
that
\[
  C \cap \psi\tau \coloneqq \psi \cdot (\mu_C)_\tau = \sum_{i=0}^{n-1} a_i \vec x_i.
\]
By Proposition~\ref{prop:smeared-invariance}, this sum (which we
call $C \cap \psi \tau$, in an abuse of notation) is invariant under
the smeared return map~$P$:
\begin{equation}
    \label{eq:stationarydirac}
    P^k \biggl( \sum_{i=0}^{n-1} a_i \vec x_i \biggr)= \sum_{i=0}^{n-1} a_i \vec x_i
\end{equation}

We need a slightly stronger fact. Recall that each term in
$[M^k(C \cap \psi\tau)]$ is a curve that follows the geodesic trajectory
$\vec x_i \to \vec x_{i+1} \to \dots$ for some time and then travels along $\tau$ to close up. We say that a segment of the
return map $\vec x_i \to \vec x_{i+1}$ is \emph{covered} with degree~$r$ in $[M^k]$ if the
weighted number of times that segment appears in
$[M^k(C \cap \psi\tau)]$ is~$r$.
\begin{lemma}\label{lem:evenly-cover-1}
  For any closed curve~$C$ and a good cross-section~$\tau$ with bump
  function~$\psi$ as above, in
  $[M^k(C \cap \psi \tau)]$, every segment
  $\vec x_{i} \to \vec x_{i+1}$ is covered with degree~$k$.
\end{lemma}

(See Example~\ref{ex:smearedreturn} for one concrete case.)

\begin{proof}
  Fix $a_i = \psi(\vec x_i)$ as above, and consider the case $k=1$. By the
  assumption that $\tau_0$ is complete cross-section, we
have $a_i=1$ for some~$i$.
  By rotating the indices, assume $a_0=1$. We prove the statement for each segment $\vec x_{i} \to \vec x_{i+1}$
  by induction on $i$. For $i=0$, it's clear, since $a_0=1$ and no
  earlier trajectories continue through $\vec x_0$. For $i>0$, we have
  $a_i$ trajectories starting at $\vec x_i$ and going
  to $\vec x_{i+1}$. By the induction hypothesis, we also have weight 1 of
  trajectories arriving at
  $\vec x_i$ from $\vec x_{i-1}$ and so a weight of $1-a_i$ for those
  continuing on to $\vec x_{i+1}$. These two types of trajectories
  have a total weight of~$1$, as desired.
  
  The statement for $k>1$ follows from
  Equation~\eqref{eq:stationarydirac} and induction.
\end{proof}

Now, for $i < j$, let $C_{ij}$ be the curve that starts at
$\vec x_i$, passes through $j-i-1$ intermediate points
to $\vec x_j$, and closes up along $\tau$, with indices
interpreted modulo $n$. Then, for some coefficients $w_{ij}$, we can write
\[
  \Bigl[M^k\Bigl(\sum a_i \vec x_i\Bigr)\Bigr] = \sum w_{ij} C_{ij}.
\]
The non-zero coefficients $w_{ij}$ that appear will have
$k \le j - i \le kn$, so as $k$ gets large the $C_{ij}$ that appear in the
weighted sum also get long.

The invariance from Equation~\eqref{eq:stationarydirac} tells us that for all $i_0$,
\begin{equation}\label{eq:weights-balance}
  \sum_{i \equiv i_0} w_{ij} = \sum_{j \equiv i_0} w_{ij} = a_{i_0}
\end{equation}
while the fact that all $n$ steps $x_i \to x_{i+1}$ are covered with degree~$k$
implies that
\begin{equation}\label{eq:weights-tot}
  \sum_{i,j} w_{ij} = kn.
\end{equation}
\begin{proposition} 
\label{prop:extends}
 If $\mu_C$ is the geodesic current associated to a weighted closed
 multi-curve~$C$, and $\tau$ is a good cross-section of the geodesic flow,
 then $f_{\tau}(\mu_C)=f(C)$.
\end{proposition}
\begin{proof}
  We first suppose $C$ is a single curve with weight~$1$.
  
  As above, let $\sum_{i=0}^{n-1} a_i \vec x_i = C \cap \psi \tau$.  For the $k$-th
  iterate set
  \[
    C_0^k = \Bigl[M^k\Bigl(\sum a_i \vec x_i\Bigr)\Bigr]=\sum_{i<j} w_{ij}C_{ij}.
  \]
  Note that $C_{i,i+rn} = C^r$; other $C_{ij}$ have a more complicated
  relation to~$C$. For $k$ sufficiently large, we will
  use Lemma~\ref{lem:join} to simplify the sum so that only curves of
  the form $C_{i,i+rn}$ appear.
  
  For each $i=0,\dots,n-1$ (in any order), consider all the curves
  that either start or end at~$\vec x_i$, starting with
  $i = i_0$. By Equation~\eqref{eq:weights-balance},
  \begin{equation*}
    \sum_{\substack{i \equiv i_0\\j \not\equiv i_0}} w_{ij}
    = \sum_{\substack{i \not\equiv i_0\\j \equiv i_0}} w_{ij} \le a_{i_0}.
  \end{equation*}

  We can therefore pair the corresponding components of $C_0^k$ against each
  other using Lemma~\ref{lem:join} pairwise in any order,
  getting a reduction
  \[
    C_0^k \cup a_{i_0} K \reducesto_{a_{i_0}w} C_1^k
  \]
  where $K$ and~$w$ are the curve and weight from Lemma~\ref{lem:join}, and
  $C_1^k$ is another weighted combination
  of the $C_{ij}$ in
  which each component that
  starts at $i_0$ also ends at~$i_0$.

   This join operation doesn't change the degree by which segments of the curves are covered, so
   Equation~\eqref{eq:weights-tot} still holds, and
   Equation~\eqref{eq:weights-balance} still holds at the other
   indices. So we can repeat this at each index. In the end we get a
   reduction
   \[
     C_0^k \cup aK \reducesto_{aw} \sum_j b_j C^j
     \eqqcolon C_n^k
   \]
   where $a = \sum_i a_i$ and $C_n^k $ is another weighted curve. By considering the
   degrees we see that $\sum jb_j = k$.
   
   Similar considerations (similar to part (b) of
   Lemma~\ref{lem:join}) show that
   \[
     C_n^k \cup  aK \reducesto_{aw} C_0^k
   \]

   By Corollary~\ref{cor:weakstrongstab}, $f$ satisfies strong stability.
   The homogeneity and strong stability properties then yield
   \[
     kf(C)=f(kC) = f\Bigl(\sum_j  b_{j} j C\Bigr)
       = f\bigl(C_n^k\bigr)
   \]
  Therefore, since $aK$ and~$aw$ are independent of $k$,
   \[
     f_{\tau}(\mu) = \lim_{k \to \infty} \frac{f(C_0^k)}{k}
     = \lim_{k \to \infty} \frac{f(C_n^k)}{k}
     = \lim_{k \to \infty} \frac{f(kC)}{k} = f(C).
   \]
   We have thus proved that $f_{\tau}$ extends $f$ on unweighted
   curves.

   For the case of a general weighted curve $C = \sum w_\ell C_\ell$,
   the proof proceeds as above, except that we start with the weighted
   intersection of~$C$ with the smeared cross-section. More precisely,
   let $\vec x_{\ell,i}$ be the intersections of~$C_\ell$ with~$\tau$;
   then we work with
   $\sum_{\ell,i} w_\ell \psi(x_{\ell,i}) \cdot \vec x_{\ell,i}$,
   in the same way as above.
 \end{proof}

 \begin{comment}
\begin{figure}
\centering{
\includegraphics[width=100mm]{match-join1.pdf}
\includegraphics[width=100mm]{match-join2.pdf}

\caption{We consider the example of $\mu$ being the current associated to a closed curve $\gamma$ whose lift to the unit tangent bundle intersects the global cross-section $\tau$ twice at $\vec x_0$, $\vec x_1$ and intersects $\tau_0$ once at $\vec x_1$. By using the join lemma we show that $f(M^1(\mu)) \asymp f(C)$. }
\label{fig:match-and-join}
}
\end{figure}
\end{comment}
 
%%% Local Variables:
%%% mode: latex
%%% TeX-master: "Smoothings"
%%% End:

\section{Hyperbolic geometry estimates}
\label{sec:hyperbolic}

We complete the proof of Theorem~\ref{thm:convex} by proving facts
about the geometry of broken paths, as used in
Section~\ref{sec:join}.

\begin{definition}  \label{def:brokenpath}
  Fix a real length $L$ and angle~$\epsilon < \pi/2$.
  A \emph{broken path} $b(L,\epsilon)$ is a
  concatenation of geodesic segments in $\mathbb{H}^2$ that alternate
  between ``long'' segments of length at least~$\ell$ and ``short''
  segments of unconstrained length, so that the angle between the long
  and short segments is within
  $\epsilon$ of $\pi/2$, alternately turning left and right.
  See Figure~\ref{fig:broken-path-disk} for an example. We will
  denote by $a_i$ the hyperbolic line containing the $i$-th short
  segment.
\end{definition}

\begin{figure}
  \centering
  \input{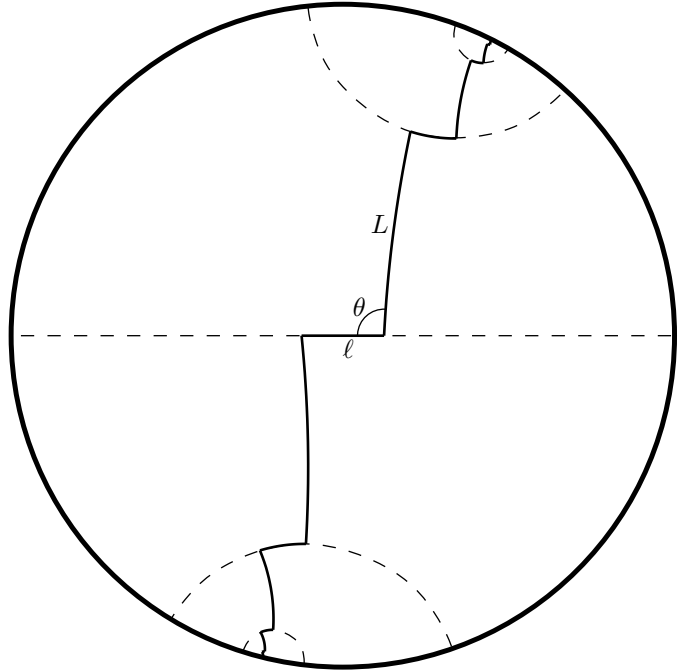}
  \caption{A broken path in the disk model. Here $\abs{\pi/2-\theta}
    < \epsilon$.}
  \label{fig:broken-path-disk}
\end{figure}

We prove some basic facts about when broken paths cross.
\begin{lemma}\label{lem:brokenpathshort}
  For any $0 < \epsilon < \pi/2$, there is a constant $L_0(\epsilon)$
  so that, for any $L > L_0(\epsilon)$, any broken path
  $b(L,\epsilon)$ converges to unique points at infinity that are on
  opposite sides of the hyperbolic line containing any short segment.
  As $\epsilon$ approaches $0$, the constant $L_0(\epsilon)$
  approaches $0$ as well.
\end{lemma}
That is, in Figure~\ref{fig:broken-path-disk}, the broken path crosses
the dashed paths.
\begin{proof}
  In fact, this is true so long as
  $L_0(\epsilon) > 2\gd^{-1}(\epsilon)$, where $\gd$ is the
  \emph{Gudermann function},
  defined, for instance, by $\gd(x) = \tan^{-1}(\sinh(x))$.
  
  Let $b$ be the broken path, and let $a_i$ be the hyperbolic line
  containing the $i$-th short segment. Since the turns in~$b$ alternate to
  the left and to the right, $b$ locally crosses each $a_i$.

  The bound on $L_0(\epsilon)$ was chosen so that $a_i$ and $a_{i+1}$
  do not cross or meet at infinity. (Another way to say this is that
  $\pi/2 - \epsilon$ is bigger than the angle of parallelism of
  $L_0(\epsilon)/2$.) Thus the path $b$ crosses the sequence of
  non-crossing segments $a_i$, and thus cannot cross a single $a_i$
  more than once, as desired.

  The fact that $L_0(\epsilon)$ is strictly greater than
  $2\gd^{-1}(\epsilon)$ means that as $i \to \pm\infty$ the endpoints
  of the segments~$a_i$ get closer by a definite factor on
  $\partial\H^2$. Thus, in either direction, $b$ converges to a
  definite point on the circle at infinity.
\end{proof}

From now on, we assume that all broken paths have $L>L_0(\epsilon)$.

\begin{lemma}\label{lem:brokenpathslinked}
  Fix $0 < \epsilon < \pi/2$ and $L > L_0(\epsilon)$. Then there is a constant
  $\kappa(\epsilon)$ with the following property. If
  $\gamma = b(L,\epsilon)$ and $\gamma' = b'(L,\epsilon)$ are two
  broken paths with a pair of short segments $s_0 \subset s_0'$ on the same line~$a_0$,
  and $s_0'$ extends at least
  $\kappa(\epsilon)$
  farther along~$a_0$ in each direction than~$s_0$, then $\gamma$
  and~$\gamma'$ cross essentially
  on~$a_0$.
\end{lemma}

Note that in the last claim there is no control on $\kappa(\epsilon)$.

\begin{proof}

\begin{figure}
  \centering
  \input{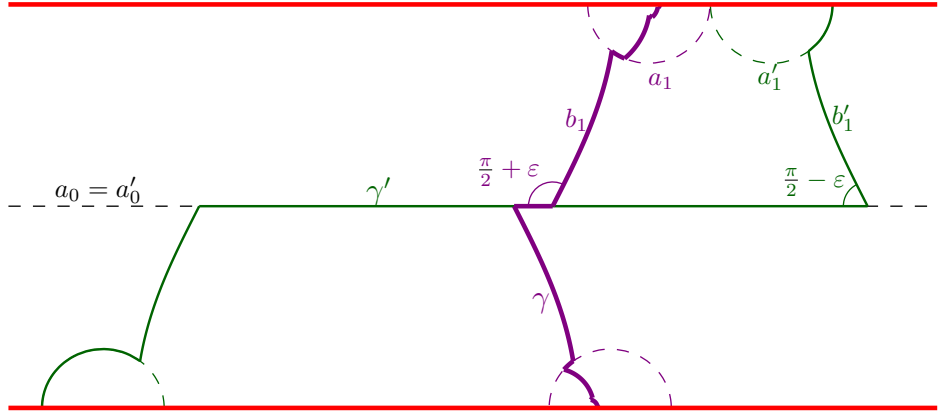}
  \caption{Crossing broken paths in the bands model for
    Lemma~\ref{lem:brokenpathslinked} and its proof, showing the case
    when the windows nearly touch.}
  \label{fig:broken-paths-cross}
\end{figure}
  It is most convenient to work in
  the band model of the hyperbolic plane as in
  Figure~\ref{fig:broken-paths-cross}. Focus first on the
  path~$\gamma$, and let $s_0$, $l_1$, and $s_1$ be the next short and
  long segments of~$\gamma$, and let $a_1$ be the line
  containing~$s_1$. The line $a_1$ defines an interval on
  $\partial\H^2$ that, by Lemma~\ref{lem:brokenpathshort}, must
  contain the endpoint of~$\gamma$. Now fix the endpoints
  of $s_0$ and vary the other parameters defining the interval
  of~$a_1$, namely
  \begin{itemize}
  \item the angles between $s_0$ and~$l_1$ and between $l_1$
    and~$s_1$, both in $[\pi/2-\epsilon,\pi/2+\epsilon]$, and
  \item the length of~$l_1$, in $[L_0(\epsilon),\infty]$.
  \end{itemize}
  (If we allow $\ell_1$ to have infinite length, the interval
  degenerates to a single point on $\partial\H^2$.) As the parameters
  vary, the interval varies continuously on~$\partial\H^2$, remaining
  disjoint from the endpoints of~$a_0$. By compactness of the domain,
  the union of these intervals is a larger interval $W \subset \partial\H^2$
  that necessarily contains the endpoint of~$\gamma$ for fixed
  endpoint of~$s_0$.
  Figure~\ref{fig:broken-paths-cross-a} shows the presumably extremal
  possibilities for~$W$ in one example, but we do not need to identify
  the precise
  values.
  \begin{figure}
    \centering
    \input{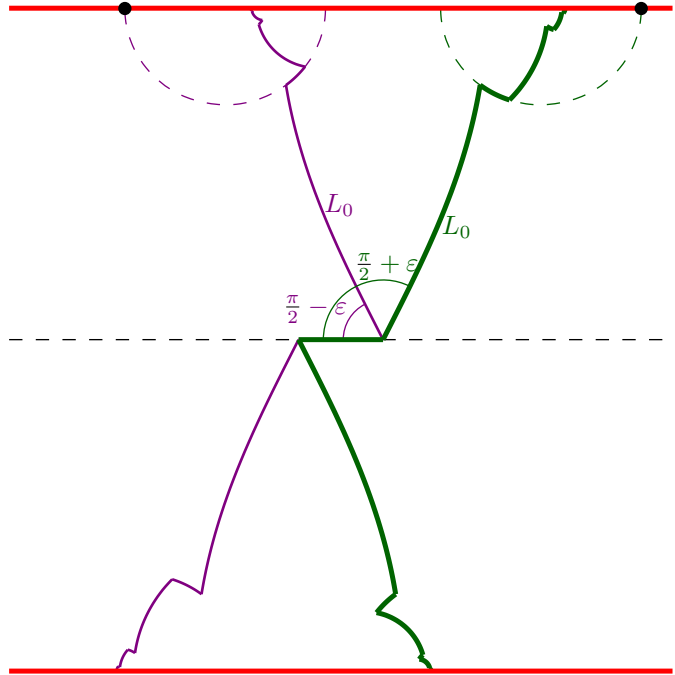}
    \caption{The window of possible endpoints of a broken path. The
      marked points are bounds on ends of broken paths $b(L,\epsilon)$ with
      $L \ge L_0$.}
    \label{fig:broken-paths-cross-a}
  \end{figure}

  A similar argument applies to endpoint of $\gamma'$ on the same side
  of $a_0$: it must lie in another window $W'$ on $\partial\H^2$. By
  symmetry, in the band model $W'$ is a translation of $W$ by a
  (Euclidean) amount proportional to $\kappa(\epsilon)$. Thus for
  $\kappa(\epsilon)$ sufficiently large, $W$ and $W'$
  will be disjoint; one extremal case is shown in
  Figure~\ref{fig:broken-paths-cross}.

  Similar arguments apply to the other endpoints of $\gamma$
  and~$\gamma'$, implying that for large enough $\kappa(\epsilon)$ the
  paths cross essentially.
\end{proof}
%%% Local Variables:
%%% mode: latex
%%% TeX-master: "Smoothings"
%%% End:

\section{Stable functions}
\label{sec:stable}

Some curve functionals satisfy quasi-smoothing and convex union but
are not stable or homogeneous on the nose. For example, the length of
a curve with respect to an arbitrary generating set is of this form
(Example~\ref{ex:puncturedtorus}). We fix this by passing to a stable
length as in Theorem~\ref{thm:stable}. Recall that the \emph{stable curve
  functional} $\norm{f}$ is
  defined by
  \[
    \norm{f}(C) \coloneqq \lim_{n \to \infty} \frac{f(C^n)}{n}.
  \]

As in the proof of Theorem~\ref{thm:weightconvex}, we will consider weighted
curve functionals.
In this section we will prove the following theorem

\begin{theorem}\label{thm:weightstable}
  Let $f$ be a weighted curve
  functional satisfying weighted quasi-smoothing and convex union.
  Then the stabilized curve functional 
  \[
    \|f\|(C) \coloneqq \lim_{n \to \infty} \frac{f(C^n)}{n}.
  \]
  satisfies weighted quasi-smoothing, convex union, strong stability, and homogeneity,
  and thus extends to a continuous function on $\GC^+(S)$.
\end{theorem}

We first prove some lemmas.

\begin{lemma}
  \label{lem:stable-def}
  For any connected curve~$C$ and any sufficiently large $n,m \geq 0$,
  there exists some
  curve $K$ and weight~$w$ so that $C^n \cup C^m \cup K \reducesto_w C^{n+m}$.
\end{lemma}

\begin{proof}
  Let $p$ be the number of intersections of our cross-section~$\tau$
  with the canonical lift of~$C$ to $UT\Sigma$, and apply
  Lemma~\ref{lem:join-classical}(b), taking $k=np$ and $l=mp$.
  (We reuse the same cross-section~$\tau$ for convenience;
    nothing here depends on the definition of the extension.)
  Then for sufficiently large $n,m$,
   \begin{align*}
    C^{n+m} \cup K_{\tau} &\reducesto_{w_\tau} C^m \cup C^n
   \end{align*}
   proving the lemma with $K = K_\tau$ and $w = w_\tau$.
\end{proof}
\begin{lemma}
  \label{lem:quasistable-a}
  For any connected curve~$C$, we have $C^{n} \reducesto_{n-1} n C$.
\end{lemma}

\begin{proof}
  This is Proposition~\ref{prop:crossings-powers}.
\end{proof}

\begin{lemma}
\label{lem:quasistable-b}
There is a curve $K$ and a constant~$w$ so
that, for any curve $C$ on~$S$ and any $n \ge 2$ we have
\[
nC \cup (n-1)K\reducesto_{(n-1)w} C^n.
\]
\end{lemma}

\begin{proof}
  Pick $0 < \epsilon < \pi/2$ so that $L_0(\epsilon)$ from
  Lemma~\ref{lem:brokenpathshort} is less than the
  \emph{systole} of~$\Sigma$, the length of the shortest closed
  geodesic on~$\Sigma$. As in Section~\ref{sec:cross-section}, find a
  curve~$K$ and a
  complete global cross-section $\tau \subset W(K, \epsilon)$. Then,
  by the arguments of Lemma~\ref{lem:join-classical}(a), there is some
  integer~$w$ so that
   \[
     2C \cup K \reducesto_{w} C^2.
   \]
   (We are not directly applying Lemma~\ref{lem:join-classical}, since
   we do not let the iteration in the return map go to infinity;
   but all of the long segments of the broken paths are long enough to
   make the arguments there
   work.) Iterating in this
   way, we deduce the desired result.
\end{proof}

\begin{proof}[Proof of Theorem~\ref{thm:weightstable}]
We must show that $\norm{f}$ is well-defined and satisfies convex union,
weighted quasi-smoothing, strong stability and homogeneity. Let $R\ge 0$ be the
quasi-smoothing constant of~$f$, and let $K$ and $w$ be
the curves and constants from Lemmas~\ref{lem:stable-def}
and~\ref{lem:quasistable-b} (which we can take to be the same), and
let $f(K)^+$ be $\max(f(K), 0)$.
\begin{itemize}
\item \textbf{Well-defined:} Lemma~\ref{lem:stable-def} shows that the sequence
  $(f(C^n) + wR + f(K)^+)_{n \in \mathbb{N}}$ is sub-additive for large enough~$n$, and
  thus by Lemma~\ref{lem:fekete} the limit defining $\norm{f}$ exists.
\item  \textbf{Convex union:} this follows immediately from the fact that $(C_1
  \cup C_2)^k = C_1^k \cup C_2^k$, the definition of $\norm{f}$, and
  convex union property of~$f$.
\item \textbf{Strong stability:}
  For any
  multi-curve $D$ and any curve~$C$,
  by Lemmas \ref{lem:quasistable-b}
  and~\ref{lem:quasistable-a} (applied to $C^k)$ and the
  quasi-smoothing and convex union properties of~$f$, we have, 
\begin{align*}
  f(D^k \cup C^{nk})-(n-1)wR &\leq f(D^k \cup nC^k) + (n-1)f(K)^+\\
  f(D^k \cup nC^k)-(n-1)R &\leq f(D^k \cup C^{nk}).
\end{align*}
Combining the inequalities, dividing
by~$k$, and letting $k$ go to infinity, we obtain
\begin{align*}
\norm{f}(D \cup C^n)&=\norm{f}(D \cup nC).
\end{align*}
We can iterate this to prove the result when $C$ is a multi-curve.
\item \textbf{Homogeneity:} It is clear from the definition of
  $\norm{f}$ that $\norm{f}(C^n) = n\norm{f}(C)$. Homogeneity then
  follows from stability.
\item \textbf{Weighted Quasi-smoothing:} Let $C = C_1 \cup C_2$ be a
  multi-curve, where the
  smoothing involves the component(s) in $C_1$, so that $C_1$ and its smoothing
  $C_1'$ each have at most two components. Thus
  $C_1^k \reducesto_{2(k-1)} k C_1$ and
      $kC_1 \reducesto_k kC_1'$.
      Then
      \begin{align*}
        \norm{f}(C)
          &=\lim_{k \to \infty} \frac{f(C_1^k \cup C_2^k)}{k} \\
          & \ge \lim_{k \to \infty} \frac{f(kC_1 \cup C_2^k)-2kR}{k}
            && \text{(Lemma~\ref{lem:quasistable-a})}\\
          & \ge \lim_{k \to \infty}  \frac{f(kC_1' \cup C_2^k) - 3kR}{k}
            && \text{(quasi-smoothing for $f$)}\\
          &\ge \lim_{k \to \infty}  \frac{f((C_1')^k \cup C_2^k) - 3kR
            - 2kwR - 2kf(K)^+}{k}
            &&\text{(Lemma~\ref{lem:quasistable-b})}\\              
        & = \norm{f}(C') - (3+2w)R - 2f(K)^+,
      \end{align*}
      so $\norm{f}$ satisfies quasi-smoothing, with
      constant $(3+2w)R + 2f(K)^+$. By
      Proposition~\ref{prop:weightextend}, $\norm{f}$ also satisfies
      weighted quasi-smoothing.
      \qedhere
    \end{itemize}
\end{proof}

Finally, we show that with other hypotheses, (weak) stability implies strong
stability, so that we don't need to assume strong stability in the
statement of Theorem~\ref{thm:weightconvex}.

\begin{corollary}
  Let $f$ be a weighted curve
  functional satisfying weighted quasi-smoothing, convexity, stability, and homogeneity. Then $f$ also satisfies strong stability.
  \label{cor:weakstrongstab}
\end{corollary}
\begin{proof}
By the definition of $\norm{f}$ and stability and homogeneity of $f$, we have, for all oriented multi-curves $C$,
\[
\norm{f}(C)=f(C).
\]
By Theorem~\ref{thm:stable}, $\norm{f}$ satisfies strong stability.
\end{proof}

The proof of this part of Theorem~\ref{thm:weightstable} does not use
Theorem~\ref{thm:weightconvex}, so we can use
Corollary~\ref{cor:weakstrongstab} in the proof of
Theorem~\ref{thm:weightconvex}.

We finish by proving Theorem~\ref{thm:stable}.

\begin{proof}[Proof of Theorem~\ref{thm:stable}]
By Proposition~\ref{prop:weightextend}, $f$ extends uniquely to a weighted curve functional satisfying convex union, homogeneity, stability, and weighted quasi-smoothing with the same constant.
Theorem~\ref{thm:weightstable} applied to this extension gives the result.
\end{proof}

%%% Local Variables:
%%% mode: latex
%%% TeX-master: "Smoothings"
%%% End:

% \nocite{*}
\bibliographystyle{hamsalpha}
\bibliography{Smoothings}

\end{document}